%% file: main_arxiv.tex
\documentclass{article}
\usepackage[T1]{fontenc}
\usepackage[utf8]{inputenc}
\usepackage{lmodern}
\usepackage{stmaryrd}
\usepackage{comment}

\usepackage{geometry} 
\usepackage{amsmath,amssymb,amsfonts,amsthm,mathtools}
\usepackage{mathrsfs}

\usepackage{standalone}

\usepackage{graphicx}
\usepackage{caption}
\usepackage{subcaption} 

\usepackage{booktabs}
\usepackage{array}
\usepackage{paralist} 
\usepackage{textcomp} 
\usepackage{xcolor} 
\usepackage{epigraph} 

\usepackage[
colorlinks=true,
urlcolor=purple,
linkcolor=purple!87!black,
citecolor=green!60!black,
pdfborder={0 0 0}
]{hyperref}
\usepackage{bookmark}

\usepackage{url}
\usepackage{float}
\usepackage[style=numeric,sorting=nyt,backref=true,maxbibnames=99]{biblatex}
\usepackage{tikz}
\usetikzlibrary{calc,arrows,decorations.pathmorphing,backgrounds,positioning,fit,petri,knots,patterns}

\usepackage{cleveref}

\theoremstyle{definition}
\newtheorem{theorem}{Theorem}[section] 
\newtheorem{prop}[theorem]{Proposition}
\newtheorem{lemma}[theorem]{Lemma}
\newtheorem{corollary}[theorem]{Corollary}

\newtheorem{remark}[theorem]{Remark}
\newtheorem{defn}[theorem]{Definition} 
\newtheorem{example}[theorem]{Example} 
\newtheorem{question}[theorem]{Question}

\newtheorem{claim}[theorem]{Claim} 

\theoremstyle{definition}
\newtheorem{thmx}{Theorem}

\newcommand{\N}{\mathbb{N}}
\newcommand{\NN}{\mathbb{N}}
\newcommand{\Z}{\mathbb{Z}}
\newcommand{\ZZ}{\mathbb{Z}}
\newcommand{\R}{\mathbb{R}}
\newcommand{\RR}{\mathbb{R}}
\newcommand{\M}{\mathcal{M}}

\newcommand{\PP}{\mathcal{P}}
\newcommand{\FF}{\mathcal{F}}
\newcommand{\Fo}{\mathcal{F}}
\newcommand{\init}{\mathfrak{i}}
\newcommand{\ter}{\mathfrak{t}}

\newcommand{\symb}[1]{\texttt{#1}}
\newcommand{\norm}[1]{\|{#1}\|}
\newcommand{\QI}{\mathtt{QI}}
\newcommand{\mov}{\operatorname{mov}}
\newcommand{\ind}{\operatorname{ind}}
\newcommand{\val}{\operatorname{val}}
\newcommand{\pos}{\operatorname{pos}}

\newcommand{\orb}{\textnormal{orb}}
\newcommand{\supp}{\textnormal{supp}}

\providecommand{\keywords}[1]{{\small\textit{Keywords:} #1}}

\definecolor{azul}{rgb}{0.29,0.56,0.89}
\definecolor{rojo}{rgb}{0.82,0.01,0.11}
\definecolor{naranjo}{rgb}{0.96,0.65,0.14}
\definecolor{verde}{RGB}{126,211,33}

\newcommand{\define}[1]{\textbf{#1}}

\tikzset
{%
  tile1/.style={thick,draw=black,fill=gray!10},
  tile2/.style={thick,draw=black,fill\input{main_ams}2w=white!10},
  pics/tile/.style={
    code={%
    \path[pic actions] (0,0)            --++ (330:{cos(30)})  --++ (60:0.5)  --++ (120:1)   -|++
                       (-0.5,{cos(30)}) --++ (150:{cos(30)})  --++ (240:0.5) --++ (300:0.5) --++
                       (210:{cos(30)})  |-++ (0.5,{-cos(30)}) --++ (300:0.5) -- cycle;
    \draw[fill = black] (-0.5,0.5) circle (0.05);
    }
    }
}

\newcommand{\Addresses}{{
		\bigskip
		
		\hskip-\parindent   S.~Barbieri, \textsc{Departamento de Matem\'{a}tica y ciencia de la computaci\'{o}n, Universidad de Santiago de Chile, Santiago, Chile.}\par\nopagebreak
		\textit{E-mail address}: \texttt{sebastian.barbieri@usach.cl}
		
		\medskip
		
		\hskip-\parindent   N.~Bitar, \textsc{
        Laboratoire Ami\'enois de Math\'{e}matique Fondamentale et Appliqu\'{e}e (UMR CNRS 7352), Universit\'{e} de Picardie Jules-Verne,  Amiens, France.}\par\nopagebreak
		\textit{E-mail address}: \texttt{nicolas.bitar@u-picardie.fr}

}}
\usepackage{morefloats}

\usepackage{scalerel}
\usepackage[ruled]{algorithm2e}

\begin{document}

\author{Sebasti\'an Barbieri and Nicol\'as Bitar}
\date{}
\title{A general framework for quasi-isometries in symbolic dynamics beyond groups}
\maketitle

\begin{abstract}
We introduce an algebraic structure which encodes a collection of countable graphs through a set of states, generators and relations. These structures, which we call blueprints, can capture standard
algebraic objects such as groups, monoids or small categories, as well as geometric tiling spaces with finite local complexity.

We provide a general framework for symbolic dynamics on blueprints under a partial monoid action, and for transferring invariants of their symbolic dynamics through quasi-isometries. In particular, we show that the undecidability of the domino problem, the existence of strongly aperiodic subshifts of finite type, and the existence of subshifts of finite type without computable points are all quasi-isometry invariants for finitely presented blueprints. As an application of this model, we show that two variants of the domino problem for geometric tilings of $\RR^d$ are undecidable for $d \geq 2$ on any underlying tiling space with finite local complexity.

\vskip .1in
\noindent\keywords{Quasi-isometries, symbolic dynamics, domino problem, aperiodicity, geometric tilings.}
\smallskip
		\noindent
		\emph{MSC2020:} \textit{Primary:}
        37B10, 
		\textit{Secondary:}
        52C23, 
        05C25. 

\end{abstract}

\section{Introduction}

A recent trend in the study of subshifts of finite type (SFT) on groups is to explore how the properties of the underlying group influence the computational and dynamical properties of the SFTs, and vice versa. This has been done through the study of computability invariants~\cite{berger1966undecidability,callard2022aperiodic,bitar2023contributions}, dynamical invariants such as aperiodicity~\cite{barbieri2023aperiodic_non_fg,bitar2024realizability}, the set of possible topological entropies~\cite{barbieri2021entropies,raymond2023shifts,bartholdi2024shifts}, Medvedev degrees~\cite{Simpson_medvedev_2014,barbieri2024medvedev}, among others. 

Two fundamental problems in this regard are the classification of groups with undecidable domino problem and the classification of groups that admit strongly aperiodic SFTs (those for which the shift action is free). The study of these two problems was launched by Berger~\cite{berger1966undecidability} who constructed the first strongly aperiodic SFT on $\Z^2$, and used it to prove the undecidability of the domino problem for this group. Since then, many groups have been shown to admit strongly aperiodic SFTs, and many have had the decidability of their domino problem classified. For the latter problem, the domino problem Conjecture (attributed to Ballier and Stein~\cite{ballier2013domino}) states that a finitely generated group has decidable domino problem if and only if the group is virtually free (see~\cite[Chapter 2]{bitar2024tesis} for a recent survey). For the former problem, it is conjectured that a finitely generated and recursively presented group admits a strongly aperiodic SFT if and only if it has one end and decidable word problem (see~\cite{bitar2024realizability} for a recent survey).

An important tool in the study of these two problems is to encode the structure of a group into another through symbolic spaces. Two pioneering examples of this technique are the techniques used by Ballier and Stein~\cite{ballier2013domino} and the work of Jeandel on translation-like actions~\cite{jeandel2015translation}. One of the most abstract ways to do these codings is through quasi-isometries. Indeed, Cohen showed in~\cite{cohen2017large} that the undecidability of the domino problem and the existence of strongly aperiodic SFTs are quasi-isometry invariants for finitely presented groups. This is achieved by using the space of quasi-isometries between the two groups to code the structure of one group on the other through local rules. The hypothesis of finite presentability is crucial to ensure the resulting subshift is an SFT. This same proof technique has been used to prove the invariance under quasi-isometries of self-simulable groups~\cite{barbieri2021groups} and the set of Medvedev degrees of SFTs~\cite{barbieri2024medvedev} (provided the quasi-isometry is computable) for finitely presented groups with decidable word problem.

Interestingly, there are a number of recent results about these two problems that implicitly employ quasi-isometries that involve structures that are not groups. This is the case of the proof of the undecidability of the domino problem for surface groups~\cite{aubrun2019domino} and more generally, for non-virtually free hyperbolic groups~\cite{bartholdi2023domino}. Similarly, combinatorial results about tilings are implicitly based on encoding a quasi-isometry with a group. For instance, the undecidability of the domino problem for rhombus tilings~\cite{hellouin2023domino} and for two-dimensional geometric tilings with finite local complexity~\cite{de2024decision}. There have been other attempts at understanding the underlying conditions that account for the undecidability of the domino problem on $\Z^2$ that go beyond groups. Among these are subjecting $\Z^2$-SFTs to horizontal constraints~\cite{aubrun2020domino,esnay2023parametrization}, studying automatic-simulations between labeled graphs~\cite{bartholdi2020simulations}, monadic second order logic on labeled graphs~\cite{bartholdi2022monadic}, and studying the domino problem on self-similar two-dimensional substitutions~\cite{barbieri2016domino}.\\

The objective of this article is to find a common framework for the aforementioned results through the introduction of structures we call \define{blueprints}. The goal of these structures is to capture spaces of graphs which are locally finite and ``finitely presented''. They are defined by a set of states, a set of generators, and a set of relations (see Definition~\ref{def:blueprint}). Each generator has an initial state and a set of terminal states, and two generators can be composed if the initial state of the second is contained on the set of terminal states of the first. Because there are multiple choices for the terminal state of a generator, we make use of functions we call \define{models} that map each word over the generators to either a state or the empty set in a way that is consistent with the composition of generators and the equivalence relation generated by the relations of the blueprint (see Definitions~\ref{def:consistent} and~\ref{def:model}). Finally, each model has an associated directed labeled graph where vertices are equivalence classes of words sent to the set of states by the model under the equivalence relation, and edges are given by the generators. Blueprints can be used to model small categories, in the sense that every Cayley graph of a small category can be realized as the space of graphs of models of a blueprint. 

The next step is studying subshifts on blueprints. The definition of a subshift in this context is similar to the one from the group setting, except for the fact that configurations are composed of a model of the blueprint and a coloring of the set of all words over the generators by a finite alphabet that is consistent with the model and the relations (see Definition~\ref{def:subshift}). With this formalism, we define a natural analogue of SFT, and recover classical results from the group setting such as the Curtis-Hedlund-Lyndon Theorem (Theorem~\ref{thm:CHL}).\\

We note that other ways of describing finite presentations for graphs are present in the literature (see~\cite{arrighi2023graph,meyer2005graphes}), which we do not explore in this article.  Furthermore, there have been similar attempts to capture symbolic dynamics on graphs. In particular, in~\cite{arrighi2023graph} produces very similar (undirected) objects, but our formalism is more natural for working with quasi-isometries.

\subsection*{Main results }
We generalize Cohen's result to finitely presented blueprints whose model graphs are strongly connected. We say two blueprints are quasi-isometric when all of their model graphs are quasi-isometric, seen as quasi-metric spaces. Theorem~\ref{thm:QI} provides a black box that embeds an SFT from a blueprint into an SFT in the other blueprint in a geometric way that preserves many of its dynamical properties.

Given a fixed blueprint, its domino problem is the formal language of all finite collections of forbidden patterns which give rise to nonempty subshifts. With the use of this black box, we prove the invariance of the undecidability of the domino problem.

\begin{thmx}[Theorem~\ref{thm:domino}]
    Let $\Gamma_1$, $\Gamma_2$ be two finitely presented strongly connected blueprints that are quasi-isometric. Then, the $\Gamma_1$-domino problem is decidable if and only if the $\Gamma_2$-domino problem is decidable.
\end{thmx}

We also prove the invariance of a pointed variant of the domino problem that requires the additional assumption of minimality on one of the blueprints (\Cref{thm:domino_model}).

Given a subshift on a blueprint, we say it is \define{strongly aperiodic} if the partial shift action is free. Again, using our black box, we show that the existence of strongly aperiodic SFTs is an invariant. 

\begin{thmx}[Theorem~\ref{thm:SA}]\label{mainthm:SA}
    Let $\Gamma_1$, $\Gamma_2$ be two finitely presented strongly connected blueprints that are quasi-isometric. $\Gamma_1$ admits a strongly aperiodic SFT if and only if $\Gamma_2$ admits a strongly aperiodic SFT.
\end{thmx}

We remark that while Cayley graphs of groups have a lot of symmetries, the same is not necessarily true for a blueprint, thus, in principle, it is much easier for a blueprint to admit a strongly aperiodic SFT. Hence~\Cref{thm:SA} can be used as a tool to show that complicated groups admit strongly aperiodic SFTs (this is in fact what has implicitly been used in the literature). We also show a variant of~\Cref{mainthm:SA} that requires less assumptions on the quasi-isometries, but more on the structure of the blueprint (\Cref{thm:SA_v2}).

We generalize two results of~\cite[Corollary 4.24]{barbieri2024medvedev} on the invariance under computable quasi-isometries of the set of Medvedev degrees of SFTs between finitely presented groups. We show that the existence of SFTs with uncomputable points is an invariant of quasi-isometry (\Cref{thm:medvedev_noassumpt}) and that under a stronger computability assumption, the whole class of Medvedev degrees of SFTs on a blueprint is an invariant of quasi-isometry (\Cref{thm:medvedev_full}).

Finally, we apply our formalism to the context of $d$-dimensional geometric tilings. From a set of punctured tiles $\PP$ that define a tiling space $\Omega(\PP)$ with finite local complexity, and two parameters $K,L\in\N$, we construct a blueprint $\Gamma(\PP,K,L)$ whose space of models is homeomorphic to the space of punctured tilings $\Omega_0(\PP)$, provided $K\geq 117$ and $L\geq 2K+14$ (Proposition~\ref{prop:geometric_main}). In this context, define a colored tiling as a tiling made up of tiles from $\PP\times A$, for a finite alphabet $A$. These tilings can be seen as geometric tilings from $\Omega(\PP)$ where each tile is given a label or color from $A$. Then, the $\PP$-domino problem is the decision problem that asks, given a finite set of colored tiles and a finite set of forbidden patterns, whether there exists a symbolic-geometric tiling where no forbidden pattern occurs. By combining all the previous results, and using the fact that for $K\geq 117$ and $L\geq 2K+14$, the blueprint $\Gamma(\PP, K,L)$ is quasi-isometric to $\Z^d$, we obtain the following result.

\begin{thmx}[Theorem~\ref{thm:domino_geom}]
       Let $d\geq 2$ and $\PP$ be a finite set of punctured tiles with finite local complexity. Then, the $\PP$-domino problem is undecidable.
\end{thmx}

We remark that in~\Cref{thm:domino_geom} we also show that a variant of the domino problem with respect to a single fixed tiling is also shown to be undecidable. This result generalizes the results from~\cite{hellouin2023domino,de2024decision} on the domino problem to all dimensions $d\geq 2$. We also refer to other attempts at capturing geometric tilings with algebraic structures~\cite{coulbois2024aperiodic} and regular grids~\cite{sadun2006tilings}. We also note that even for polygonal tilesets with rational coordinates, it is undecidable whether their associated tiling space has finite local complexity~\cite{de2024decision}.

\subsection*{Structure of the article}

We begin by introducing blueprints in Section~\ref{sec:blueprints}. Here we introduce the notion of a model, its corresponding model graphs, a blueprint's model space as well as the topology and dynamics of this space. We then move to subshifts over blueprints in Section~\ref{sec:subshifts}. We define two notions of subshift: one that depends on a fixed model that generalizes the natural notion from groups, but is not endowed with dynamics (Definition~\ref{def:phi_subshift}), and one over the whole blueprint composed of model-configuration pairs that is endowed by a partial monoid action (Definition~\ref{def:subshift}). For this latter definition we recover results from the classic theory including the Curtis-Hedlund-Lyndon Theorem. Section~\ref{sec:quasi-isometries} is the heart of the article, where we talk about quasi-isometries quasi-metric spaces (blueprints in particular), and prove the black-box theorem (Theorem~\ref{thm:QI}). In Section~\ref{sec:qi_rigidity}, we use the theorem to prove that the undecidability of the domino problem, and the existence of strongly aperiodic subshifts of finite type are quasi-isometry invariants for finitely presented strongly connected blueprints. We also prove the invariance of Medvedev degrees of subshifts, provided the quasi-isometries are computable and the blueprints have decidable word problems. Finally, Section~\ref{sec:geometric} is concerned with $d$-dimensional geometric tilings of the Euclidean space. Here we prove that the structure of a tiling with finite local complexity can be captured by a blueprint (Proposition~\ref{prop:geometric_main}), and use some of the aforementioned results to prove that two variants of the geometric domino problem are undecidable for these tilings (\Cref{thm:domino_geom}). We finish the section by discussing the case of tilings of $\RR^d$ given by a finite set of tiles up to isometries, and the case of tilings of hyperbolic spaces.

%
%
%
%
\section{Blueprints}
\label{sec:blueprints}

\begin{defn}
\label{def:blueprint}
    A \define{blueprint} is a tuple $\Gamma = (M,S,\init,\ter,R)$ which consists of:
    \begin{itemize}
        \item A nonempty set $M$ of \define{states}.
        \item A nonempty set $S$ of \define{generators}.
        \item Two functions $\init\colon S\to M$ and $\ter\colon S\to\mathcal{P}(M)\setminus \{\varnothing\}$ which denote respectively the \define{initial state} and \define{possible final states} of each generator.
        \item A set $R$ of \define{relations}, which contains pairs of the form $(u,v)$ with $u,v\in S^*$.
    \end{itemize}
\end{defn}

For a nonempty word $w = w_1\dots w_n \in S^*$, the maps $\init$ and $\ter$ extend naturally by $\init(w) = \init(w_1)$ and $\ter(w) = \ter(w_n)$. In order to simplify the notation, we will often only write $\Gamma = (M,S,R)$ and leave the maps $\init$ and $\ter$ implicit. We say that a blueprint is \define{finitely generated} if both $M$ and $S$ are finite, and we say it is \define{finitely presented} if $M$, $S$ and $R$ are finite sets.\\


A word $w = w_1\dots w_n \in S^*$ is called \define{$\Gamma$-consistent} if $w$ is either the empty word or for every $i \in \{1,\dots, n-1\}$ we have $\init(w_{i+1}) \in \ter(w_i)$. Given two $\Gamma$-consistent words $u,v \in S^*$ we say they are \define{$\Gamma$-similar} if there exists a relation $(w,w') \in R$ and words $x,y \in S^*$ such that $u = xwy$ and $v = xw'y$. We say that $u,v$ are \define{$\Gamma$-equivalent} if they are equivalent for the equivalence relation generated by $\Gamma$-similarity. We denote by $\underline{w}_{\Gamma}$ the equivalence class under $\Gamma$-equivalence of a consistent word $w$.

\begin{defn}
\label{def:consistent}
    A map $\varphi\colon S^*\to M\cup\{\varnothing\}$ is \define{$\Gamma$-consistent} if $\varphi(\varepsilon)\in M$ and for all $w\in S^*$ and $s \in S$,
    \begin{itemize}
        \item if $\init(s) = \varphi(w)$, then $\varphi(ws)\in \ter(s)$,
        \item if $\init(s) \neq \varphi(w)$ then $\varphi(ws)= \varnothing$.
    \end{itemize}
\end{defn}

The support of a $\Gamma$-consistent map $\varphi$, denoted $\supp(\varphi)$, is defined as the set of words $w\in S^*$ such that $\varphi(w)\in M$. Notice that if $\varphi$ is $\Gamma$-consistent, then $\supp(\varphi)$ is an infinite subtree of $S^*$ whose paths are made up of $\Gamma$-consistent words.

\begin{defn}
\label{def:model}
We say a $\Gamma$-consistent map $\varphi\colon S^{*} \to M\cup\{\varnothing\}$ is a $\Gamma$-\define{model} if for every pair of $\Gamma$-equivalent words $u,v\in \supp(\varphi)$ we have $\varphi(u) = \varphi(v)$. The space of all $\Gamma$-models is defined as
\[\M(\Gamma) = \{\varphi\in (M\cup \{\varnothing\})^{S^*}\
: \varphi \text{ is a $\Gamma$-model } \}.\]
\end{defn}

We want to think on a model as a geometrical realization of a particular choice of states taken from a blueprint. With that aim in mind, we associate a graph to each model. More precisely, given $\varphi\in\M(\Gamma)$, we define its \define{model graph} as the directed labeled graph $G(\Gamma, \varphi) = (V,E)$ given by the set of vertices $V = \{ \underline{w}_{\Gamma} : w\in \supp(\varphi)\}$ and edges \[E = \{ (\underline{w}_{\Gamma}, \underline{ws}_{\Gamma}, s) : w \in \supp(\varphi), s\in S \mbox{ and } \init(s) = \varphi(w)\}.\] 

Before equipping these spaces with topology and dynamics, let us provide a few simple examples.

\begin{example}\label{ex:1-2-3}
    Consider the ``$1$-$2$ tree'' blueprint $\Gamma = (M,S,R)$ given by $M = \{\texttt{1},\texttt{2}\}$, $S = \{ s,\ell,r\}$ and $R = \varnothing$, where the initial and terminal functions are described in~\Cref{table:rules_for_unary_binary_tree}.
    \begin{table}[ht!]
			\centering
			\begin{tabular}{|c||c|c|c|}
				\hline
				$S$ & $s$ & $\ell$ & $r$\\
				\hline
				$\init$ & $\texttt{1}$ & $\texttt{2}$ & $\texttt{2}$ \\
				\hline
				$\ter$ & $\{\texttt{1},\texttt{2}\}$ & $\{\texttt{1},\texttt{2}\}$ & $\{\texttt{1},\texttt{2}\}$\\
				\hline
			\end{tabular}
   \caption{Rules for the $1$-$2$ tree blueprint.}
        \label{table:rules_for_unary_binary_tree}
		\end{table}

  In this blueprint, all words in $S^*$ are $\Gamma$-consistent, and as $R=\varnothing$ the space of $\Gamma$-models coincides with the space of $\Gamma$-consistent maps, which represents the space of rooted trees in which every vertex can arbitrarily have one or two descendants. The graph of a typical model is shown in~\Cref{fig:123-model}.

  	\begin{figure}[ht!]
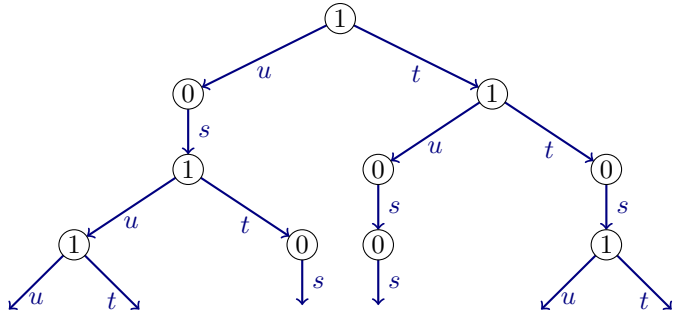

		\centering
        \include{fig/123}
		\caption{The graph of a model in the $1$-$2$ tree blueprint. The state $\texttt{1}$ is always followed by the single generator $s$ while the state $\texttt{2}$ is followed by the two generators $\ell,r$.}
		\label{fig:123-model}
	\end{figure}
\end{example}

In the next example we show that when we consider blueprints with a single state, we can recover Cayley graphs of monoids.

\begin{example}\label{ex:monoide_y_grupo}
    Let $\Gamma = (M,S,R)$ be a blueprint such that $M = \{m\}$ is a singleton. Then the initial and terminal maps are superfluous as the only option is that for every $s \in S$, $\init(s)=m$ and $\ter(s)=\{m\}$. It follows that the space of models is the singleton which consists of the constant map $\varphi\colon S^* \to \{m\}$. In this case, the graph $G(\Gamma, \varphi)$ corresponds to the Cayley graph of the monoid generated by $S$ given by the set of relations $R$. In the case where for every $s \in S$ there is $s^{-1}\in S$ with the relations $(ss^{-1},\varepsilon)$ and $(s^{-1}s,\varepsilon)$, then $G(\Gamma, \varphi)$ is the Cayley graph of the group generated by $S$ under the relations $R$.

    In other words, there is a correspondence between monoids and blueprints with $|M|=1$.
\end{example}

In the case where there is more than one state but the terminal function is deterministic, we obtain Cayley graphs for small categories and groupoids (see~\cite{Ibort2019} for an introductory reference).

\begin{example}
    Let $\Gamma = (M,S,R)$ be a blueprint such that for every $s \in S$, $|\ter(s)|=1$. Then the graph of each model corresponds to a connected component of a Cayley graph of a small category. Furthermore, if inverse relations are added as in~\Cref{ex:monoide_y_grupo}, the graph of each model will correspond to a connected component of a Cayley graph of a groupoid.
\end{example}


Finally, let us show a more interesting example that yields combinatorial models of unrooted binary tilings of the hyperbolic plane.

\begin{example}\label{ex:hyperbolic_tilings}
    Consider the \define{hyperbolic tiling blueprint} $\mathcal{H} = (M,S,R)$ where $M = \{\texttt{0},\texttt{1}\}$ and $S = \{s_{0}^0, s_{0}^1, s_{1}^0, s_{1}^1, t^{\pm1}_0, t^{\pm1}_1, p_0, p_1\}$ and with functions given by~\Cref{table:rules_for_hyperbolic plane}.
    \begin{table}[ht!]
			\centering
			\begin{tabular}{|c||c|c|c|c|c|c|c|c|c|c|}
				\hline
				$S$ & $s_0^0$ & $s_0^1$ & $s_1^0$& $s_1^1$ & $t_0^{+}$ & $t_0^{-}$ & $t_1^{+}$ & $t_1^{-}$ & $p_0$ & $p_1$\\
				\hline
				$\init$ & $\texttt{0}$ & $\texttt{0}$ & $\texttt{1}$ & $\texttt{1}$ & $\texttt{0}$ & $\texttt{1}$ & $\texttt{1}$ & $\texttt{0}$ & $\texttt{0}$ & $\texttt{1}$\\
				\hline
				$\ter$ & $\{\texttt{0}\}$ & $\{\texttt{1}\}$ & $\{\texttt{0}\}$ & $\{\texttt{1}\}$ & $\{\texttt{1}\}$ & $\{\texttt{0}\}$ & $\{\texttt{0}\}$ & $\{\texttt{1}\}$ & $\{\texttt{0},\texttt{1}\}$& $\{\texttt{0},\texttt{1}\}$\\
				\hline
			\end{tabular}
   \caption{Rules for the binary hyperbolic tiling blueprint.}
        \label{table:rules_for_hyperbolic plane}
		\end{table}

    For the relations, we let $R = R_0 \cup R_1$ where:
    \[R_0 = \{(s_0^0 t^{+}_0,s_0^1), \ (s_{0}^1t^{+}_1,t^{+}_0 s_1^0), \ (p_0 s_0^1,t^{+}_0), \ (s_0^0p_0,\varepsilon), \ (p_0s_1^1,t^{+}_0), \ (s_0^1 p_1,\varepsilon), \ (t_0^{+}t_1^{-},\varepsilon), \ (t_0^{-}t^{+}_1,\varepsilon)\},\]
    and 
    \[R_1 = \{(s_1^0 t^{+}_0,s_1^1), \ (s_{1}^1t^{+}_1,t^{+}_1 s_0^0), \ (p_1 t^{+}_0 s_1^0,t^{+}_1), \ (s_1^0p_0,\varepsilon), \ (p_1t^{+}_1s_0^0,t^{+}_1), \ (s_1^1 p_1,\varepsilon), \ (t^{+}_1t_0^{-},\varepsilon), \ (t_1^{-}t^{+}_0,\varepsilon)\}.\]

     Let us look at different models and their corresponding graphs for $\mathcal{H}$. Notice that by choosing the value of $\varphi(\varepsilon)\in M$, we determine the value of all other words which do not contain $p_0$ or $p_1$ (as their final states are completely determined). In fact, using the relations it can be deduced that a model is uniquely determined by the sequence of terminal states chosen for $p_i$'s. Explicitly, take $x\in\{0,1\}^\N$ and define $\varphi(x)\in\M(\mathcal{H})$ by $\varphi(x)(\varepsilon) = x_0$ and $\varphi(x)(p_{x_0}p_{x_1}\dots p_{x_n}) = x_{n+1}$. By the argument above, the rest of the values of $\varphi$ are uniquely determined by the sequence. It follows that the map $\varphi\colon\{0,1\}^\N \to \M(\mathcal{H})$ defines a bijection. A finite portion of the graph of a typical model is shown in~\Cref{fig:hyper}.
    
\begin{figure}[H]
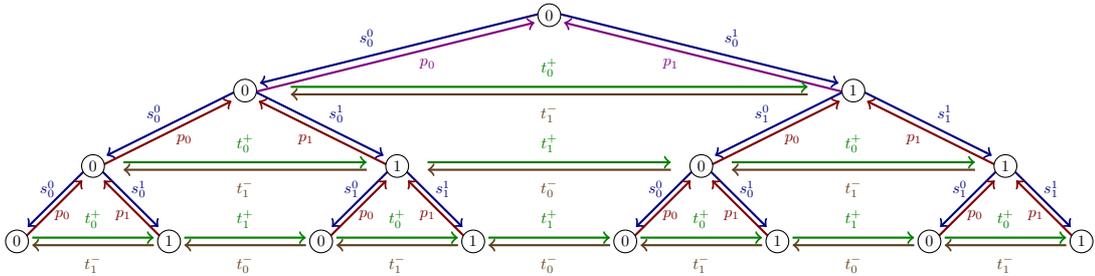

    \centering
    \includestandalone[width=\textwidth]{fig/hyperbolic}
    \caption{A portion of a model graph of $\mathcal{H}$}
    \label{fig:hyper}
\end{figure}
\end{example}


A particular notion for blueprints that we will require later on is that of a quotient.

\begin{defn}
    Consider a blueprint $\Gamma = (M, S, R, \init, \ter)$. We say that the blueprint $\Gamma' = (M, S, R', \init, \ter)$ is a \define{quotient} of $\Gamma$ if for all $u,v\in S^*$, 
    \[\underline{u}_{\Gamma} = \underline{v}_{\Gamma} \ \implies \underline{u}_{\Gamma'} = \underline{v}_{\Gamma'}\]
\end{defn}

Notice that by Definition~\ref{def:model}, if $\Gamma'$ is a quotient of $\Gamma$, every $\Gamma'$-model is a $\Gamma$-model.

\begin{remark}
    We can make an alternative definition of a quotient that resembles that of a group. For a blueprint $\Gamma$ generated by the set $S$, we denote its set of $\Gamma$-equivalence classes by $\underline{\Gamma} = \{\underline{u}_{\Gamma} : u\in S^*\}$. A \define{blueprint morphism} between $\Gamma_1$ and $\Gamma_2$ is a function $f\colon\underline{\Gamma_1}\to \underline{\Gamma_2}$ such that $f(\underline{\varepsilon}_{\Gamma_1})=\underline{\varepsilon}_{\Gamma_2}$ and $f(\underline{uv}_{\Gamma_1}) = f(\underline{u}_{\Gamma_1})f(\underline{v}_{\Gamma_1})$ for all $u,v\in S_1^*$. We make an abuse of notation and write $f:\Gamma_1\to \Gamma_2$. We say such a map is a \define{blueprint isomorphism} when it is bijective. With this definition, a short proof shows that if there exists a surjective blueprint morphism $\pi:\Gamma\to \Omega$, then there exists a quotient $\Gamma'$ of $\Gamma$ that is blueprint isomorphic to $\Omega$. In fact, if $\Gamma = (M,S,\init,\ter,R)$, the quotient is given by $\Gamma' = (M,S,\ter,\init, R')$ where $R' = \{(u,v)\in (S^*)^2 : \pi(\underline{u}_\Gamma) = \pi(\underline{v}_\Gamma)\}$.
\end{remark}

\subsection{Topology and dynamics of the model space}

Let $\Gamma = (M, S, R)$ be a finitely generated blueprint, and $\M(\Gamma)$ its corresponding model space. We endow $\M(\Gamma)$ with the topology induced by the prodiscrete topology on $(M\cup \{\varnothing\})^{S^*}$. In other words, a sequence of maps $(\varphi_n)_{n \in \N}$ in $(M\cup \{\varnothing\})^{S^*}$ converges to $\varphi \in(M\cup \{\varnothing\})^{S^*}$ if for every $w \in S^*$ we have that $\varphi_n(w) = \varphi(w)$ for every large enough $n$. Clearly $\M(\Gamma)$ is closed in $(M\cup \{\varnothing\})^{S^*}$ and thus the induced topology on $\M(\Gamma)$ makes it a compact metrizable space.\\

The space $\M(\Gamma)$ admits a natural partial right monoid action by $S^*$, which is given by \[(\varphi\cdot w)(u) = \varphi(wu),\] and defined only when $w\in \supp(\varphi)$. In this way, $\supp(\varphi\cdot w) = \{u\in S^{*} : wu\in \supp(\varphi)\}$. The \define{orbit} of $\varphi \in \M(\Gamma)$ is given by
\[\orb(\varphi) \coloneqq \{\varphi\cdot w : w\in \supp(\varphi)\}.\]
A model $\varphi$ is called \define{dense} if $\overline{\orb(\varphi)} = \M(\Gamma)$.

\begin{remark}
    In the case when $w\notin \supp(\varphi)$, we get that $\varphi(wu)=\varnothing$ for every $u \in S^*$, thus we may set $\varphi \cdot w$ as the constant $\varnothing$ map. This is not a model, but it makes some definitions (for instance~\Cref{def:morphism}) more natural.
\end{remark}

\begin{defn}
    A blueprint $\Gamma$ is \define{transitive} if there exists a dense model, we say that it is \define{minimal} if every model is dense.
\end{defn} 

\begin{example}
    Consider the blueprint $\Gamma =(M,S,R)$ given by $M = \{\texttt{0},\texttt{1}\}$, $S = \{a,b,c\}$, $R = \varnothing$ and initial and terminal functions given by~\Cref{table:rules_for_boring_example}.
    \begin{table}[ht!]
			\centering
			\begin{tabular}{|c||c|c|c|}
				\hline
				$S$ & $a$ & $b$ & $c$\\
				\hline
				$\init$ & $\texttt{0}$ & $\texttt{1}$ & $\texttt{1}$ \\
				\hline
				$\ter$ & $\{\texttt{0}\}$ & $\{\texttt{1}\}$ & $\{\texttt{1}\}$\\
				\hline
			\end{tabular}
   \caption{Rules for a blueprint which is neither minimal nor transitive.}
        \label{table:rules_for_boring_example}
		\end{table} 

There are precisely two models $\varphi_1,\varphi_2$ for $\Gamma$ which depend upon the value $\varphi(\varepsilon)$. The model with $\varphi_1(\varepsilon)=\texttt{0}$ has support $\{a\}^*$ and is constantly $\texttt{0}$ in its support, whereas the model with $\varphi_2(\varepsilon)=\texttt{1}$ has support $\{b,c\}^*$ and is constantly $\texttt{1}$ in its support. Geometrically, $G(\Gamma, \varphi_1)$ is a one-sided infinite path, whereas $G(\Gamma, \varphi_2)$ is the rooted infinite binary tree. Clearly neither of these two models is dense in $\M(\Gamma)$.
\end{example}

\begin{example}
    Recall the $1$-$2$ tree blueprint from~\Cref{ex:1-2-3}. This blueprint is clearly non-minimal as the model with support $\{s\}^*$ and constantly $\texttt{1}$ in its support is not dense. However, notice that in this blueprint given two models $\varphi_1$ and $\varphi_2$ and $n \in \N$ one can always choose $u \in \supp(\varphi_1)$ of length $n$ and construct a new model $\varphi'$ with $\varphi'(w) = \varphi_1(w)$ and $\varphi(uw) = \varphi_2(w)$ for all words with $|w| \leq n$. Enumerating all possible restrictions and iterating this process one can construct a model with dense orbit, thus this blueprint is transitive.
\end{example}

Clearly blueprints with a single state (that is, representations of finitely presented monoids) are minimal. We shall provide a less obvious example.

\begin{example}
   The hyperbolic tiling blueprint $\mathcal{H}$ from~\Cref{ex:hyperbolic_tilings} is minimal. Recall that we have a bijection $f\colon\{0,1\}^\N \to \M(\mathcal{H})$, where $x$ determines the sequence of $p_i$'s of the model. Consider the odometer map $t\colon\{0,1\}^\N\to \{0,1\}^\N$ which changes the leftmost $0$ of a sequence to a $1$ and turns all $1$s to the left of it to $0$s. It is easy to verify that $f(tx) = f(x)\cdot t^+_{x_0}$ and $f(t^{-1}x) = f(x)\cdot t^-_{x_0}$

   As the $\ZZ$-action induced by $t$ on $\{0,1\}^\N$ is minimal, it follows that the orbit of a model by $\{t_0^+,t_0^-,t_1^+,t_1^-\}^*$ is already dense, thus $\mathcal{H}$ is minimal. 
\end{example}

%
\section{Subshifts on Blueprints}
\label{sec:subshifts}

The notion of subshift is usually defined on a group or monoid as a closed and shift-invariant subspace of the space of all maps to a finite set and can be characterized as the set of configurations in that structure which avoid a list of forbidden patterns. In this section, we will develop an analogous notion for blueprints. Although not technically the same object, related generalizations can be found in~\cite{aubrun2019domino, bartholdi2022monadic, barbieri2016domino, bartholdi2020simulations,arrighi2023graph}. For the remainder of the section we will fix a blueprint $\Gamma = (M,S,R)$ and a finite set $A$ called \define{alphabet} which does not contain the symbol $\varnothing$. As in the case of models, for a function $f\colon X \to Y$ with $\varnothing \in Y$, we denote its support by $\supp(f) = X\setminus f^{-1}(\varnothing)$. \\

A \define{pattern} is a map $p \colon S^* \to (M \times A) \cup \{\varnothing\}$ with finite support. We begin with the definition of subshift for a fixed $\Gamma$-model $\varphi$. 
\begin{defn}
\label{def:phi_subshift}
    Let $\varphi \in \M(\Gamma)$ and let $\mathcal{F}$ be a set of patterns. The $\varphi$-\define{subshift} induced by $\mathcal{F}$  is given by \[ X[\Gamma,\varphi,{\mathcal{F}}] = \{ x \in (A\cup \{\varnothing\}  )^{S^*} : \mbox{ (s1)-(s3) are satisfied }   \},    \]
    where \begin{itemize}
        \item[(s1)] $\supp(x) = \supp(\varphi)$,
        \item[(s2)] For every pair of $\Gamma$-equivalent words $u,v \in \supp(\varphi)$, we have $x(u)=x(v)$,
        \item[(s3)] For every $p \in \mathcal{F}$ and $w \in \supp(\varphi)$ there exists $u \in \supp(p)$ such that $(\varphi(wu),x(wu)) \neq p(u)$.
    \end{itemize}
\end{defn}

\begin{example}
    If we let $\mathcal{F}=\varnothing$ we obtain the following subshift which we call the \define{full $\varphi$-shift} and we denote by $A[\Gamma,\varphi]$. \[ A[\Gamma,\varphi] = \{x \in  (A\cup \{\varnothing\}  )^{S^*} :  \mbox{ (s1) and (s2) are satisfied } \}.   \]
\end{example}

\begin{remark}
    Given a configuration $x$ in a $\varphi$-subshift, condition (s2) ensures that the map $\widehat{x}\colon G(\Gamma,\varphi)\to A$ given by $\widehat{x}(\underline{w}_{\Gamma}) = x(w)$ is well defined, thus a subshift in a model can also be thought as set of colorings of the associated model graph described by a set of forbidden patterns. 
\end{remark}

The word ``subshift'' seems inappropriate for the objects above, as in general the spaces $X[\Gamma,\varphi,\mathcal{F}]$ do not admit any kind of natural shift action which leaves them invariant. However, if we look at the space of pairings $(\varphi,x)$ with $\varphi$ a model and $x \in X[\Gamma,\varphi,\mathcal{F}]$, a partial ``shift'' action naturally appears.

\begin{defn}
\label{def:subshift}
    Given a set of patterns $\mathcal{F}$, the $\Gamma$-\define{subshift} induced by $\mathcal{F}$  is given by \[ X[{\Gamma},\mathcal{F}] = \{ (\varphi,x) : \varphi \in \M(\Gamma), x \in  X[\Gamma, \varphi, \mathcal{F}]\}.    \]

    Similarly, the full $\Gamma$-shift is the space $A[\Gamma] = \{  (\varphi,x) : \varphi \in \M(\Gamma), x \in  A[\Gamma,\varphi]\}$.
\end{defn}

Notice that the space $A[\Gamma]$ is a closed subset of $((M\times A)\cup \{\varnothing\})^{S^*}$ with the prodiscrete topology, and it admits a natural partial right action of $S^*$ where \[ \left((\varphi,x)\cdot w\right)(u) = (\varphi(wu),x(wu))  \mbox{ for every } w \in \supp(\varphi), u \in S^*. \]

\begin{prop}
    A set $X\subset A[\Gamma]$ is a $\Gamma$-subshift for some set of forbidden patterns $\mathcal{F}$ if and only if it is closed and invariant under the partial action of $S^*$.
\end{prop}

\begin{proof}
    It is clear by definition that a $\Gamma$-subshift is closed and invariant under the partial action of $S^*$. Conversely, given a pattern $p \colon S^* \to (M\times A) \cup \{\varnothing\}$, define its cylinder by \[ [p] = \{ (\varphi,x) \in ((M\times A)\cup \{\varnothing\})^{S^*} : (\varphi,x)(u) = p(u) \mbox{ for every } u \in \supp(p) \}.\]
    As cylinders form a base of the prodiscrete topology in $((M\times A)\cup \{\varnothing\})^{S^*}$, it follows that there exists a set of patterns $\mathcal{F}$ such that \[  X = A[\Gamma] \cap \left(  ((M\times A)\cup \{\varnothing\})^{S^*}  \setminus  \bigcup_{p \in \mathcal{F}}[p]\right) = A[\Gamma] \setminus \bigcup_{p \in \mathcal{F}}[p]. \]
    In particular, this shows that $X[\Gamma,\mathcal{F}]\subset X$. We claim that $X = X[\Gamma,\mathcal{F}]$ and thus is precisely the $\Gamma$-subshift induced by $\mathcal{F}$. Indeed, let $(\varphi,x) \in X$. As $X\subset A[\Gamma]$, it follows that $\varphi \in \M(\Gamma)$ and that $x$ satisfies conditions (s1) and (s2). It suffices to verify that $x$ satisfies condition (s3). Fix $p \in \mathcal{F}$, as $X$ is invariant under the partial action of $S^*$, it follows that for every $w \in \supp(\varphi)$ we have that $(\varphi,x)\cdot w \in X$ and thus that $(\varphi,x)\cdot w \notin [p]$, in other words, that there exists $u \in \supp(p)$ for which $(\varphi(wu),x(wu))\neq p(u)$. Thus condition (s3) also holds.
\end{proof}

\begin{defn}
    Let $\Gamma$ be a blueprint and $\varphi$ be a $\Gamma$-model. We say a $\varphi$-subshift (resp. a $\Gamma$-subshift) $X$ is of \define{finite type}, which we abbreviate as $\varphi$-SFT (resp. $\Gamma$-SFT), if there exists a finite set $\mathcal{F}$ of forbidden patterns such that $X = X[\Gamma,\varphi, \mathcal{F}]$ (resp. $X = X[\Gamma,\mathcal{F}]$). 
\end{defn}

\begin{example}
\label{ex:hardsquare}
    Consider the alphabet $A = \{ \symb{0},\symb{1}\}$ and the set $\mathcal{F}$ of all patterns with support $\{\varepsilon,s\}$ for some $s \in S$ such that $p(\varepsilon) = (m,\symb{1})$ and $p(s)=(m',\symb{1})$ for some $m,m'\in M$. For a model $\varphi$ the subshift $X[\Gamma,\varphi,\mathcal{F}]$ represents the space of all maps from $G(\Gamma,\varphi)$ to $A$ in such a way that no pair of symbols $\symb{1}$ occur adjacent to each other. We call $X[\Gamma,\FF]$ the \define{hard-square} shift on $\Gamma$. An example of a configuration of this subshift on the hyperbolic tiling model from Example~\ref{ex:hyperbolic_tilings} is shown in Figure~\ref{fig:hyper_square}, where $\symb{1}$ is represented by \begin{tikzpicture}
        \draw[black, thick, fill=rojo] (0,0) circle (0.1); 
    \end{tikzpicture}, and $\symb{0}$ by \begin{tikzpicture}
        \draw[black, thick, fill=azul] (0,0) circle (0.1); 
    \end{tikzpicture}.
    \begin{figure}[H]
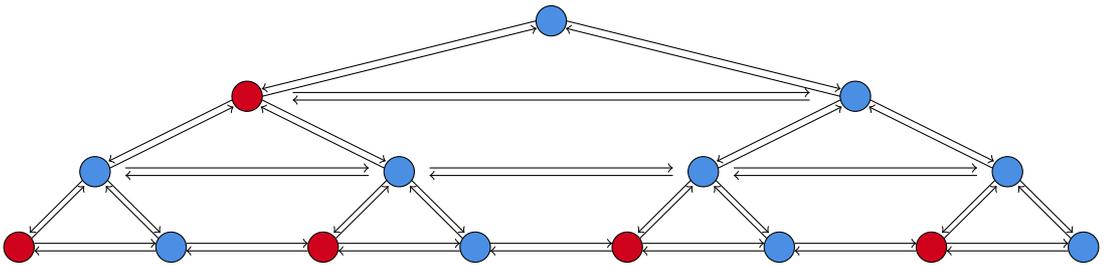

    \centering
    \includestandalone[width=\textwidth]{fig/hard_square}
    \caption{A portion of the hard-square subshift on a model graph of the hyperbolic tiling blueprint $\mathcal{H}$. }
    \label{fig:hyper_square}
\end{figure}
\end{example}

\begin{defn}
    Let $\varphi$ be a $\Gamma$-model. We say a $\varphi$-subshift (resp. a $\Gamma$-subshift) $X$ is a \define{nearest neighbor SFT}, if there exists a finite set $\mathcal{F}$ of forbidden patterns, all of them with support of the form $\{\varepsilon,s\}$ for some $s \in S$ such that $X = X[\Gamma,\varphi, \mathcal{F}]$ (resp. $X = X[\Gamma,\mathcal{F}]$).
\end{defn}

This particular class of SFTs captures the dynamics of every SFT through the following notion of equivalence.

\begin{defn}\label{def:morphism}
    Let $X,Y$ be two $\Gamma$-subshifts. A map $\phi\colon X\to Y$ is said to be a \define{morphism} if it is continuous and for every $w \in S^*$ we have that $\phi((\varphi,x)\cdot w) = \phi(\varphi,x)\cdot w$. Furthermore, if $\phi$ is bijective, we say it is a \define{conjugacy} and that $X$ and $Y$ are \define{topologically conjugate}.
\end{defn}

\begin{remark}
    In all of our applications the morphisms will fix the model and only modify the alphabet, in this case we also have the property that $\phi((\varphi,x)\cdot w)$ is well defined if and only if $\phi(\varphi,x)\cdot w$ is well-defined. This corresponds to what happens in the classical setting of symbolic dynamics on groups. We keep this more general definition mostly because the next proofs work in that setting.
\end{remark}

Morphisms between subshifts on blueprints behave much in the same way as morphisms between subshifts over groups. We say a map $\phi\colon A[\Gamma]\to B[\Gamma]$ is a \define{sliding-block code} if there exists a finite subset $F\Subset S^*$ and a local map $\Phi\colon ((M\times A)\cup \{\varnothing\})^{F}\to((M\times B)\cup \{\varnothing\})$ such that $\phi(\varphi, x)(w) = \Phi\bigl(((\varphi, x)\cdot w)|_{F}\bigr)$. We state a generalization of the classic Curtis-Hedlund-Lyndon theorem for blueprints.

\begin{theorem}
\label{thm:CHL}
    Let $X\subseteq A[\Gamma]$ and $Y\subseteq B[\Gamma]$ be two $\Gamma$-subshifts, and $\phi:X\to Y$ a map. Then, $\phi$ is a morphism if and only if it is a sliding-block code.
\end{theorem}

\begin{proof}
For simplicity, we denote $\mathcal{A} = (M\times A)\cup \{\varnothing\}$, and $\mathcal{B} = (M\times B)\cup \{\varnothing\}$. Suppose $\phi\colon X\to Y$ is a morphism. Then, because the projection map $\pi:\mathcal{B}[\Gamma]\to \mathcal{B}$ defined by $\pi(\varphi, x) = (\varphi(\varepsilon), x(\varepsilon))$ is continuous for the prodiscrete topology, the composition $\pi\circ\phi$ is also continuous. Therefore, for each $(\varphi,x)$ there exists finite subset $P_{(\varphi,x)}$ that determines the value $\pi\circ\phi(\varphi,x)$, that is, $(\varphi',x')|_{P_{(\varphi,x)}} = (\varphi,x)|_{P_{(\varphi,x)}}$ implies $\pi\circ\phi(\varphi',x') = \pi\circ\phi(\varphi,x)$. Consider the open sets
\[U(\varphi,x, P_{(\varphi,x)}) = \{(\varphi',x') : (\varphi', x')|_{P_{(\varphi,x)}} = (\varphi,x)|_{P_{(\varphi,x)}}\},\]
for all $(\varphi,x)\in X$. Evidently, every pair $(\varphi, x)$ is contained in $U(\varphi,x,P_{(\varphi,x)})$. Therefore, the union of all these open sets covers $A[\Gamma]$. As $X$ is compact, we can extract a finite subcover $\{U(\varphi_i, x_i, P_i)\}_{i=1}^n$ with $P_i = P_{(\varphi_i,x_i)}$ . Let $F$ be the union of all the $P_i$. Then, if $(\varphi, x)$ and $(\varphi', x')$ coincide on $F$, their images coincide, that is, $\pi\circ\phi(\varphi, x) = \pi\circ\phi(\varphi',x')$. We can therefore define a local function from the patterns $\{(\varphi, x)|_F : (\varphi, x)\in X\}$ to $\mathcal{B}$ and extend it arbitrarily to a map $\Phi\colon \mathcal{A}^F\to \mathcal{B}$. Then, because $\phi$ is a morphism,
\[\phi(\varphi, x)(w) = (\phi(\varphi, x)\cdot w)(\varepsilon) = \phi((\varphi, x)\cdot w)(\varepsilon) = \Phi\bigl(((\varphi, x)\cdot w)|_{F}\bigr).\]

Conversely, suppose $\Phi\colon\mathcal{A}^{F}\to\mathcal{B}$ is a local map such that $\phi\colon X\to Y$ is given by $\phi(\varphi, x)(w) = \Phi\bigl(((\varphi, x)\cdot w)|_{F}\bigr)$. Then, for $(\varphi, x)\in X$ and $v,w\in S^*$, 
$$(\phi(\varphi, x)\cdot v)(w) = \phi(\varphi,x)(vw) = \Phi\bigl(((\varphi,x)\cdot vw)|_{F}\bigr) = \phi((\varphi,x)\cdot v)(w).$$
It remains to show that $\phi$ is continuous. Let $p$ be a pattern on $\mathcal{B}$. By definition it can be decomposed as a finite intersection of cylinders of the form $[\beta]_w = \{(\varphi,x)\in Y : (\varphi(w),x(w))=\beta\}$, with $\beta\in \mathcal{B}$. For each of these cylinders,
$$\phi^{-1}([\beta]_w) = \{(\varphi, x)\in X : \Phi\bigl(((\varphi, x)\cdot w)|_F\bigr) = \beta\},$$
which is an open set, thus $\phi^{-1}([p])$ is also open. As cylinders form a basis of the topology, it follows that $\phi$ is continuous.
\end{proof}

\begin{prop}
    Every SFT is topologically conjugate to a nearest neighbor SFT.
\end{prop}

The proof of this proposition goes along the same lines as the proof of this result for subshifts over groups.

\begin{proof}
Let $X$ be a $\Gamma$-SFT defined by a set of forbidden patterns $\Fo$. Let \[N=\max_{p\in\Fo}\max_{w\in\supp(p)}|w|.\] Let $F = \{w \in S^* : |w|\leq N\}$. We define the alphabet $$B = \{\alpha\in ((M\times A)\cup \{\varnothing\})^{F} : \alpha(\varepsilon)\neq \varnothing \mbox{ and } \forall p\in\Fo, \exists w \in \supp(p), \alpha(w) \neq p(w)\}.$$

For $\alpha \in B$ and $w \in \supp(\alpha)$ we use the notation $\alpha(w) = (\alpha_M(w),\alpha_A(w))\in M \times A$.

We define $\mathcal{G}$ as the set of nearest neighbor patterns $q$ of support $\{\varepsilon,s\}$ over $(M\times B)$, with $s\in S$, such that if we write $q(\varepsilon) = (m,\alpha)$ and $q(s)=(m',\alpha')$ then either:
\begin{enumerate}
    \item We have $\alpha_M(\epsilon)\neq m$.
    \item We have $\alpha'_M(\epsilon) \neq m'$.
    \item There exists $w\in S^*$ with $|w|\leq N-1$ and $\alpha(sw)\neq \alpha'(w)$.
\end{enumerate}

In other words, these are all patterns where either the states are not consistent with what is encoded by their alphabet coordinates, or such that the overlap between their alphabet coordinates does not match.

 Let $Y$ be the nearest neighbor $\Gamma$-SFT defined by $\mathcal{G}$. Consider the map $\phi \colon X \to ((M\times B)\cup\{\varnothing\})^{S^*}$ given by \[  \phi(\varphi,x)(w) = \begin{cases} (\varphi(w), ((\varphi,x)\cdot w)|_{F}) & \mbox{if } w \in \supp(\varphi)\\
 (\varnothing,\varnothing) & \mbox{otherwise}.
 \end{cases}\]
 
 The map $\phi$ is a morphism by~\Cref{thm:CHL}. It is also clear that $\phi$ is injective and preserves the first coordinate. Moreover, it is clear by the definition that $\phi(X) \subset B[\Gamma]$. Finally, a direct argument shows that no forbidden patterns from $\mathcal{G}$ can occur in $\phi(\varphi,x)$ for any $(\varphi,x) \in X$ and thus we conclude that $\phi(X)\subset Y$.

 It only remains to show that for every $(\varphi,y) \in Y$ there is $(\varphi,x) \in X$ such that $\phi(\varphi,x)=(\varphi,y)$.  Let us fix $(\varphi,y) \in Y$, for $w\in \supp(\varphi)$ denote $y(w)=(m_w,\alpha_w)$. We define $x\in (A\cup \{\varnothing\})^{S^*}$ as follows \[ x(w) = \begin{cases}
   (\alpha_w)_A(\varepsilon)   & \mbox{if } w \in \supp(\varphi).\\
     \varnothing & \mbox{otherwise}.
 \end{cases}  \]

 It is clear that $x$ satisfies conditions (s1) and (s2) of~\Cref{def:phi_subshift} and thus $(\varphi,x)\in A[\Gamma]$. Let $(\varphi,z)=\phi(\varphi,x)$ and take $w \in \supp(\varphi)$, then by definition we have $z(w) = ((\varphi,x)\cdot w)|_{F}$, in other words, for every $u \in F$, \[z(w)(u) = (\varphi(wu),  x(wu)) = (\varphi(wu),  (\alpha_{wu})_A(\varepsilon)). \]

 As no forbidden patterns from $\mathcal{G}$ occur in $y$, one obtains that $m_w = \varphi(w)$ for all $w$, and also one inductively deduces that \[ (\alpha_{w})_A(u) = (\alpha_{wu})_A(\varepsilon). \]
 From where one obtains that $z = y$. 

 We finally check that $(\varphi,x) \in X$. Suppose that $(\varphi,x)\notin X$, then there exists $p \in \FF$ and $w \in \supp(\varphi)$ with $(\varphi,x)\cdot w \in [p]$. In particular, if we take $\alpha = ((\varphi,x)\cdot w)|_F$, we would have that $\alpha(\varepsilon)=(\varphi(w),x(w))\neq \varnothing$ and for all $w \in \supp(p)$, $\alpha(w)=p(w)$, from where we get that $y_w = ((\varphi,x)\cdot w)|_F \notin M \times B$, a contradiction with the assumption that $y \in Y$.\end{proof}

 \begin{remark}
     In the definition of morphism, one might go further and consider maps between subshifts $X\subseteq A[\Gamma]$ and $Y\subseteq B[\Gamma']$ on different blueprints $\Gamma$ and $\Gamma'$ with generators $S$ and $S'$ respectively. This generalized morphism would be determined by maps $\nu \colon S^* \to S'^*$ and 
     $\phi\colon X \to Y$ which satisfy that \[ \phi((\varphi,x)\cdot w) = \phi((\varphi,x))\cdot \nu(w) \mbox{ for every } (\varphi,x) \in X \mbox{ and } w\in S^*.   \]
     With this definition, one would be able to obtain any nearest neighbor $\Gamma$-SFT as a factor of one on a blueprint $\Gamma'$ that has trivial alphabet by setting $M' = M \times A$ as the set of states of $\Gamma'$ and consider for every pair $(s,a)\in S\times A$ a generator $s'$ with $\init(s')=(m,a)$ and such that $\ter(s')$ consists on all pairs $(m',a')$ with $m'\in \ter(s)$ and which do not create a forbidden pattern. We will not make use of this definition anywhere in the article.
 \end{remark}

\subsection{Computability for subshifts on blueprints}

Because we want to study the domino problem on blueprints, as well as their Medvedev degrees, in this section we give some preliminary notions from computability theory that are necessary for the rest of the article. We refer the reader to~\cite{sipser1996introduction,soare2016turing} for an in depth introduction. We say a language $L\subseteq A^*$ is
    \begin{itemize}
        \item \define{recursively enumerable} if there exists a Turing machine that accepts $w\in A^*$ if and only if $w\in L$,
        \item  \define{decidable} if there exists a Turing machine that halts on all inputs, accepts $w\in A^*$ when $w\in L$, and rejects when $w\notin L$.
    \end{itemize}
Similarly, we can define a notion of computability for functions. We say a function $f\colon A^*\to B^*$ is \define{total computable} if there exists a Turing machine that halts on all inputs $w\in A^*$ leaving $f(w)$ on the tape. To compare the decidability of different decision problems, we use the following notion of reduction.
\begin{defn}
$L$ \define{many-one reduces} to $L'\subseteq B^*$, denoted $L\leq_m L'$, if there exists a total computable function $f\colon A^*\to B^*$ such that $w\in L$ if and only if $f(w)\in L'$ for every $w$. We say $L$ is \define{many-one equivalent} to $L'$ if both $L\leq_m L'$ and $L'\leq_m L$.
\end{defn}

We remark that if $L \leq_m L'$ and $L'$ is recursively enumerable (resp. decidable), then so is $L$.\\

In this work, we want to codify the space of patterns as a formal language. Because patterns have finite support, given a pattern $p\colon S^*\to (M\times A)\cup\{\varnothing\}$, we can describe it as a finite set of tuples $(p(w), w)$ with $w\in\supp(p)$. Then, a set of forbidden patterns $\Fo$ corresponds to a union of sets of the form $\{(m_i,a_i,w_i)\}_{i\in I}$ where $I$ is finite, $(m_i,a_i) \in M \times A$, and $w_i\in S^*$.

\begin{defn}
    Let $\Gamma$ be a blueprint. We say a $\Gamma$-subshift $X$ is \define{effective}, if there exists a recursively enumerable set of forbidden patterns $\mathcal{F}$ such that $X = X[\Gamma,\mathcal{F}]$. 
\end{defn}

Lastly, we say an object $u$ can be \define{effectively constructed} from another object $t$, if there is a total computable map which having on input a description of $t$ produces a description of $u$. We will use this terminology later on to say that there is an algorithm that taking on input a description of a finite set of patterns $\FF$ and a positive integer $N$, produces a new finite set of patterns $\FF'$.

%
\section{Quasi-isometries between finitely presented blueprints}
\label{sec:quasi-isometries}

We now move on to our main result on the invariance of subshift properties by quasi-isometries. With this objective in mind and to make use of quasi-isometries, we must first understand the geometry of blueprints and their models. We do this through the notion of quasi-metric spaces.

\subsection{Quasi-metrics and Quasi-isometries}

A quasi-metric space is a tuple $(X,\rho)$ where $X$ is a set and $\rho$ is a quasi-metric, that is, a map $\rho \colon X \times X \to \R_{\geq 0}$ which satisfies all the assumptions of a metric excepting symmetry. In a quasi-metric space $(X,\rho)$ we think of $\rho(x,y)$ as the distance from $x$ to $y$, which can be different from the distance $\rho(y,x)$ from $y$ to $x$. For more information on quasi-metrics we refer the reader to~\cite{kelly1963bitopological,wilson1931quasi}. 

\begin{example}
    Let $(V,E)$ be a strongly connected directed graph. A natural quasi-metric $\rho$ on $V$ is given by the shortest directed path between the vertices, namely \[\rho(x,y) = \min\{ n \in \NN: \mbox{ there is } (v_i)_{i=0}^n \mbox{ with } v_0 = x, v_n = y \mbox{ and } (v_i,v_{i+1})\in E \}.\]
\end{example}

To make use of quasi-metrics for blueprints, we will say a blueprint $\Gamma$ is \define{strongly connected} if for every model $\varphi\in \mathcal{M}(\Gamma)$, its model graph $G(\Gamma, \varphi)$ is strongly connected, that is, if there is a path between every pair of vertices.\\

Let $(X, \rho_X)$ and $(Y, \rho_Y)$ be two quasi-metric spaces. We say a function $f\colon X\to Y$ is a \define{quasi-isometry} if there exists constants $C, D\geq 0$ and $\lambda\geq1$ such that
\begin{enumerate}
    \item $f$ is a \define{quasi-isometric embedding}: for all $x,y\in X$
    \[\frac{1}{\lambda}\rho_X(x,y) - C \leq \rho_Y(f(x), f(y))\leq \lambda \rho_X(x,y) + C,\]
    \item $f$ is \define{relatively dense}: for all $z\in Y$ there exists $x\in X$ such that 
    $$\max\{\rho_Y(z, f(x)),\rho_Y(f(x),z)\}\leq D.$$
\end{enumerate}

If there exists a quasi-isometry between $(X, \rho_X)$ and $(Y, \rho_Y)$, we say they are quasi-isometric quasi-metric spaces.

\begin{remark}
\label{rem:Nto1}
Consider two strongly connected directed graphs $G_1=(V_1,E_1)$ and $G_2=(V_2,E_2)$ with their natural quasi-metrics $\rho_1$ and $\rho_2$ respectively, and suppose that both $G_1$ and $G_2$ have uniformly bounded degree. Let $f\colon V_1\to V_2$ be a quasi-isometry given by constants $C,D,\lambda$ as above. If $u,v \in V_1$ are such that $f(u)=f(v)$, then it follows that both $\rho_1(u,v)$ and $\rho_1(v,u)$ are bounded by $\lambda C$. In particular, as the degree of $G_1$ is uniformly bounded, it follows that there is $M \in \NN$ such that $f$ is at most $M$-to-$1$. Similarly, if $x,y \in V_1$ are such that $\rho_1(x,y) = 1$, then $\rho_2(f(x),f(y)) \leq C+\lambda$. It follows that in this bounded degree case we may always choose a single large enough positive integer $N$ such that the image of adjacent vertices lie at distance $N$, and the map is at most $N$-to-$1$.
\end{remark}

\begin{defn}
    Two strongly connected blueprints $\Gamma_1,\Gamma_2$ are \define{quasi-isometric} if for all $\varphi_1\in\M(\Gamma_1)$ and $\varphi_2\in\M(\Gamma_2)$, the model graph $G(\Gamma_1, \varphi_1)$ is quasi-isometric to $G(\Gamma_2, \varphi_2)$.
\end{defn}

\begin{remark} If $\Gamma_1$ and $\Gamma_2$ are quasi-isometric, then within each blueprint all model graphs are quasi-isometric to each other, i.e. the model space of each blueprint is contained in a unique quasi-isometry class.
\end{remark}

\begin{example} Every model graph of the hyperbolic tiling blueprint from Example~\ref{ex:hyperbolic_tilings} is quasi-isometric to the hyperbolic plane, and therefore every two model graphs are quasi-isometric. 
\end{example}


\subsection{Encoding quasi-isometries through subshifts}

As stated above, our goal is to construct a subshift that encodes bounded-to-1 quasi-isometries between two finitely presented blueprints, that additionally allows us to pass subshifts from one blueprint to the other in a way that preserves some aspects of the subshift's dynamics. This is a generalization of Cohen's construction~\cite{cohen2017large} and takes elements from adaptations of that original construction~\cite{barbieri2021groups,barbieri2024medvedev}. \\

Consider two finitely presented  strongly connected blueprints $\Gamma_1$ and $\Gamma_2$, an alphabet $A$, and a set of forbidden patterns $\Fo$ for $\Gamma_2$. The objective is to construct a $\Gamma_1$ subshift $\QI(\Fo, N)$, depending on both $\Fo$ and a positive integer $N$, such that each configuration encodes a configuration of $X[\Gamma_2, \Fo]$ through a quasi-isometry from a model of $\Gamma_2$ to a model of $\Gamma_1$ that is at most $N$-to-1. Recall from Remark~\ref{rem:Nto1}, that every quasi-isometry between strongly connected uniformly bounded directed graphs is of this form for some $N$. 

The alphabet of this subshift, which will be later denoted by $B$, is made up of $N$-tuples $B_1\times\dots\times B_N$ each encoding the information of one of the possible $N$ pre-images of a point under some quasi-isometry. The information contained in each letter $b\in B_i$ for every $i\in\{1,\dots, N\}$ is either the symbol $*$ representing the fact that there is no pre-image being encoded, or the following information:
\begin{itemize}
    \item a state from $M_2$,
    \item a letter from $A$,
    \item a function from $S_2$ to $S_1^{\leq2N}$ which tells us where to move in $\Gamma_1$ if we want to move by a given generator in $\Gamma_2$ (notice that the bound $2N$ is given by the quasi-isometry),
    \item a function from $S_2$ to $\{1,\dots, N\}$ which tells us to which of the $N$ pre-images we arrive if we move by a given generator.
\end{itemize}
This way, a letter $b\in B$ encodes the image of the neighborhoods of each of the pre-images that fall on the point. The forbidden patterns of $\QI(\Fo,N)$ are given by conditions \textbf{C1} to \textbf{C5} below, which are constructed to make sure that configurations are actually coding quasi-isometries, and that the relations of $\Gamma_2$ are respected. It is here that the hypothesis of finite presentability is key to make sure $\QI(\Fo, N)$ is an SFT when $\Fo$ is finite, as we must code each relation of $\Gamma_2$ as a forbidden pattern.\\

With all the information above, we want to be able to take a configuration $(\varphi_1, x)\in \QI(\Fo, N)$, and starting from $\varepsilon$ extract a configuration $(\varphi_2, x)\in X[\Gamma_2, \Fo]$ by following the encoding of the shift. But, $\QI(\Fo, N)$ contains configurations where no pre-image is coded at $\varepsilon$. Nevertheless, because of the relative density of quasi-isometries, we know there must be a point coding a pre-image at distance at most $N$ of the origin. We therefore introduce a set $\QI'(\Fo,N)\subseteq \QI(\Fo,N)\times \{1,\dots,N\}$ of all configurations $(\varphi, x)$ and indices $i\in \{1,\dots,N\}$ such that $(\varphi(\varepsilon), x(\varepsilon))_i\neq *$, i.e. that code some pre-image at the origin. Furthermore, we define a function $\theta\colon\QI(\Fo, N)\to (S_1)^{\leq N}\times \{1,\dots,N\}$ that given a configuration $(\varphi, x)$, gives a word $w$ and an index $i$ such that $((\varphi,x)\cdot w, i)\in\QI'(\Fo, N)$.

We also introduce a function $\gamma\colon \QI'(\Fo,N)\to X[\Gamma_2, \Fo]$ that recovers the pre-image of the encoded configuration from a configuration of $\QI'(\Fo, N)$. Further still, we link the dynamics of the two subshifts through the map $\mu\colon \QI'(\Fo,N)\times (S_2)^*\to (S_1)^*\times \{1,\dots,N\}$ which tells us by which word of $(S_1)^*$ we must shift a configuration $(\varphi, x)$ (such that $(\varphi, x,i)\in\QI'(\Fo,N)$ for some $i$) and which index we must use to obtain the value of the $\gamma(\varphi, x,i)$ at a word from $(S_2)^*$ in its support.\\

We summarize the properties of these three maps in the following definition.
\begin{defn}
\label{def:coding}
    Consider two finitely presented  strongly connected blueprints $\Gamma_1$ and $\Gamma_2$, a set of forbidden patterns $\Fo$ for $\Gamma_2$, and a positive integer $N$. We say a $\Gamma_1$-subshift $X$ \define{codes $\Fo$ through quasi-isometries for $N$}, if there exist  $Q\subseteq X\times\{1, ..., N\}$ and computable maps $\theta\colon X \to (S_1)^{\leq N} \times \{1,...,N\}$, $\mu\colon Q\times (S_2)^*\to (S_1)^* \times \{1,\dots,N\}$, and $\gamma\colon Q\to X[\Gamma_2, \Fo]$ such that
\begin{enumerate}
    \item For all $(\varphi, x)\in X$, if $\theta(\varphi, x)=(w,i)$, then $(\varphi\cdot w, x\cdot w, i)\in Q$
    
    \item For all $(\varphi, x, i)\in Q$, and all $u \in \supp(\gamma(\varphi, x, i))$, if we let $(v,j) = \mu(\varphi,x,i,u)$, then
        \[\gamma(\varphi, x, i)\cdot u = \gamma(\varphi\cdot v, x\cdot v, j).\]

        \item For all $(\varphi, x, i)\in Q$ and $(w,j)\in (S_1)^*\times \{1,...,N\}$ such that $(\varphi\cdot w, x\cdot w, j)\in Q$, there exists $u\in (S_2)^*$ such that $\mu(\varphi,x,i, u)=(w,j)$.
    \end{enumerate}
\end{defn}

Following the aforementioned notion, we obtain the following result.

\begin{theorem}
\label{thm:QI}
Let $\Gamma_1,\Gamma_2$ be finitely presented and strongly connected blueprints. For each positive integer $N$ and set $\FF$ of forbidden patterns for $\Gamma_2$, there is a set of forbidden patterns $\Fo'=\Fo'(\FF,N)$ which defines a $\Gamma_1$-subshift $\QI(\FF,N)\subset B[\Gamma_1]$ such that the following properties hold:

\begin{enumerate}
    \item $\FF'$ can be effectively constructed from $N$ and $\FF$. Furthermore, $\Fo'$ is finite (resp. effective) if and only if $\FF$ is finite (resp. effective).
    \item If there exist models $\varphi_1\in\M(\Gamma_1)$, $\varphi_2\in\M(\Gamma_2)$ such that $G(\Gamma_2,\varphi_2)$ is quasi-isometric to $G(\Gamma_1, \varphi_1)$, then there exists $N \geq 1$ such that for any $\FF$ such that $X[\Gamma_2,\varphi_2,\Fo]\neq\varnothing$, then there exists $x\in B[\Gamma_1,\varphi_1]$ such that $(\varphi_1,x) \in \QI(\Fo,N)$. 
    \item If $\QI(\Fo,N)$ is non-empty for some $N \geq 1$ and $\Fo$, then it codes $\Fo$ through quasi-isometries for $N$. Additionally, there exists a quotient $\Gamma'_2$ of $\Gamma_2$, and models $\varphi_1\in\M(\Gamma_1)$, $\varphi_2\in\M(\Gamma'_2)\subseteq\M(\Gamma_2)$ such that $G(\Gamma_2',\varphi_2)$ is quasi-isometric to $G(\Gamma_1, \varphi_1)$, and $X[\Gamma_2,\varphi_2,\Fo]\neq\varnothing$.
    \end{enumerate}
\end{theorem}

Let us write $\Gamma_1 = (M_1, S_1, R_1)$, $\Gamma_2 = (M_2, S_2, R_2)$, and $I = \{1,\dots,N\}$. Fix an alphabet $A$ and a set of forbidden patterns $\FF$ for $\Gamma_2$ over $A$. Consider the alphabet 

    \[  \widehat{B} = B_1 \times \dots \times B_N, \]
   
where, $B_i = \{\ast\}\cup \left( M_2 \times A \times \partial P \times \partial I \right)$ for every $ i \in I$, with

     \[ \partial P = \left( (S_1)^{\leq 2N} \cup \{ \diamond \}  \right)^{S_2} \mbox{ and } \partial I = \left(I \cup \{ \diamond \}  \right)^{S_2}.  \]


Denote elements of $\widehat{B}$ by $b = (b_1,\dots,b_N)$. If $b_i \neq \ast$, we write $M_2(b_{i})$, $A(b_{i})$, $\partial P (b_{i})$ and $\partial I(b_{i})$ for its projection to each of its coordinates. The alphabet of $\QI(\Fo,N)$ is the set $B$ of all elements $(b_1,\dots,b_n)\in \widehat{B}$ which satisfy condition \textbf{C0} below.\\
 

\noindent
\textbf{Condition C0:} (state consistency)\\
For every $i \in I$, $s \in S_2$ and $b_i\in B_i$ with $b_i \neq \ast$ we have that:
\[\init(s) = M_2(b_i) \iff \partial P(b_{i})(s)\neq \diamond \iff \partial I(b_{i})(s) \neq \diamond\]

We define the $\Gamma_1$-subshift $\QI(\Fo,N)\subset B[\Gamma_1]$ as the space of configurations $(\varphi,q)$ which satisfy conditions \textbf{C1}--\textbf{C5} given below.\\

\noindent
\textbf{Condition C1:} (density condition) \\
There exists $w,w' \in (S_1)^{\leq N}$ and $i \in I$ such that $w\in \supp(x)$, $x(w)_i \neq \ast$ and $\underline{ww'}_{\Gamma_1}=\underline{\varepsilon}_{\Gamma_1}$.\\

This condition codes the relative density condition of a quasi-isometry, that is, for every $\varphi\in \mathcal{M}(\Gamma_1)$ there exists $w\in (S_1)^{\leq N}$ and an index $i \in I$ which is in the image of a quasi-isometry ($x(w)_i\neq \ast$), and furthermore there is a short path from it to the starting point ($\exists w'\in (S_1)^{\leq N}: \underline{ww'}_{\Gamma_1}=\underline{\varepsilon}_{\Gamma_1}$).\\

\noindent
\textbf{Condition C2:} (partial action)\\
Suppose $\varepsilon\in \supp(x)$ and $i \in I$ is such that $x(\varepsilon)_i \neq \ast$. For every $s \in S_2$ with $\init(s) = M_2(x(\varepsilon)_{i})$ denote $u = \partial P(x(\varepsilon)_{i})(s)$ and $j = \partial I(x(\varepsilon)_{i})(s)$. We impose that $u \in \supp(x)$, $x(u)_{j} \neq \ast$ and $M_2(x(u)_{j}) \in \ter(s)$.\\

With the previous condition, we define a partial action from $S_2^*$ on pairs $(\varphi, x, i)$ where $(\varphi, x)\in B[\Gamma_1]$ satisfies condition \textbf{C2} and $i\in I$ is such that $x(\varepsilon)_i\neq\ast$. For $s\in S_2$ such that $\init(s) = M_2(x(\varepsilon)_i)$, this action is defined as
$$(\varphi, x, i) \circ s = \left(\varphi\cdot \partial P(x(\varepsilon)_{i})(s),\ x\cdot \partial P(x(\varepsilon)_{i})(s),\ \partial I(x(\varepsilon)_{i})(s)\right),$$
By iteration, this induces a partial action of $S_2^*$. For ${\xi}=(\varphi, x,i)$ as above and $u \in S_2^*$ such that $\xi \circ u$ is defined and $(\varphi, x, i)\circ u = (\varphi\cdot w_u, x\cdot w_u, j_u)$, we shall use the notation \begin{equation}\label{eqn:movind}
    (\mov_{\xi}(u),\ind_{\xi}(u)) = (w_u,j_u).
\end{equation}

\noindent
\textbf{Condition C3:} (reachability) \\
Let $\xi=(\varphi,x,i)$ and $v\in (S_1)^{\leq 2N +1}$. If both $\varepsilon,v\in \supp(x)$ and there are $i, j\in I$ with $x(\varepsilon)_{i}\neq\ast$, $x(v)_{j}\neq\ast$, then there exists $w\in (S_2)^{\leq N(3N+1)}$ such that $(\varphi, x, i) \circ w$ is defined and 
\[ (\underline{\mov_{\xi}(w)}_{\Gamma_1},\ind_{\xi}(w)) = (\underline{v}_{\Gamma_1},j). \]

\noindent
\textbf{Condition C4:} (relations)\\
Let $\xi=(\varphi,x,i)$ such that $x(\varepsilon)_i\neq \ast$. For every relation $(u,v) \in R_2$ if both $(\varphi,x,i)\circ u$ and $(\varphi,x,i)\circ v$ are defined, then 
\[\underline{\mov_{\xi}(u)}_{\Gamma_1} = \underline{\mov_{\xi}(v)}_{\Gamma_1} \mbox{ and } \ind_{\xi}(u) = \ind_{\xi}(v).\]

Finally, Let $\xi=(\varphi,x,i)$ such that $x(\varepsilon)_i\neq \ast$. We define $\gamma(\xi) = (\widehat{\varphi}, \widehat{x})\in A[\Gamma_2]$ as
\begin{align}\label{eq:induced_configurations_QI}
\begin{split}
    \widehat{\varphi}(u) & = \begin{cases}
    M_2(x(\mov_{\xi}(u))_{\ind_{\xi}(u)}) & \mbox{if } (\varphi, x, i)\circ u \mbox{ is defined},\\
    \varnothing  & \mbox{otherwise},
\end{cases}\\
\widehat{x}(u) & = \begin{cases}
    A(x(\mov_{\xi}(u))_{\ind_{\xi}(u)}) & \mbox{if } (\varphi, x, i)\circ u \mbox{ is defined},\\
    \varnothing & \mbox{otherwise},
\end{cases}
\end{split}
\end{align}
 for every $u \in S_2^*$. \\

\noindent
\textbf{Condition C5:} (avoidance of forbidden patterns)\\
Let $\xi=(\varphi,x,i)$ such that $x(\varepsilon)_i\neq \ast$. For every $W\Subset (S_2)^*$, if we consider $\gamma(\xi)=(\widehat{\varphi},\widehat{x})$ as above, then \[\gamma(\xi)|_{W} \notin \Fo.\]

These five conditions provide the forbidden patterns $\FF'=\Fo'(\Fo,N)$ that define our subshift $\QI(\Fo,N)$. Next we will define $\QI'(\Fo,N)$ and the maps $\theta,\mu$ and $\gamma$.

\begin{defn}(The set $\QI'(\Fo,N)$)\label{def:set}
    We take $\QI'(\Fo,N)\subset \QI(\Fo,N)\times I$ as the set of all configurations that code a pre-image at the origin on index $i$, that is \[ \QI'(\Fo,N) = \{ (\varphi,x,i) \in \QI(\Fo,N)\times I :\varepsilon \in \supp(x) \mbox{ and } x(\varepsilon)_i \neq \ast \}. \]
\end{defn}

For the rest of this section, we will fix a lexicographical order in $(S_1)^*$.

\begin{defn}(The map $\theta$)
\label{def:theta}
    We let $\theta\colon \QI(\Fo,N) \to (S_1)^{\leq N} \times I$ be given by $\theta(\varphi,x) = (u,i)$ if and only if $u \in (S_1)^{\leq N}$ is the lexicographically smallest word, and $i \in I$ the smallest index such that $u \in \supp(x)$ and $x(u)_i \neq \ast$.
\end{defn}

We remark that $\theta$ is well defined by condition \textbf{C1}.

\begin{defn}(The map $\mu$)
\label{def:mu}
    We let $\mu\colon \QI'(\Fo,N)\times (S_2)^* \to (S_1)^{\leq N}\times I$ be given by \[ \mu(\xi,u) = \begin{cases}
        (\mov_{\xi}(u),\ind_{\xi}(u)) & \mbox{if } \xi \circ u \mbox{ is defined},\\
        (\varepsilon,1) & \mbox{otherwise.}
    \end{cases}  \] Where $(\mov_{\xi}(u),\ind_{\xi}(u))$ are as in~\Cref{eqn:movind}.
    
\end{defn}

\begin{defn}(The map $\gamma$)
\label{def:gamma}
    We let $\gamma\colon \QI'(\Fo,N) \to A[\Gamma_2]$ be given by $\gamma(\varphi,x,i) = (\widehat{\varphi},\widehat{x})$, where $(\widehat{\varphi},\widehat{x})$ are the configurations given in Equation~\eqref{eq:induced_configurations_QI}.
\end{defn}

It is direct from the construction that $\theta$, $\mu$ and $\gamma$ are computable maps.

\begin{lemma}\label{lem:movement_is_well_defined}
    Let ${\xi} = (\varphi,x,i) \in \QI'(\Fo,N)$ and $u,v\in (S_2)^*$ be two $\Gamma_2$-equivalent words. Then, $$\underline{\mov_{\xi}(u)}_{\Gamma_1} = \underline{\mov_{\xi}(v)}_{\Gamma_1} \mbox{ and } \ind_{\xi}(u) = \ind_{\xi}(v).$$
\end{lemma}

\begin{proof}
    We start by taking $u,v \in (S_2)^*$ two $\Gamma_2$-similar words, that is, words such that there exist $w,u',v',z\in(S_2)^*$ such that $u = wu'z$, $v = wv'z$, and $(u', v')\in R_2$. Because $(\varphi, x)\cdot\mov_{\xi}(w)\in\QI(\Fo,N)$ satisfies \textbf{C4}, and $(u',v')\in R_2$, we have that 
    \[\underline{\mov_{{\xi}\circ w}(u')}_{\Gamma_1} = \underline{\mov_{{\xi}\circ w}(v')}_{\Gamma_1} \mbox{ and } \ind_{{\xi}\circ w}(u') = \ind_{{\xi}\circ w}(v').\]
    Notice that $\ind_{{\xi}\circ w}(u') = \ind_{\xi}(wu')$ and $\ind_{{\xi}\circ w}(v') = \ind_{\xi}(wv')$, as well as, $\underline{\mov_{\xi}(wu')}_{\Gamma_1} = \underline{\mov_{\xi}(w)\mov_{{\xi}\circ w}(u')}_{\Gamma_1}$ and $\underline{\mov_{\xi}(wv')}_{\Gamma_1} = \underline{\mov_{\xi}(w)\mov_{{\xi}\circ w}(v')}_{\Gamma_1}$. Therefore,
    \[\underline{\mov_{\xi}(wu')}_{\Gamma_1} = \underline{\mov_{\xi}(wv')}_{\Gamma_1} \mbox{ and } \ind_{\xi}(wu') = \ind_{\xi}(wv').\]
    Using that $\varphi$ is a $\Gamma_1$-model we conclude that
    \[\underline{\mov_{\xi}(wu'z)}_{\Gamma_1} = \underline{\mov_{\xi}(wv'z)}_{\Gamma_1} \mbox{ and } \ind_{\xi}(wu'z) = \ind_{\xi}(wv'z).\]
    Finally, if $u,v \in (S_2)^*$ are $\Gamma_2$-equivalent, there exists a sequence of words $u_1, ..., u_{n-1}$ such that $u_k$ is $\Gamma_2$-similar to $u_{k+1}$ for all $k\in\{0,...,n-1\}$ where $u_0 = u$ and $u_n = v$. The previous argument implies that $\underline{\mov_{\xi}(u_k)}_{\Gamma_1} = \underline{\mov_{\xi}(u_{k+1})}_{\Gamma_1}$ and ${\ind_{\xi}(u_k)} = {\ind_{\xi}(u_{k+1})}$ for all $k\in\{0,...,n-1\}$ and thus $\underline{\mov_{\xi}(u)}_{\Gamma_1} = \underline{\mov_{\xi}(v)}_{\Gamma_1}$ and $\ind_{\xi}(u) = \ind_{\xi}(v)$ as required.
\end{proof}

\begin{lemma}
\label{lem:propiedades_gamma}
    For all $\xi = (\varphi, x, i)\in\QI'(\Fo,N)$ we have that $\gamma(\xi) = (\widehat{\varphi},\widehat{x}) \in X[\Gamma_2,\FF]$. Furthermore, $\supp(\widehat{\varphi}) = \{ u\in (S_2)^* : \xi \circ u \mbox{ is defined}\}$, and for all such $u$,
    $$\gamma(\xi \circ u) = (\widehat{\varphi},\widehat{x})\cdot u.$$
\end{lemma}

\begin{proof}
We begin by showing the identity  $\supp(\widehat{\varphi}) = \{ u\in (S_2)^* : \xi \circ u \mbox{ is defined}\}$. Because we are considering $(\varphi,x,i)\in\QI'(\Fo, N)$, we have $\widehat{\varphi}(\varepsilon) = M_2(x(\varepsilon)_i) \neq \varnothing$. Let $m \geq 0$ and suppose inductively that \[ W_m = \{w \in \supp(\widehat{\varphi}): |w|\leq m\} = \{ u\in (S_2)^{\leq m} : \xi \circ u \mbox{ is defined}\}.  \]

Next, take $w \in W_m$ with $|w|=m$ and note that $\widehat{\varphi}(w) = M_2(x(\mov_{\xi}(w))_{\ind_{\xi}(w)})$. By condition \textbf{C2} on the pair $(\varphi \cdot \mov_{\xi}(w), x \cdot \mov_{\xi}(w))$ it follows that $(\varphi \cdot \mov_{\xi}(w), x \cdot \mov_{\xi}(w), \ind_{\xi}(w))\circ s$ is defined exactly for those $s \in S_2$ such that $\init(s) = \widehat{\varphi}(w)$ and that $$\widehat{\varphi}(ws) = M_2(x(\mov_{\xi}(ws))_{\ind_{\xi}(ws)}) \in \ter(s).$$
On the other hand,  for those $s' \in S_2$ with $\init(s') \neq \widehat{\varphi}(w)$, by definition we have $\widehat{\varphi}(ws') =\varnothing$. This shows that the inductive hypothesis holds for $m+1$ and thus shows the required identity.

From the identity, it follows directly that $\widehat{\varphi}$ is $\Gamma_2$-consistent. Now, take $u,v \in (S_2)^*$ to be $\Gamma_2$-equivalent. By Lemma~\ref{lem:movement_is_well_defined}, $\underline{\mov_{\xi}(u)}_{\Gamma_1} = \underline{\mov_{\xi}(v)}_{\Gamma_1}$ and $\ind_{\xi}(u) = \ind_{\xi}(v)$. Thus,
$$\widehat{\varphi}(u) = M_2(x(\mov_{\xi}(u))_{\ind_{\xi}(u)}) = M_2(x(\mov_{\xi}(v))_{\ind_{\xi}(v)}) = \widehat{\varphi}(v).$$

We conclude that $\widehat{\varphi}$ is a $\Gamma_2$-model.

We now proceed to check conditions (s1) to (s3) from Definition~\ref{def:phi_subshift}, to show $\widehat{x}\in X[\Gamma_2, \widehat{\varphi},\Fo]$. For (s1), notice that for $w\in (S_2)^*$, $\widehat{\varphi}(w)\neq \varnothing$ if and only if $\widehat{x}(w)\neq \varnothing$, therefore $\supp(\widehat{x})=\supp(\widehat{\varphi})$. For condition (s2), take two $\Gamma_2$-equivalent words $u,v\in (S_2)^*$. As was the case for $\widehat{\varphi}$, by Lemma~\ref{lem:movement_is_well_defined}, $\underline{\mov_{\xi}(u)}_{\Gamma_1} = \underline{\mov_{\xi}(v)}_{\Gamma_1}$ and $\ind_{\xi}(u) = \ind_{\xi}(v)$. Thus,
$$\widehat{x}(u) = A(x(\mov_{\xi}(u))_{\ind_{\xi}(u)}) = A(x(\mov_{\xi}(v))_{\ind_{\xi}(v)}) = \widehat{x}(v).$$

Finally, to prove (s3), we need the following. Take $u\in \supp(\widehat{\varphi})$. If we denote $(\widehat{\varphi}',\widehat{x}') = \gamma(\xi\circ u)$, we have
\begin{align*}
    \widehat{\varphi}'(w) &= M_2(x\cdot\mov_{\xi}(u)(\mov_{{\xi}\circ u}(w))_{\ind_{{\xi}\circ u}(w)}),\\
    &= M_2(x(\mov_{\xi}(u)\mov_{{\xi}\circ u}(w))_{\ind_{{\xi}\circ u}(w)}).
\end{align*}

As we saw before, $\underline{\mov_{\xi}(uw)}_{\Gamma_1} = \underline{\mov_{\xi}(u)\mov_{{\xi}\circ u}(w)}_{\Gamma_1}$ and $\ind_{\xi}(uw) = \ind_{{\xi}\circ u}(w)$. Therefore,
\[\widehat{\varphi}'(w) = M_2(x(\mov_{\xi}(uw))_{\ind_{\xi}(uw)}) = \widehat{\varphi}(uw).\]

An analogous procedure shows, $\widehat{x}'(w) = \widehat{x}(uw)$. This shows $\gamma(\xi \circ u) = (\widehat{\varphi},\widehat{x})\cdot u$. With this formula at hand, it follows by  condition \textbf{C5} that for every $u\in \supp(\widehat{\varphi})$ and and $W\Subset (S_2)^*$ we have \[\bigl((\widehat{\varphi},\widehat{x})\cdot u\bigr)|_{W} = (\widehat{\varphi}',\widehat{x}')|_W \notin\Fo.   \]

Thus condition (s3) holds.
\end{proof}

%

\begin{lemma}
\label{lem:ida}
    Suppose there exist models $\varphi_1\in\M(\Gamma_1)$ and $\varphi_2\in\M(\Gamma_1)$, $X[\Gamma_2, \varphi_2, \Fo]$ is non-empty and $G(\Gamma_1, \varphi_1)$ is quasi-isometric to $G(\Gamma_2,\varphi_2)$. Then there exists a positive integer $N$ and $x \in B[\Gamma_1, {\varphi_1}]$ such that $(\varphi_1,x) \in \QI(\Fo,N)$.
\end{lemma}

\begin{proof}
     Let $f\colon G(\Gamma_2,\varphi_2)\to G(\Gamma_1, \varphi_1)$ be a quasi-isometry and let $N$ be a positive integer which both bounds the constants in the quasi-isometry $f$ and the number of pre-images of elements of $G(\Gamma_1, \varphi_1)$ and let $y\in X[\Gamma_2, \varphi_2, \Fo]$. We shall construct $x \in B[\Gamma_1, {\varphi_1}]$ such that $(\varphi_1,x) \in \QI(\Fo,N)$.
    As $f$ is $N$-to-$1$, if we let $I = \{1,\dots,N\}$, it follows (using choice) that there exists an injective map $\widehat{f}\colon G(\Gamma_2,\varphi_2) \to G(\Gamma_1,\varphi_1)\times I$ with the property that for every $h\in G(\Gamma_2, \varphi_2)$ then $\widehat{f}(h) = (f(h),i)$ for some $i \in I$.

    Because, $f$ is a quasi-isometry, for all $u\in\supp(\varphi_2)$ and $s\in S_2$ such that $\init(s) = \varphi_2(u)$ we have that $\rho_1(f(\underline{u}_{\Gamma_2}), f(\underline{us}_{\Gamma_2})) \leq 2N$. Therefore, there exists a word $w(u,s)\in (S_1)^{\leq 2N}$ such that \[\underline{f(\underline{u}_{\Gamma_2})w(u,s)}_{\Gamma_1} = f(\underline{us}_{\Gamma_2}).\]

    Consider $v \in (S_1)^*$. If $v \notin \supp(\varphi_1)$ we set $x(v)=\varnothing$. Otherwise, we have 
    $v\in \supp(\varphi_1)$ and take $i \in I$. If $(v,i)$ is such that $(\underline{v}_{\Gamma_1},i)$ is not in the image of $\widehat{f}$, we set $x(v)_i = \ast$. On the other hand, if $(\underline{v}_{\Gamma_1},i) = \widehat{f}(\underline{u}_{\Gamma_2})$ for some $u \in \supp(\varphi_2)$, we set $M_2(x(v)_i) = \varphi_2(u)$, $A(x(v)_i) = y(u)$, 
    $$\partial P(x(v)_i)(s) = \begin{cases}
        w(u,s) & \mbox{ if } \init(s) = \varphi_2(u) \\
        \diamond & \mbox{ otherwise},
    \end{cases}$$
    and
    $$\partial I(x(v)_i)(s) = \begin{cases}
        j & \mbox{ if } \init(s) = \varphi_2(u) \\
        \diamond & \mbox{ otherwise},
    \end{cases}$$
    where $j \in J$ is the index such that $\widehat{f}(\underline{us}_{\Gamma_2}) = (f(\underline{us}_{\Gamma_2}), j)$. By construction, it is clear that $x(v)\in B$ for every $v\in \supp(\varphi_1)$. Therefore, $(\varphi_1, x)$ satisfies \textbf{C0} and it is clear that $(\varphi_1, x) \in B[\Gamma_1]$. Let us look at the rest of the conditions.\\

 \noindent
 \textbf{C1}: Because $f$ is a quasi-isometry, by the relative density condition, there exists $u\in (S_2)^*$ such that 
 $$\max\{\rho_1(\underline{\varepsilon}_{\Gamma_1}, f(\underline{u}_{\Gamma_2})),\rho_1(f(\underline{u_{\Gamma_2}}),\underline{\varepsilon}_{\Gamma_1})\}\leq N.$$
 Then, there exists $i\in I$ and $w,w'\in (S_1)^{\leq N}$ such that $\widehat{f}(\underline{u}_{\Gamma_2}) = (\underline{w}_{\Gamma_1},i)$ and $\underline{ww'}_{\Gamma_1} = \varepsilon$. This implies $x(w)\in B$ and $x(w)_i\neq\ast$.\\

 \noindent
 \textbf{C2}: Let $v\in \supp(q)$ and $i\in I$ such that $x(v)_i\neq \ast$. By construction, there exists $u\in (S_2)^*$ such that $\widehat{f}(\underline{u}_{\Gamma_2}) = (\underline{v}_{\Gamma_1},i)$. Then, for $s\in S_2$ such that $\init(s) = \varphi_2(u)$ we have that $x(vw(u,s))\in B$, $x(vw(u,s))_j \neq \ast$, and $M_2(x(vw(u,s))_j) = \varphi_2(us)\in\ter(s)$, where $ j =\partial I (x(v)_i)(s)$. Hence condition \textbf{C2} holds. \\

 Take $\xi = (\varphi_1,x,i)$ and recall that now that condition \textbf{C2} holds, we have access to the partial action
 $$\xi \circ s = \left(\varphi_1\cdot \partial P(x(\varepsilon)_{i})(s),\ x\cdot \partial P(x(\varepsilon)_{i})(s),\ \partial I(x(\varepsilon)_{i})(s)\right).$$
 
 It follows from our definition of $q$ that this is defined on all $u \in \supp(x)$ and we can write
 \[ \xi\circ u = (\varphi_1 \cdot \mov_{\xi}(u), x\cdot \mov_{\xi}(u), \ind_{\xi}(u)). \]

 \noindent
 \textbf{C3}: Take $v\in (S_1)^{\leq 2N+1}$ such that there exists $i,j\in I$ with $x(\varepsilon)_i,x(v)_j\neq\ast$. Define $\xi = (\varphi_1,x,i)$. By the definition of $x$, there exist $w, u\in (S_2)^*$ such that $\widehat{f}(\underline{w}_{\Gamma_2}) = (\underline{\varepsilon}_{\Gamma_1},i)$ and $\widehat{f}(\underline{u}_{\Gamma_2})=(\underline{v}_{\Gamma_1},j)$. Then, as $f$ is a quasi-isometry, 
 $$\rho_2(\underline{w}_{\Gamma_2},\underline{u}_{\Gamma_2})\leq N(\rho_1(f(\underline{w}_{\Gamma_2}), f(\underline{u}_{\Gamma_2})) + N) = N(\rho_1(\underline{\varepsilon}_{\Gamma_1}, \underline{v}_{\Gamma_1}) + N) \leq N(3N+1).$$
In other words, there exists $w'\in (S_2)^{\leq N(3N+1)}$ such that $\underline{ww'}_{\Gamma_2}= \underline{u}_{\Gamma_2}$. Finally, a simple computation shows $\underline{\mov_{\xi}(w')}_{\Gamma_1} = \underline{v}_{\Gamma_1}$ and $\ind_{\xi}(w') = j$.\\ 

 
 \noindent
 \textbf{C4}: Let $i \in I$, $\xi=(\varphi_1,x,i)$, a relation $(u,v)\in R_2$ and suppose that both $\xi\circ u$ and $\xi\circ v$ are defined. First, there exists $w\in (S_2)^*$ such that $\widehat{f}(\underline{w}_{\Gamma_2}) = (\underline{\varepsilon}_{\Gamma_1}, i)$. As $(u,v) \in R_2$, we have $\underline{u}_{\Gamma_2} = \underline{v}_{\Gamma_2}$ and therefore, $$(\underline{\mov_{\xi}(u)}_{\Gamma_1},\ind_{\xi}(u)) = \widehat{f}(\underline{wu}_{\Gamma_2}) = \widehat{f}(\underline{wv}_{\Gamma_2}) = (\underline{\mov_{\xi}(v)}_{\Gamma_1},\ind_{\xi}(v)).$$

  \noindent
 \textbf{C5}: Let $\xi = (\varphi_1,x,i)$ for some $i \in I$ and suppose that $x(\varepsilon)_i \neq \ast$. Let $(\widehat{\varphi}, \widehat{x})$ be the configuration defined in Equation~\eqref{eq:induced_configurations_QI}. Take $w\in (S_2)^*$ such that $\widehat{f}(\underline{w}_{\Gamma_2}) = (\underline{\varepsilon}_{\Gamma_1}, i)$. We claim $(\widehat{\varphi},\widehat{x}) = (\varphi_2,y)\cdot w$, from where it will follow that no forbidden pattern from $\FF$ may occur and condition \textbf{C5} is satisfied. Indeed, let $u \in (S_2)^*$ such that $(\varphi_1, x, i)\circ u$ is defined. By the previous arguments, we have $\widehat{f}(\underline{wu}_{\Gamma_2})=(\mov_{\xi}(u),\ind_{\xi}(u))$. Therefore, 
 \begin{align*}
     \widehat{\varphi}(u) & = M_2(x(\mov_{\xi}(u))_{\ind_{\xi}(u)}) = \varphi_2(wu),\\
     \widehat{x}(u) & = A(x(\mov_{\xi}(u))_{\ind_{\xi}(u)}) = y(wu).
 \end{align*}
As $(\varphi_1,x)$ satisfies conditions \textbf{C1-C5}, we conclude that $(\varphi_1,x) \in \QI(\Fo,N)$.
\end{proof}

\begin{lemma} (All positions can be reached)
\label{lem:interpolacion}
    Let ${\xi} = (\varphi,x,i)\in \QI'(\Fo,N)$. Let $v \in \supp(x)$ and $j \in I$ such that $x(v)_j \neq \ast$. There exists $u \in (S_2)^*$ such that \[ \underline{\mov_{\xi}(u)}_{\Gamma_1} = \underline{v}_{\Gamma_1} \mbox{ and } \ind_{\xi}(u) = j.  \]
Furthermore, we have $|u| \leq N(3N+1)(|v|+1)$.
\end{lemma}

\begin{proof}
    Suppose first that $v=\varepsilon$. In this case condition \textbf{C3} ensures that there exists $u \in (S_2)^{\leq N(3N+1)}$ such that $\underline{\mov_{{\xi}}(u)}_{\Gamma_1}= \underline{\varepsilon}_{\Gamma_1}$ and $\ind_{\xi}(u) = j$, thus the statement holds.

    Now suppose $v \neq \varepsilon$, let $n = |v| >0$ and consider a sequence $v_0,v_1,\dots,v_n$ such that $v_0 = \varepsilon$, $v_n = v$ and for every $k \in \{0,\dots,n-1\}$ we have $v_{k+1} = v_{k}s_k$ for some $s_k \in S_1$. By condition \textbf{C1}, for each $v_k$ we can find $w_k,w'_k\in (S_1)^{\leq N}$ and $j_k$ such that $q(v_kw_k)_{j_k}\neq \ast$ and $\underline{v_kw_kw'_k}_{\Gamma_1}=\underline{v_k}_{\Gamma_1}$. 

    Let $t_k = v_kw_k$. If we let $(\varphi^k,x^k)=(\varphi,x)\cdot t_k$ then as $|w'_ks_kw_{k+1}| \leq 2N+1$ and $\underline{t_{k+1}}_{\Gamma_1}= \underline{t_kw'_ks_kw_{k+1}}_{\Gamma_1}$, it follows by condition \textbf{C3} that there exists $u_k \in (S_2)^{\leq N(3N+1)}$ such that, if denote ${\xi}_k = (\varphi^k,x^k, j_k)$, then $\underline{\mov_{{\xi}_k}(u_k)}_{\Gamma_1} = \underline{w'_ks_kw_{k+1}}_{\Gamma_1}$ and $\ind_{{\xi}_k}(u_k) = j_{k+1}$. 
    
    Let $u = u_0\dots u_{n-1}$. We have that $|u| \leq N(3N+1)|v|$ and composing everything, we obtain \[\underline{\mov_{\xi}(u)}_{\Gamma_1} = \underline{\mov_{\xi}(u_0)\mov_{{\xi}_1}(u_1) \dots \mov_{{\xi}_{n-1}}(u_{n-1})}_{\Gamma_1} = \underline{v}_{\Gamma_1}, \mbox{ and } \ind_{\xi}(u) = j.    \]
    As $\min\{ N(3N+1)|v|, N(3N+1)\} \leq  N(3N+1)(|v|+1)$, we obtain the required bound.
\end{proof}

\begin{lemma}
\label{lem:vuelta}
    If there exists $N\geq 1$ such that $\QI(\Fo,N)$ is nonempty, then there there exist models $\varphi_1\in\M(\Gamma_1)$, $\varphi_2\in\M(\Gamma_2)$, and a quotient $\Gamma_2'$ of $\Gamma_2$ such that $X[\Gamma_2, \varphi_2,\Fo]$ is non-empty and $G(\Gamma_1, \varphi_1)$ is quasi-isometric to $G(\Gamma'_2,\varphi_2)$.
\end{lemma}

\begin{proof}

Consider $N\geq 1$ such that $\QI(\Fo,N)$ is nonempty, and let $(\varphi,x)\in \QI(\Fo,N)$. By condition \textbf{C1} if we let $(u_0,i) = \theta(\varphi,x)$, then $x(u_0)_i \neq \ast$. Without loss of generality we replace $(\varphi, x)$ by $(\varphi\cdot u_0, x\cdot u_0)$ to have $x(\varepsilon)_i\neq\ast$, and thus ${\xi}=(\varphi,x,i) \in \QI'(\Fo,N)$. Let $(\widehat{\varphi},\widehat{x})$ as in Equation~\eqref{eq:induced_configurations_QI}. By Lemma~\ref{lem:propiedades_gamma} we have that $(\widehat{\varphi},\widehat{x}) \in X[\Gamma_2, \Fo]$ and thus $\widehat{x}\in X[\Gamma_2, \widehat{\varphi},\Fo]$. \\

It remains to show that there exists a quotient $\Gamma_2'$ of $\Gamma_2$ such that $G(\Gamma_2',\widehat{\varphi})$ is quasi-isometric to $G(\Gamma_1,\varphi)$. Consider the set of relations $R = R_2 \cup R'$ over the generators $S_2$ given by 
$$(u,v)\in R' \ \iff \ \underline{\mov_{\xi}(u)}_{\Gamma_1} = \underline{\mov_{\xi}(v)}_{\Gamma_1} \ \text{ and } \ind_{\xi}(u)=\ind_{\xi}(v),$$
for $u,v\in (S_2)^*$. We define the blueprint $\Gamma_2' = (M_2, S_2, R)$ which by definition is a quotient of $\Gamma_2$. We denote the quasi-metric on $\Gamma'_2$ by $\rho_2'$.

By Lemma~\ref{lem:movement_is_well_defined}, we have that if $u,v\in \supp(\widehat{\varphi})$ are words such that $\underline{u}_{\Gamma_2} = \underline{v}_{\Gamma_2}$, then $(u,v)\in R'$. By Lemma~\ref{lem:propiedades_gamma} and Definition~\ref{def:gamma} we have that $\hat{\varphi}$ is $\Gamma_2'$, so not only $\widehat{\varphi} \in \mathcal{M}(\Gamma_2)$, but  $\widehat{\varphi} \in \mathcal{M}(\Gamma_2')$. \\

Consider the map  $f\colon G(\Gamma'_2,\widehat{\varphi}) \to G(\Gamma_1,\varphi)$ given by $f(\underline{u}_{\Gamma'_2}) = \underline{\mov_{\xi}(u)}_{\Gamma_1}$. Notice that $f$ is well defined by the definition of $\Gamma_2'$. We will show that $f$ is a quasi-isometry.\\

\noindent
\textbf{$f$ is a quasi-isometric embedding:} \\
Take $u\in\supp(\widehat{\varphi})$ and $s\in S_2$ such that $i(s) = \widehat{\varphi}(u)$. We have 
\begin{align*}
    \rho_1(f(\underline{u}_{\Gamma'_2}), f(\underline{us}_{\Gamma'_2})) &= \rho_1\!\left(\underline{\mov_{\xi}(u)}_{\Gamma_1}, \underline{\mov_{\xi}(us)}_{\Gamma_1}\right),\\
    &\leq\rho_1\!\left(\underline{\mov_{\xi}(u)}_{\Gamma_1}, \underline{\mov_{\xi}(u)\mov_{{\xi}\circ u}(s)}_{\Gamma_1}\right).
\end{align*}
Because $\mov_{{\xi}\circ u}(s)\in (S_1)^{\leq 2N}$, we get
$$\rho_1(f(\underline{u}_{\Gamma'_2}), f(\underline{us}_{\Gamma'_2}))\leq 2N.$$
Now, for $w = s_1 ... s_n\in (S_2)^*$ such that $uw\in\supp(\widehat{\varphi})$ and $w$ is a geodesic connecting $\underline{u}_{\Gamma'_2}$ and $\underline{uw}_{\Gamma'_2}$ in $\Gamma_2'$:
\begin{align*}
    \rho_1(f(\underline{u}_{\Gamma'_2}), f(\underline{uw}_{\Gamma'_2})) &\leq \sum_{i=1}^{n-1}\rho_1\!\left(f(\underline{us_1 ... s_i}_{\Gamma'_2}), f(\underline{us_1 ... s_{i+1}}_{\Gamma'_2})\right)\\
    &\leq 2N\rho_2'(\underline{u}_{\Gamma'_2},\underline{uw}_{\Gamma'_2}).
\end{align*}

For the second inequality of the quasi-isometric embedding, take $u, v\in \supp(\widehat{\varphi})$. Denote $j = \ind_{\xi}(u)$ and $k = \ind_{\xi}(v)$. Because, $G(\Gamma_1, \varphi)$ is strongly connected, there exists $w\in (S_1)^*$ such that \[\underline{f(\underline{u}_{\Gamma'_2})w}_{\Gamma_1} = f(\underline{v}_{\Gamma'_2})\]and $|w| = \rho_1(f(\underline{u}_{\Gamma'_2}), f(\underline{v}_{\Gamma'_2}))$. Then, if we denote $(\varphi', x', j) = (\varphi, x, i)\circ u$, we have that $x'(\varepsilon)_j\neq \ast$ and $x'(w)_k \neq\ast$. By Lemma~\ref{lem:interpolacion}, there exists $t\in (S_2)^*$ such that $\underline{\mov_{{\xi}\circ u}(t)}_{\Gamma_1} = \underline{w}_{\Gamma_1}$, $\ind_{{\xi}\circ u}(t) = k$ and $|t|\leq N(3N+1)(|w|+1)$. Furthermore,
we have that $\underline{\mov_{\xi}(ut)}_{\Gamma_1} = \underline{\mov_{\xi}(u)\mov_{{\xi}\circ u}(t)}_{\Gamma_1}$, and $\ind_{\xi}(ut)=\ind_{{\xi}\circ u}(t)$. This implies, $\underline{\mov_{\xi}(ut)}_{\Gamma_1} = \underline{\mov_{\xi}(v)}_{\Gamma_1}$ and $\ind_{\xi}(ut)=\ind_{\xi}(v)$, meaning $(ut,v)\in R$. Thus,
$$\rho_2'(\underline{u}_{\Gamma'_2},\underline{v}_{\Gamma'_2})\leq |t| \leq N(3N+1)(\rho_1(f(\underline{u}_{\Gamma'_2}), f(\underline{v}_{\Gamma'_2}))+1).$$

\noindent
\textbf{$f$ is relatively dense:} 

Let $v \in \supp(\varphi)$. By condition \textbf{C1} there are $w,w'\in (S_1)^{\leq N}$ and $j \in I$ such that $x(vw)_j \neq \ast$ and $\underline{ww'}_{\Gamma_1}=\underline{\varepsilon}_{\Gamma_1}$. By Lemma~\ref{lem:interpolacion}
there exists $u \in (S_2)^*$ such that $\underline{vw}_{\Gamma_1} = \underline{\mov_{\xi}(u)}_{\Gamma_1}$ and $j = \ind_{\xi}(u)$, thus $f(\underline{u}_{\Gamma'_2}) = \underline{vw}_{\Gamma_1}$. It follows that \[ \rho_1(\underline{v}_{\Gamma_1},f(\underline{u}_{\Gamma'_2})) \leq |w| \leq N \mbox{ and } \rho_1(f(\underline{u}_{\Gamma'_2}),\underline{v}_{\Gamma_1}) \leq |w'| \leq N.  \]
Therefore $f$ is relatively dense with constant $N$.
\end{proof}

We finally move on to the proof of the main result.

\begin{proof}[Proof of Theorem~\ref{thm:QI}]
We first show a set of forbidden patterns $\FF'$ for $\QI(\Fo,N)$ can be effectively constructed from a description of $\Fo$. Notice that to encode conditions \textbf{C1}-\textbf{C4} only finitely many patterns are required (for \textbf{C4} we use the fact that $\Gamma_2$ is finitely presented) and that the only place where $\Fo$ intervenes is in condition \textbf{C5}. We say a finite sequence of pairs $\beta = ((b_k, i_k)\in B\times I)_{k=0}^{\ell}$ is consistent with a 3-tuple $(w,m,a)\in (S_2)^*\times M_2 \times A$ with $w = s_1 \dots s_{{\ell}}$ if the following hold
\begin{enumerate}
    \item For all $k \in \{0,\dots,{\ell}\}$, $(b_k)_{i_k}\neq \ast$,
    \item For all $k \in \{0,\dots,{\ell}-1\}$ we have \[\init(s_{k+1}) = M_2((b_{k})_{i_{k}}), \quad i_{k+1} = \partial I((b_k)_{i_k})(s_k), \mbox{ and }M_2((b_{k+1})_{i_{k+1}}))\in\ter(s_{k+1}).\]
    \item $M_2(b_{\ell}) = m$ and $A(b_{\ell})=a.$
\end{enumerate}
When $\beta$ is consistent with $(w,m,a)$, we define $W_0(\beta)=\varepsilon$ and for $k \in \{1,\dots,{\ell}\}$,
\[W_k(\beta) = W_{k-1}(\beta)\cdot \partial P((b_{k-1})_{i_{k-1}})(s_{k})\in (S_1)^{\leq kN}.\]

For a consistent $\beta$ we let $T_{\beta}$ denote the set of all maps $T \colon S_1^* \to (M_1\times B)\cup \varnothing$ with support contained in $(S_1)^{\leq |w|N}$ with the property that \[ T(W_k(\beta)) = (m_k,b_k) \mbox{ for all } k \in \{0,\dots,|w|\} \mbox{ for some } m_k \in M_1   \]

Denote by $\mathcal{B}_{(w,m,a)}$ the set of all sequences of pairs $\beta$ consistent with $(w,m,a)$ and let \[T_{(w,m,a)} = \bigcup_{\beta \in \mathcal{B}_{(w,m,a)}} T_{\beta}.\]

Note that $T_{(w,m,a)}$ represents all patterns for which the action $\circ\ w$ is locally defined and which carry the pair $(m,a)$ after following $w$. Finally, for a pattern $p\in\Fo$ we consider the set of patterns \[  T_p = \bigcap_{w \in \supp(p)}T_{(w,p(w))}.\]

It follows that $T_p$ is computable from a description of $p$ and that it encodes condition \textbf{C5} for a fixed pattern $p$. From here, it follows that the set of forbidden patterns for \textbf{C5}, that is, $T = \bigcup_{p \in \FF}T_p$, can be effectively constructed from $\Fo$. Therefore we may take $\FF'$ as the union of $T$ with the finitely many forbidden patterns that encode \textbf{C1}-\textbf{C4}. It follows that $\FF'$ can be effectively computed from $\FF$ and $N$. In particular, if $\FF$ is effective, then so is $\FF'$.

Finally, suppose $\Fo$ is finite. By the analysis above for every $p \in \Fo$ we have that \[|T_p|\leq  (|B||M_1|+1)^{N\max\{|w|: w \in \supp(p)\}},\] from where it follows that $T$ is a finite set of patterns. As the set of patterns required to implement \textbf{C1-C4} is finite, we deduce that $\FF'$ is finite.

The second point of the statement is given by Lemma~\ref{lem:ida}.

For the third point, suppose $\QI(\Fo,N)\neq \varnothing$. To see $\QI(\Fo,N)$ codes $\Fo$ through quasi-isometry for $N\geq 1$, we set $Q = \QI'(\Fo, N)$ from~\Cref{def:set}, and the maps $\theta$, $\mu$ and $\gamma$ as in Definitions~\ref{def:theta},~\ref{def:mu} and~\ref{def:gamma} respectively. Then, point 1 of Definition~\ref{def:coding} comes from the definition of $\theta$, point 2 from Lemma~\ref{lem:propiedades_gamma}, and point 3 from Lemma~\ref{lem:interpolacion}. The computability of the three maps is direct from the definition.

Finally, the last part of the statement is proven in Lemma~\ref{lem:vuelta}. 
\end{proof}

\section{QI-rigidity}
\label{sec:qi_rigidity}

\subsection{Domino problems for blueprints and models}

We extend the definition of the classical domino problem to blueprints and models.

\begin{defn}
 Let $\Gamma$ be a blueprint. We say the $\Gamma$-\define{domino problem} is decidable, if there exists an algorithm which given a description of an alphabet $A$ and a finite set of nearest neighbor forbidden patterns $\Fo$ for $\Gamma$ over $A$, decides whether the $\Gamma$-SFT $X[\Gamma,\Fo]$ is non-empty.
\end{defn}

\begin{defn}
 Let $\Gamma$ be a blueprint and $\varphi \in \M(\Gamma)$. We say the $\varphi$-\define{domino problem} is decidable, if there exists an algorithm which given a description of an alphabet $A$ and a finite set of nearest neighbor forbidden patterns $\Fo$ for $\Gamma$ over $A$, decides whether the $\varphi$-SFT $X[\Gamma,\varphi,\Fo]$ is non-empty.
\end{defn}

Our first result is a generalization of Cohen's theorem to finitely presented blueprints.
\begin{theorem}
\label{thm:domino}
        Let $\Gamma_1$, $\Gamma_2$ be two finitely presented strongly connected blueprints that are quasi-isometric. Then, the $\Gamma_1$-domino problem is many-one equivalent to the $\Gamma_2$-domino problem.
\end{theorem}

\begin{proof}
    As $\Gamma_1$ and $\Gamma_2$ are quasi-isometric, by Theorem~\ref{thm:QI} there exists a positive integer $N$ such that, for any set of forbidden patterns $\Fo$ for $\Gamma_2$, the $\Gamma_1$-SFT $\QI(\Fo, N)$ is non-empty if and only if $X[\Gamma_2, \Fo]$ is non-empty. Furthermore, a set of defining forbidden patterns for $\QI(\Fo, N)$ can be effectively constructed from $N,\Fo$. This shows that the domino problem on $\Gamma_2$ many-one reduces to the domino problem on $\Gamma_1$. By symmetry, it follows that they are many-one equivalent.
\end{proof}

Notice that~\Cref{thm:domino} applies in the situation where $\Gamma_1$ and $\Gamma_2$ are blueprints that represent two finitely presented groups whose Cayley graphs are quasi-isometric. In particular, our result generalizes Cohen's result that states that the domino problem is a quasi-isometry invariant for finitely presented groups with decidable word problem~\cite[Theorem 1.8]{cohen2017large}.

Asking for all model graphs to be quasi-isometric is quite a strong condition. To get rid of it, we must ask something extra out of the dynamics of the space of models.

\begin{lemma}
\label{lem:minimal}
        Consider a blueprint $\Gamma = (M,S,R)$, and $\Fo$ a set of forbidden patterns for $\Gamma$. If there exists $\varphi\in\M(\Gamma)$ such that $X[\Gamma,\varphi,\Fo]\neq\varnothing$, then for all models $\varphi'\in\overline{\orb(\varphi)}$ we have $X[\Gamma,\varphi',\Fo]\neq\varnothing$.
\end{lemma}

\begin{proof}
    Let  $x\in X[\Gamma,\varphi,\Fo]$. By definition, it is clear that for every $w \in \supp(x)$ we have $x\cdot w \in X[\Gamma,\varphi \cdot w,\Fo]$.
    
    By assumption, given $\varphi'\in\overline{\orb(\varphi)}$ there exists a sequence $(w_n)_{n \geq 0}$ of elements in $\supp(\varphi)$ such that $\varphi \cdot w_n$ converges to $\varphi'$. Let $x'$ be any limit point of the sequence $(x\cdot w_n)_{n \geq 0}$, thus, passing through a subsequence, we can assume without loss of generality that $(\varphi \cdot w_n,x\cdot w_n)$ converges to $(\varphi',x')$.

    We claim that $x'\in X[\Gamma,\varphi',\FF]$. Indeed, as we are taking the prodiscrete topology, it follows that for each $u\in S^*$ there exists $N(u)\in\N$ such that for all $n\geq N(u)$,
    $$\varphi(w_nu) = (\varphi \cdot w_n)(u) = \varphi'(u) \mbox{ and } x(w_nu) = (x\cdot w_n)(u) = x'(u)$$

    From here it follows easily that conditions (s1), (s2) and (s3) are satisfied by $x'$ and thus $x'\in X[\Gamma,\varphi',\FF]$.\end{proof}

    \begin{corollary}\label{coro:minimalongo}
        Let $\Gamma$ be a minimal blueprint and $\Fo$ a set of forbidden patterns. $X[\Gamma,\Fo]\neq\varnothing$ if and only if for every model $\varphi$ we have $X[\Gamma,\varphi,\Fo]\neq\varnothing$. In particular, the $\varphi$-domino problem and the $\Gamma$-domino problem are many-one equivalent.
    \end{corollary}

    \begin{theorem}\label{thm:domino_model}
        Let $\Gamma_1, \Gamma_2$ be two finitely presented strongly connected blueprints and let $\varphi_2 \in \M(\Gamma_2)$. Suppose that $\Gamma_1$ is minimal and that there exists $\varphi_1 \in \M(\Gamma_1)$ such that $G(\Gamma_1, \varphi_1)$ is quasi-isometric to $G(\Gamma_2,\varphi_2)$. Then the $\Gamma_1$-domino problem many-one reduces to the $\varphi_2$-domino problem.
    \end{theorem}

    \begin{proof}
        From Theorem~\ref{thm:QI} it follows that for any $N \geq 1$ and any finite set of forbidden patterns $\FF$ for $\Gamma_1$ we can compute from $\FF$ and $N$ a finite set of forbidden patterns $\FF'=\FF'(\FF,N)$ for the $\Gamma_2$-subshift $\QI(\Fo,N)$.

        Furthermore, as $G(\Gamma_1, \varphi_1)$ is quasi-isometric to $G(\Gamma_2,\varphi_2)$, there is a fixed $\widehat{N}\geq 1$ such that for any such $\FF$, if $X[\Gamma_1,\varphi_1,\FF]\neq\varnothing$ then there is $x\in B[\Gamma_2,\varphi_2]$ with $(\varphi_2,x)\in \QI(\Fo,\widehat{N})$.

        Fix $\widehat{N}$ as above, and consider the computable map that transforms $\FF$ into $\FF'=\FF'(\FF,\widehat{N})$ that defines $\QI(\Fo,\widehat{N})$. We will show that this is the required many-one reduction.

        Suppose $X[\Gamma_1,\FF]\neq\varnothing$, by minimality of $\Gamma_1$, it follows that $X[\Gamma_1,\varphi_1,\FF]\neq\varnothing$ and thus $(\varphi_2,x)\in X[\Gamma_2,\varphi_2,\FF']\neq \varnothing$.

        Conversely, if $X[\Gamma_2,\varphi_2,\FF']\neq \varnothing$, then $\QI(\Fo,\widehat{N})\neq \varnothing$ and it follows by the last part of Theorem~\ref{thm:QI} that there is a model $\varphi^* \in \M(\Gamma_1)$ such that $X[\Gamma_1,\varphi^*,\FF]\neq \varnothing$ and thus that $X[\Gamma_1,\FF]\neq\varnothing$. \end{proof}

We note that in the previous proof, by minimality of $\Gamma_1$ we can replace the $\Gamma_1$-domino problem by the $\varphi_1$-domino problem, thus we also obtain that the $\varphi_1$-domino problem many-one reduces to the $\varphi_2$-domino problem. However, without the assumption of minimality on $\Gamma_1$, we do not get this reduction because $X[\Gamma_2,\varphi_2,\FF']\neq \varnothing$ just implies that $X[\Gamma_1,\FF]\neq \varnothing$, but it could be that $X[\Gamma_1,\varphi_1,\FF]$ is empty.

\begin{corollary}
    Let $\Gamma_1, \Gamma_2$ be two minimal finitely presented strongly connected blueprints. If there exist $\varphi_1\in\M(\Gamma_1)$ and $\varphi_2\in\M(\Gamma_2)$ such that $G(\Gamma_1, \varphi_1)$ is quasi-isometric to $G(\Gamma_2,\varphi_2)$, then the $\Gamma_1$-domino problem is decidable if and only if the $\Gamma_2$-domino problem is decidable.
\end{corollary}

\subsection{Strong aperiodicity}

Our second application concerns strongly aperiodic SFTs.

\begin{defn}
    We say a $\Gamma$-subshift $X$ is free (or strongly aperiodic) if for any configuration $(\varphi,x)\in X$, $(\varphi, x)\cdot w = (\varphi,x)$ implies $\underline{w}_{\Gamma} = \underline{\varepsilon}_{\Gamma}$ for all $w\in\supp(\varphi)$.
\end{defn}

This generalizes the notion of strong aperiodicity from subshifts on groups. The next result is a generalization of a theorem of Cohen~\cite[Theorem 1.9]{cohen2017large} which states that admitting a strongly aperiodic SFT is an invariant of quasi-isometry for finitely presented groups. We note that the arXiv preprint of~\cite{cohen2017large} requires the additional hypothesis of torsion-freeness, but this is not really needed and is gone in the published version.

\begin{theorem}
\label{thm:SA}
        Let $\Gamma_1$, $\Gamma_2$ be two finitely presented strongly connected blueprints that are quasi-isometric. $\Gamma_1$ admits a strongly aperiodic SFT if and only if $\Gamma_2$ admits a strongly aperiodic SFT.
\end{theorem}

\begin{proof}
Suppose $\Gamma_2$ admits a non-empty strongly aperiodic SFT given by the finite set of forbidden patterns $\Fo$. As $\Gamma_1$ and $\Gamma_2$ are quasi-isometric, by Theorem~\ref{thm:QI} there exists $N\geq 1$ such that the $\Gamma_1$-SFT $\QI(\Fo, N)$ is non-empty. Take $(\varphi,q)\in\QI(\Fo, N)$ and $w\in (S_1)^*$ such that $(\varphi,q)\cdot w = (\varphi, q)$.

Take the maps $\theta, \gamma$ and $\mu$ given by the fact that $\QI(\Fo, N)$ codes $\Fo$ through quasi-isometries for $N$ (Definition~\ref{def:coding}). Denote $\theta(\varphi, q) = (u_0, i)$. Because $G(\Gamma_1, \varphi)$ is strongly connected, there exists $v\in (S_1)^*$ such that $\underline{u_0v}_{\Gamma_1} = \underline{\varepsilon}_{\Gamma_1}$. If we denote $(\varphi',q') = (\varphi, q)\cdot u_0$ and $w'= vwu_0$, we have
\[  (\varphi', q')\cdot w' = (\varphi, q) \cdot u_0vwu_0 = (\varphi,q)\cdot wu_0 = (\varphi,q)\cdot u_0 = (\varphi',q').\]

In particular, $(\varphi'\cdot w',q'\cdot w',i) \in Q$. By point (3) of Definition~\ref{def:coding}, there exists a word $u\in(S_2)^*$ such that $(w',i) = \mu(\varphi',q',i,u)$. It follows that if we denote $(\widehat{\varphi},\widehat{x}) = \gamma(\varphi',q',i)\in X[\Gamma_2,\Fo]$, then $(\widehat{\varphi},\widehat{x})\cdot u = (\widehat{\varphi},\widehat{x})$. Because $X[\Gamma_2,\Fo]$ is a strongly aperiodic SFT, $\underline{u}_{\Gamma_2} = \underline{\varepsilon}_{\Gamma_2}$. This implies $\underline{w'}_{\Gamma_1} = \underline{\varepsilon}_{\Gamma_1}$, which in turn implies $\underline{w}_{\Gamma_1} = \underline{\varepsilon}_{\Gamma_1}$. Therefore, $\QI(\Fo,N)$ is a strongly aperiodic $\Gamma_1$-SFT. By exchanging the roles of $\Gamma_2$ and $\Gamma_1$ in the previous argument, we conclude the equivalence.
\end{proof}

As was the case with the domino problem, we can weaken the quasi-isometry hypothesis provided the blueprints are minimal.

\begin{theorem}\label{thm:SA_v2}
    Let $\Gamma_1, \Gamma_2$ be two minimal finitely presented strongly connected blueprints. If there exist $\varphi_1\in\M(\Gamma_1)$ and $\varphi_2\in\M(\Gamma_2)$ such that $G(\Gamma_1, \varphi_1)$ is quasi-isometric to $G(\Gamma_2,\varphi_2)$, then the $\Gamma_1$ admits a strongly aperiodic SFT if and only if $\Gamma_2$ admits a strongly aperiodic SFT.
\end{theorem}

\begin{proof}
    Suppose $\Gamma_2$ admits a non-empty strongly aperiodic SFT given by the forbidden patterns $\Fo$. In particular, there exists a model $\varphi\in\M(\Gamma_2)$ such that $X[\Gamma_2,\varphi,\Fo]\neq\varnothing$. By Lemma~\ref{lem:minimal}, we also have $X[\Gamma_2,\varphi_2,\Fo]\neq\varnothing$. Now, as $G(\Gamma_1, \varphi_1)$ is quasi-isometric to $G(\Gamma_2,\varphi_2)$, we know from Theorem~\ref{thm:QI} that there exists $N\geq 1$ such that $\Gamma_1$-SFT $\QI(\Fo, N)$ is non-empty. Consider $(\varphi, x)\in\QI(\Fo,N)$, and $w\in (S_1)^*$ such that $(\varphi, x)\cdot w = (\varphi, x)$. Also, consider the maps $\theta, \gamma$ and $\mu$ given by the fact that $\QI(\Fo, N)$ codes $\Fo$ through quasi-isometries for $N$. As we saw in the previous proof, we can suppose without loss of generality that there exists $i\in\{1,\dots,N\}$ such that $(\varphi, x, i), (\varphi\cdot w, x\cdot w, i)\in Q$. By point (3) of Definition~\ref{def:coding}, there exists $u\in(S_2)^*$ such that $(w,i) = \mu(\varphi,x,i,u)$. If we denote $(\widehat{\varphi},\widehat{x}) = \gamma(\varphi,q,i)\in X[\Gamma_2,\Fo]$, then $(\widehat{\varphi},\widehat{x})\cdot u = (\widehat{\varphi},\widehat{x})$. Because $X[\Gamma_2,\Fo]$ is a strongly aperiodic SFT, we have that $\underline{u}_{\Gamma_2} = \underline{\varepsilon}_{\Gamma_2}$. This implies $\underline{w}_{\Gamma_1} = \underline{\varepsilon}_{\Gamma_1}$. Therefore, $\QI(\Fo,N)$ is a strongly aperiodic $\Gamma_1$-SFT. By exchanging the roles of $\Gamma_2$ and $\Gamma_1$ in the previous argument, we conclude the equivalence.
\end{proof}

\subsection{Medvedev degrees}

Our third application concerns Medvedev degrees of subshifts defined over blueprints. In this section we only provide a very brief and functional introduction to Medvedev degrees. A more thorough presentation of the subject can be found in~\cite{barbieri2024medvedev}.

A map $f\colon X\subset \{0,1\}^{\NN}\to \{0,1\}^{\NN}$ is \define{computable} if there exists an algorithm which, given as oracle $x\in X$ (that is, such that for any $n \in \NN$ it has access to the value $x(n)$) and on input $k \in \NN$, outputs $f(x)(k)$ in finite time. Consider two sets $X,Y\subset \{0,1\}^{\NN}$. We say that $Y$ is \define{Medvedev reducible to} $X$ and write $Y \preceq_{\mathfrak{m}} X$ if there exists a computable map $\psi \colon X \to \{0,1\}^{\NN}$ with the property that $\psi(X)\subset Y$. If both $X\preceq_{\mathfrak{m}} Y$ and $Y\preceq_{\mathfrak{m}} X$ are verified, we say $X$ and $Y$ are \define{Medvedev equivalent}. We denote by $\mathfrak{m}(X)$ the equivalence class of sets which are Medvedev equivalent to $X$ and call it the \define{Medvedev degree} of $X$. The collection $\mathfrak{M}$ of Medvedev degrees is a distributive lattice with the order $\preceq_{\mathfrak{m}}$. The minimum of this lattice is denoted by $0_{\mathfrak{M}}$, and consists on all sets that contain a computable point.\\

Intuitively, if one thinks of $X,Y\subset \{0,1\}^{\NN}$ as sets of solutions to some ``problems'' $P_X,P_Y$, the fact that $Y \preceq_{\mathfrak{m}} X$ means that using as a black box a solution of $P_X$ we can compute a solution of $P_Y$. The intuitive reason for $0_{\mathfrak{M}}$ being the minimum in this viewpoint, is that in this case we may always ignore the input of the problem and output a computable point, that is, we may take $\psi$ as the constant map whose unique value is a computable point in $Y$. \\

The definition of Medvedev degrees is naturally extended to spaces which are recursively homeomorphic to $\{0,1\}^{\NN}$ with the canonical computable structure (see~\cite{barbieri2024medvedev}), and thus one can speak about Medvedev reduction and Medvedev degrees of subsets of $A^{S^*}$ for finite sets $A,S$. 

\begin{theorem}\label{thm:medvedev_noassumpt}
    Let $\Gamma_1$, $\Gamma_2$ be two finitely presented strongly connected blueprints that are quasi-isometric. For every $\Gamma_2$-SFT $X$ there exists a $\Gamma_1$-SFT $Y$ such that $X \prec_{\mathfrak{m}} Y$. In particular, if $\mathfrak{m}(X) \neq 0_{\mathfrak{M}}$, then $\mathfrak{m}(Y) \neq 0_{\mathfrak{M}}$.
    \end{theorem}

    \begin{proof}
        Write $X=X[\Gamma_2,\FF]$. By~\Cref{thm:QI} there is $N\in \NN$ such that $Y = \QI(\FF,N)$ codes $\FF$ through quasi-isometries for $N$. Furthermore, a set of forbidden patterns that describes $Y$ can be computed from a description of $\FF$. Let $\theta$ and $\gamma$ as in~\Cref{def:coding} and remark that they are computable maps. Given $(\varphi_1,y)\in Y$, first we get $\theta(\varphi_1,y) = (w,i)$ and define $(\varphi_2,x) = \gamma(\varphi \cdot w,x\cdot w,i) \in X[\Gamma_2,\FF]$. As this is the composition of two computable maps, we deduce that $X \preceq_{\mathfrak{m}} Y$.
    \end{proof}
    
    In particular, this theorem implies that for finitely presented blueprints, admitting an SFT with no computable points is an invariant of quasi-isometry.\\

    Next we shall introduce a condition which ensures that we get equality in~\Cref{thm:medvedev_noassumpt}. This will require both a notion of decidable word problem for blueprints, and the existence of a computable map that takes a model from one blueprint, and outputs both the data of a quasi-isometric model of the other blueprint, and of the quasi-isometry itself.

    \begin{defn}\label{def:WP_blueprints}
        Let $\Gamma=(M,S,R)$ be a finitely generated blueprint. We say that $\Gamma$ has decidable word problem, if there's an algorithm that takes two $\Gamma$-consistent words $w,w'\in S^*$ and decides whether they are $\Gamma$-equivalent.
    \end{defn}
        
    \begin{defn}\label{def:mega-super-uniform-qi}
        Let $\Gamma_1=(M_1,S_1,R_1)$ and $\Gamma_2=(M_2,S_2,R_2)$ be two blueprints. We say that $\Gamma_1$ has a \define{computable quasi-isometric image} in $\Gamma_2$ if there exists a constant $N \geq 1$ and a pair of computable maps $\tau \colon \mathcal{M}(\Gamma_1) \to \mathcal{M}(\Gamma_2)$ and $f \colon \mathcal{M}(\Gamma_1) \times (S_1)^* \to (S_2)^*$ with the property that for each $\varphi_1 \in \mathcal{M}(\Gamma_1)$ and $u,v \in \supp(\varphi_1)$ then:
        \begin{enumerate}
            \item if $\underline{u}_{\Gamma_1} = \underline{v}_{\Gamma_1}$, then $\underline{f(\varphi_1,u)}_{\Gamma_2} = \underline{f(\varphi_1,v)}_{\Gamma_2}$. Thus $f(\varphi_1,\cdot)$ induces a well defined map from $G(\Gamma_1,\varphi_1)$ to $G(\Gamma_2,\tau(\varphi_1))$.
            \item $f(\varphi_1,\cdot) \colon G(\Gamma_1,\varphi_1) \to G(\Gamma_2,\tau(\varphi_1))$ is a quasi-isometry which is at most $N$-to-1.
        \end{enumerate}
        If $\Gamma_1$ and $\Gamma_2$ are quasi-isometric, and both $\Gamma_1$ has a computable quasi-isometric image in $\Gamma_2$ and vice-versa, we say that $\Gamma_1$ and $\Gamma_2$ are \define{computably quasi-isometric}.
    \end{defn}

    \begin{remark}
        While the conditions on~\Cref{def:mega-super-uniform-qi} seem very hard to satisfy, they become much simpler if the two blueprints are quasi-isometric and $\Gamma_2$ admits a computable model $\varphi_2$. In this case, the map $\tau$ can be taken constantly equal to $\varphi_2$. For instance, this happens if each blueprint represents a group with decidable word problem and their corresponding Cayley graphs are quasi-isometric.
    \end{remark}

\begin{theorem}\label{thm:medvedev_full}
    Let $\Gamma_1$, $\Gamma_2$ be two finitely presented and strongly connected blueprints, with decidable word problem, which are quasi-isometric, and such that $\Gamma_2$ has a computable quasi-isometric image in $\Gamma_1$. For every $\Gamma_2$-SFT $X$ there exists a $\Gamma_1$-SFT $Y$ such that $\mathfrak{m}(X)=\mathfrak{m}(Y)$.
\end{theorem}

\begin{proof}
    Write $X=X[\Gamma_2,\FF]$. Take $N$ larger than the constant from~\Cref{def:mega-super-uniform-qi} and consider $Y = \QI(\FF,N)$ as in~\Cref{thm:QI}. We already know by~\Cref{thm:medvedev_noassumpt} that $X \preceq_{\mathfrak{m}} Y$. Conversely, given $(\varphi_2,x)\in X$, we use $\tau$ and $f$ to produce a model $\varphi_1 = \tau(\varphi_2) \in \mathcal{M}(\Gamma_1)$ and map $g=f(\varphi_2)$ which induces a quasi-isometry from $G(\Gamma_2,\varphi_2) \to G(\Gamma_1,\varphi_1)$ that is at most $N$-to-1. The fact that both blueprints have decidable word problem makes the transformation of $g$ to an injective map $\widehat{g}\colon G(\Gamma_2,\varphi_2) \to G(\Gamma_1,\varphi_1) \times \{1,\dots, N\}$ from~\Cref{lem:ida} a computable process (we solve the word problem and assign indices lexicographically). Following the proof of~\Cref{lem:ida}, we obtain a point in $\QI(\FF,N)$ which is computable from a description of $(\varphi_2,x)\in X$. This shows that $Y \preceq_{\mathfrak{m}} X$ and thus we get that $\mathfrak{m}(X)=\mathfrak{m}(Y)$.
\end{proof}

For a blueprint $\Gamma$, denote by $\mathfrak{M}_{\texttt{SFT}}(\Gamma)$ the set of Medvedev degrees of all $\Gamma$-SFTs. The following corollary is a direct consequence of~\Cref{thm:medvedev_full}.

\begin{corollary}
    Let $\Gamma_1$, $\Gamma_2$ be two finitely presented strongly connected blueprints with decidable word problem that are computably quasi-isometric, then $\mathfrak{M}_{\texttt{SFT}}(\Gamma_1)=\mathfrak{M}_{\texttt{SFT}}(\Gamma_2)$.
\end{corollary}

\subsection{A hyperbolic example}

By applying the previous theorems to finitely generated co-compact Fuchsian groups, we obtain a new proof of the undecidability of their domino problem (originally proved for surface groups in~\cite{aubrun2019domino}, and later for all hyperbolic groups in~\cite{bartholdi2023domino}) and of the existence of strongly aperiodic SFTs (originally proved for surface groups in~\cite{cohen2017strongly}, and later for all hyperbolic groups~\cite{cohen2022strongly}). The following result relies on known results for tilings of the hyperbolic plane by regular pentagons, namely, the undecidability of its domino problem~\cite{kari2008undecidability}, and the existence of strongly aperiodic hyperbolic Wang tilings~\cite{kari2016piecewise}. Notice that such a tiling can be codyfied using the hyperbolic tiling blueprint (Example~\ref{ex:hyperbolic_tilings}) whose model graphs represent the dual graph of the tiling, and is such that the strongly aperiodic tiling by Wang tiles becomes a strongly aperiodic nearest neighbor SFT on $\mathcal{H}$.

\begin{theorem}
    Finitely generated co-compact Fuchsian groups have undecidable domino problem and admit strongly aperiodic SFTs.
\end{theorem}

\begin{proof}
    Consider the hyperbolic tiling blueprint $\mathcal{H}$ from Example~\ref{ex:hyperbolic_tilings}. We know from~\cite{kari2008undecidability} that the $\mathcal{H}$-domino problem is undecidable (see also~\cite[Theorem 1]{aubrun2019domino}), and from~\cite{kari2016piecewise} that $\mathcal{H}$ admits strongly aperiodic SFTs. Furthermore, every model graph of $\mathcal{H}$ is quasi-isometric to the hyperbolic plane $\mathbb{H}^2$. Next, by the \v{S}varc–Milnor Lemma we know that every finitely generated co-compact Fuchsian group is quasi-isometric to $\mathbb{H}^2$, and therefore quasi-isometric to every model graph of $\mathcal{H}$. By Theorems~\ref{thm:domino} and~\ref{thm:SA}, these groups have undecidable domino problem and admit strongly aperiodic SFTs.
\end{proof}




%
\section{The domino problem on tilings of the Euclidean space}
\label{sec:geometric}

The goal of this section is to apply our main result on subshifts defined on blueprints to show the undecidability of a variant of the domino problem for geometric tilings of the Euclidean space. The fundamental observation is that the space of geometric tilings given by a finite number of shapes can be modeled by a finitely presented blueprint under the assumption of finite local complexity.

\subsection{Geometric Tilings}

Let us give a brief introduction of geometric tilings. For further information we refer the reader to~\cite{Baake_Grimm_2013,Baake_Grimm_2017,Lagarias1999}. For the remainder of this section, the ambient space is assumed to be $\R^d$ equipped with the Euclidean norm for some fixed $d \geq 1$. We denote by $B_r(x)$ the closed ball of radius $r \geq 0$ centered in $x\in \R^d$, and write $B_r$ as a shorthand for $B_r(0)$.

A \define{tile} $t$ is a subset of $\R^d$ that is homeomorphic to the closed unit ball of $\R^d$. A \define{partial tiling} is a collection of tiles $\{t_i\}_{i\in I}$ whose interiors are pairwise disjoint, and we say it is \define{finite} if the index set $I$ is finite. The \define{support} of a partial tiling $P = \{t_i\}_{i\in I}$ is the union $\bigcup_{i \in I}t_i$. A partial tiling whose support is $\R^d$ is called a \define{tiling}.

Given a partial tiling $P$ and a subset $F\subseteq\R^d$, we denote by $P \sqcap F$ the set of all tiles from $P$ that intersect $F$, that is, \[P\sqcap F = \{ t\in P : t\cap F\neq\varnothing\}.\] We say two partial tilings $P_1$ and $P_2$ \define{match} in $F$ whenever $P_1\sqcap F = P_2 \sqcap F$. A partial tiling $P$ is called \define{locally finite} if for every compact $K\subset \R^d$ we have that $P \sqcap K$ is finite. In this case we call $P \sqcap K$ a \define{cluster}. Furthermore, if $K$ is convex, we call $P \sqcap K$ a \define{patch}.

Given $x\in \R^d$ and a partial tiling $P=\{t_i\}_{i\in I}$, its \define{translation} $P+x$ is the partial tiling given by \[P+x = \{t_i + x\}_{i\in I}.\]

A set of \define{punctured tiles} is a collection of tiles $\PP = \{p_1,\dots,p_n\}$ with the property that $0 \in \operatorname{int}(p_i)$ for $i \in \{1,\dots,n\}$ and such that distinct tiles do not coincide up to translation, that is, if $p_i = v + p_j$ for some $v \in \R^d$ and $i,j \in I$, then $i =j$ and $v = 0$. 

Let $\PP$ be a finite set of punctured tiles. A tiling $T = \{p_i\}_{i \in I}$ is \define{generated} by $\PP$ if for every $i \in I$ there exists a position $\operatorname{pos}(p_i) \in \R^d$ and $p \in \PP$ such that $p_i = \operatorname{pos}(p_i) + p$. Let us remark that as distinct tiles do not match up to translation, these positions and the corresponding punctured tile are uniquely defined for each $i \in I$. A tiling $T = \{p_i\}_{i \in I}$ generated by a set of punctured tiles $\PP$ is called \define{punctured} if there exists $i \in I$ for which $p_i \in \PP$, or equivalently, if $\operatorname{pos}(p_i)=0$ for some $i \in I$.

The space of all tilings of $\R^d$ generated by a set of punctured tiles $\PP$ is denoted by $\Omega(\PP)$ and its subspace of punctured tilings is denoted by $\Omega_{\circ}(\PP)$. We say that $\PP$ (and its tiling space $\Omega(\mathcal{P})$) has \define{finite local complexity} (FLC) if for every $r>0$ the set $\{T\sqcap B_r : T \in \Omega_\circ(\mathcal{P})\}$ is finite. Under the assumption of FLC, the space of punctured tilings $\Omega_\circ(\PP)$ is a compact metric space with the metric given by \[ d(T_1,T_2) = 2^{-\sup\{r \geq 0\ :\ T_1\sqcap B_r = T_2 \sqcap B_r\}} \mbox{ for all } T_1,T_2  \in \Omega_{\circ}(\mathcal{P}). \]

\begin{example}\label{ex:monotile}
The set of hat punctured tiles $\mathcal{P}_{\mathrm{hat}}$ is the collection given by the twelve tiles in $\RR^2$ that can be obtained by reflecting and rotating by multiples of $\pi/6$ the hat tile shown in~\Cref{fig:monotile}.

\begin{figure}[ht!]
    \centering
    \begin{tikzpicture}
        \pic[tile1,rotate=60]{tile};
    \end{tikzpicture}
    \caption{The hat tile}
    \label{fig:monotile}
\end{figure}
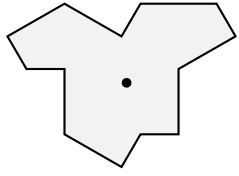

This is a famous example introduced by Smith, Myers, Kaplan and Goodman-Strauss~\cite{SmithMyersKaplanGS2024_monotile}. It has FLC and the remarkable property that its space of tilings is nonempty and contains no elements with translational symmetries. A patch of the monotile is shown in~\Cref{fig:monotiling}.

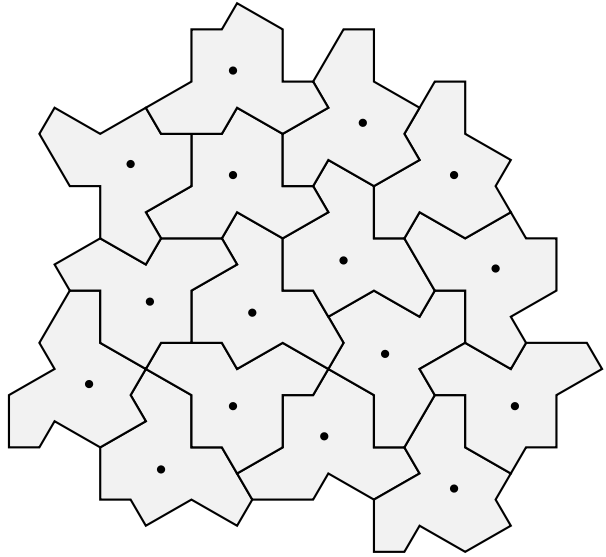
\begin{figure}[ht!]
    \centering
    \begin{tikzpicture}[scale=0.8]
\foreach\i in {0,...,11} \foreach[count=\j]\jj in {A,...,H} 
  \coordinate (\jj\i) at ({\j+0.5*mod(\i,2)},{\i*cos(30)});
\pic[tile1,scale=0.8]                      at (B1) {tile};
\pic[tile1,rotate=240,scale=0.8]           at (D2) {tile};
\pic[tile1,rotate=120,scale=0.8]           at (G2) {tile};
\pic[tile1,rotate=300,xscale=-1,scale=0.8] at (B3) {tile};
\pic[tile1,rotate=300,scale=0.8]           at (E3) {tile};
\pic[tile1,rotate=60,scale=0.8]            at (H3) {tile};
\pic[tile1,rotate=120,scale=0.8]           at (A4) {tile};
\pic[tile1,scale=0.8]                      at (D4) {tile};
\pic[tile1,rotate=60,scale=0.8]            at (B5) {tile};
\pic[tile1,scale=0.8]                      at (E5) {tile};
\pic[tile1,rotate=180,scale=0.8]           at (G6) {tile};
\pic[tile1,rotate=240,scale=0.8]           at (B7) {tile};
\pic[tile1,rotate=180,scale=0.8]           at (A8) {tile};
\pic[tile1,rotate=120,scale=0.8]           at (G8) {tile};
\pic[tile1,rotate=120,scale=0.8]           at (E9) {tile};
\pic[tile1,rotate=240,scale=0.8]           at (B9) {tile};
\end{tikzpicture}
    \caption{A patch generated by the punctured tiles of the monotile.}
    \label{fig:monotiling}
\end{figure}

\end{example}

\subsection{A blueprint for punctured geometric tilings}\label{subsec:blueprint_geometric}

Let $\PP$ be a finite set of punctured tiles with FLC. Our objective is to create a blueprint on which the space of models represents the space $\Omega_{\circ}(\PP)$ of punctured tilings generated by $\PP$.

With that end in mind, define \[ \rho = \inf\{ r >0 : \mbox{ for all } p \in \PP, p \subset B_r\}. \]
In other words, $\rho$ is the radius of the smallest closed ball centered at the origin that contains all punctured tiles. 

Let us fix a positive integer $K$. The set of states $M$ for our blueprint is given by the set of all patches of radius $K\rho$.
$$M = \{T\sqcap B_{K\rho} : T\in \Omega_\circ(\PP)\}.$$
It follows that $M$ is finite by the assumption that $\PP$ has FLC. Next, we define the set of generators $S$ as follows: \[ S = \{ (m,v) \in M \times \R^d: 0 < \|{v}\| \leq 3\rho \ \mbox{ and } \ v = \pos(t) \mbox{ for some } t \in M\}. \]

We note that the bound $3\rho$ might appear unnatural at this time, as clearly $2\rho$ suffices to move to any adjacent tile, however, the choice of $3\rho$ will allow us to construct a word that simulates moving in the $\rho$-neighborhood of a segment. This will be justified by~\Cref{lem:sausage} later on.

Given $s = (m,v) \in S$, we declare its initial vertex as $\init(s) = m$, and define its \define{valuation} as $\val(s)= v$. In addition, let \[ \ter(s) = \{m' \in M \colon  m'+\val(s) \mbox{ matches with } m \mbox{ in } B_{K\rho}(0) \cap B_{K\rho}(\val(s))\}.    \]
In other words, our generators connect two patches if they overlap correctly. We extend the valuation to words in $S^*$ by setting $\val(\varepsilon) = 0$ and for a nonempty word $w = s_1\dots s_n \in S^*$, we set $\val(w) = \sum_{i=1}^n \val(s_i)$.
Finally, fix a positive integer $L$. We define the set of relations as \[ R = \{ (w,w') \in S^*\times S^* \colon |w|+|w'| \leq L, \init(w) = \init(w') \mbox{ and}  \val(w)=\val(w') \}.    \]

Notice again that as there are finitely many pairs of words $w,w' \in S^*$ with $|w|+|w'|\leq L$, the set of relations is finite.

\begin{defn}\label{def:patchblueprint}
    Fix positive integers $K,L$. The \define{patch blueprint} of a set of punctured tiles $\PP$ with FLC is the finitely presented blueprint $\Gamma(\PP,K,L) = (M,S,\init,\ter,R)$ where $M, S, \init,\ter$ and $R$, are given as above.
\end{defn}

\subsection{Correspondence between models and punctured tilings}\label{appendix}
 Fix a set of punctured tiles $\PP$ such that $\Omega_\circ(\PP)$ has finite local complexity and let $\rho$ as in the previous section. The goal of this section is to prove that for large enough constants $K$ and $L$, the space of models $\mathcal{M}(\Gamma(\PP,K,L))$ of the patch blueprint completely captures the structure of the space of punctured tilings $\Omega_{\circ}(\PP)$.

We begin by proving three technical lemmas that will be needed in the proof and also used in the next section. For brevity, we shall write $\Gamma = \Gamma(\PP,K,L)$ for the patch blueprint.

\begin{lemma}\label{lem:valuation_is_well_defined}
    Let $\varphi \in \mathcal{M}(\Gamma)$ and let $w,w'\in \supp(\varphi)$ be $\Gamma$-equivalent words. Then $\val(w)=\val(w')$. 
\end{lemma}

\begin{proof}
    Suppose first that $w, w'$ are $\Gamma$-similar, that is, such that there exist $x,y,u,v\in S^*$ such that $w = xuy$, $w' = xvy$ and $(u,v)\in R$. By definition of $R$ we get that $\val(u)=\val(v)$. We deduce that
    $$\val(w) = \val(x) + \val(u) + \val(y) = \val(x) + \val(v) + \val(y) = \val(w').$$
    By induction, it follows that for any pair of $\Gamma$-equivalent words $w,w'$ we have $\val(w) = \val(w')$.
\end{proof}

\begin{lemma}[Interpolation]
\label{lem:sausage}
    Let $T$ be a partial tiling by translations of punctured tiles in $\PP$ whose support contains a convex set $C\subset \RR^d$. For any distinct $x,y \in C$ if we take $n = \lceil  \rho^{-1}\norm{y-x}\rceil$, there exists a sequence of tiles $t_0,\dots, t_{n}  \in T$ such that:
    \begin{enumerate}
        \item For every $k \in \{0,\dots,n\}$ we have $\norm{\pos(t_k)- (x +\frac{k}{n}(y-x))}\leq \rho$,
        \item For every $k \in \{1,\dots,n\}$ we have $\norm{\pos(t_k)-\pos(t_{k-1})}\leq 3\rho$.
        \end{enumerate}
\end{lemma}

\begin{proof}
    For each $k \in \{0,\dots,n\}$ let $x_k = x + \frac{k}{n}(y-x)$. As $C$ is convex, we have that $x_k \in C$ and as $C$ is contained in $\supp(T)$ we deduce that for each such $k$ there exists a tile $t_k$ such that $\norm{ \pos(t_k) - x_k } \leq \rho$. 

    We claim the collection of tiles $t_0,\dots,t_k$ satisfies the above requirements. The fist one is obvious from our construction. For the second one, let $k \in \{1,\dots,n\}$ and note that \begin{align*}
        \norm{\pos(t_k)-\pos(t_{k-1})} & \leq   \norm{\pos(t_k)-x_k} + \norm{x_k - x_{k-1}} +  \norm{\pos(t_{k-1})-x_{k-1}}\\
         & \leq 2\rho  +\frac{1}{n}\norm{y-x}\\
         & \leq 2\rho  +\lceil  \rho^{-1}\norm{y-x}\rceil^{-1}\norm{y-x}\\
         & \leq 3\rho.\qedhere
    \end{align*}
\end{proof}

\begin{lemma}[Visibility]\label{lem:approach}
    Let $T$ be a partial tiling by translations of punctured tiles in $\PP$ whose support contains a convex set $C\subset \RR^d$. Let $t,t'\in T$ such that $B_{2\rho}(\pos(t))\cup B_{2\rho}(\pos(t')) \subset C$. For every $x \in B_{K\rho}(\pos(t))\cap B_{K\rho}(\pos(t'))$, there exists a sequence of tiles $t = t_0,t_1,\dots t_n = t'$ in $T$ such that 
    \begin{enumerate}
        \item $n \leq  2+\lceil \rho^{-1}\norm{\pos(t)-\pos(t')}\rceil$,
        \item For every $k \in \{1,\dots,n\}$ we have $\pos(t_k)\in C$ and $\norm{\pos(t_{k})-\pos(t_{k-1})}\leq 3\rho$,
        \item $x\in \bigcap_{k=0}^n B_{K\rho}(\pos(t_{k}))$.
    \end{enumerate}
\end{lemma}

\begin{proof}
    Let $t,t'$ and $x$ as in the statement. If $\norm{\pos(t)-\pos(t')}\leq 3\rho$ we take $t_0 =t$, $t_1 = t'$ and there is nothing to prove. Otherwise, we define the points 
    $$y = \pos(t) + \rho\frac{x-\pos(t)}{\norm{x-\pos(t)}}, \ \textnormal{and} \ z = \pos(t') + \rho\frac{x-\pos(t')}{\norm{x-\pos(t')}}.$$
    
    Let $C' =  \{x' \in C : B_{\rho}(x')\subset C\} \cap B_{(K-1)\rho}(x)$ and notice that $y,z \in C'$. Applying~\Cref{lem:sausage} to the partial tiling $T$, the convex set $C'$, and to $y,z\in C'$, we obtain that if we let $m = \lceil  \rho^{-1}\norm{z-y}\rceil$, there exists a sequence of tiles $t_1,\dots,t_{m+1}$ which belong to $T$ and satisfy that
    \begin{enumerate}
        \item For every $k \in \{1,\dots,m+1\}$ we have $\norm{\pos(t_k)- (y +\frac{k-1}{m}(z-y))}\leq \rho$,
        \item For each $k \in \{2,\dots,m+1\}$, $\norm{\pos(t_k)-\pos(t_{k-1})}\leq 3\rho$.
    \end{enumerate}
    
    Fix $n = m+2$, $t_0 =t$ and $t_n = t'$. We claim that the collection $t_0,t_1,\dots,t_n$ satisfies the three conditions. First, an elementary computation shows that $\norm{z-y}\leq \norm{\pos(t)-\pos(t')}$ and thus \[n = 2 + \lceil  \rho^{-1}\norm{z-y}\rceil \leq 2 +\lceil\rho^{-1}\norm{\pos(t)-\pos(t')}\rceil.\]

    For the second condition, we already have $\norm{\pos(t_k)-\pos(t_{k-1})}\leq 3\rho$ for $k \in \{2,\dots,n-1\}$. For the border cases the bound follows from the following computation:\[ \norm{\pos(t_1)-\pos(t_0)} \leq \norm{\pos(t_1)-y}+\norm{y-\pos(t)} \leq 2\rho.  \]
    \[ \norm{\pos(t_{n})-\pos(t_{n-1})} \leq \norm{\pos(t')-z}+\norm{z-\pos(t_{n-1})} \leq 2\rho.  \]
    Finally, as for every $k \in \{1,\dots,n-1\}$ we have $\norm{\pos(t_k)- (y +\frac{k-1}{m}(z-y))}\leq \rho$, and $(y +\frac{k-1}{m}(z-y))\in C'$, we conclude that $\pos(t_k)\in C \cap  B_{K\rho}(x)$. In particular, as we already know by hypothesis that $x \in B_{K\rho}(\pos(t_{n}))$, we conclude that $x\in \bigcap_{k=0}^n B_{K\rho}(\pos(t_{k}))$.\end{proof}

Next we will make the correspondence between models in $\Gamma$ and tilings $\Omega_\circ(\PP)$ precise. To that end, we define $\Psi \colon \Omega_\circ(\PP) \to (M \cup \{\varnothing \})^{S^*}$ inductively as follows:
\begin{enumerate}
    \item $\Psi(T)(\varepsilon) = T \sqcap B_{K\rho}$.
    \item For every $w \in S^*$ and $s \in S$:
    \begin{enumerate}
        \item If $\Psi(T)(w) \neq \varnothing$, $\init(s) = \Psi(T)(w)$ and $T-\val(ws) \in \Omega_\circ(\PP)$, we let \[\Psi(T)(ws) = (T-\val(ws))\sqcap B_{K\rho}.    \]
        \item Otherwise, we set $\Psi(T)(ws) = \varnothing$.
    \end{enumerate}
\end{enumerate}

The correspondence between $\Omega_{\circ}(\PP)$ and $\mathcal{M}(\Gamma(\PP,K,L))$ can be expressed formally as follows.

\begin{prop}\label{prop:geometric_main}
    For $K \geq 117$ and $L \geq 2K+14$ the map $\Psi$ is an homeomorphism between $\Omega_{\circ}(\PP)$ and $\M(\Gamma(\PP,K,L))$.
\end{prop}

The rest of this section will be completely devoted to prove~\Cref{prop:geometric_main} and none of those results will serve any other purpose. We suggest the reader to first assume this proposition as true and jump to the next section before delving into this proof. We also do not claim that the values of $K$ and $L$ are optimal, they are just the smallest ones that make our arguments work.

Now we begin the proof of~\Cref{prop:geometric_main}. From this point onward we fix parameters $K \geq 117$ and $L \geq 2K+14$. Let $n$ be a positive integer. Clearly if $T,T'\in \Omega_{\circ}(\PP)$ match in $B_{(3n+K)\rho}$ then $\Psi(T)$ and $\Psi(T')$ must coincide on all words $w\in S^*$ of length at most $n$, from where it follows that the map $\Psi$ is continuous. The injectivity of $\Psi$ also follows directly from the recursive definition. 

\begin{lemma}\label{lem:geom_is_model}
    $\Psi(T)\in \M(\Gamma)$ for every $T\in \Omega_\circ(\PP)$.
\end{lemma}

\begin{proof}
    Let $T = \{t_i\}_{i \in I}  \in \Omega_{\circ}(\PP)$. The fact that $\Psi(T)$ is $\Gamma$-consistent is immediate from the definitions of $\Psi$, $S$ and $M$. Let us argue that if $w,w'\in \supp(\Psi(T))$ are $\Gamma$-equivalent, then $\Psi(T)(w) = \Psi(T)(w')$. Indeed, notice that for any $w\in \supp(\Psi(T))$ then $\Psi(T)(w) = (T-\val(w))\sqcap B_{K\rho}$. By~\Cref{lem:valuation_is_well_defined} it follows that if $w,w'\in \supp(\Psi(T))$ are $\Gamma$-equivalent, then $\val(w)=\val(w')$. In particular  \[ \Psi(T)(w) = (T-\val(w))\sqcap B_{K\rho} = (T-\val(w')) \sqcap B_{K\rho} = \Psi(T)(w'). \]
    From where we deduce that $\Psi(T)$ is a model.\end{proof}

All that remains to prove~\Cref{prop:geometric_main} is to show that $\Psi$ is surjective. For a $\Gamma$-model $\varphi$ and $w \in \supp(\varphi)$ we let $t_w = \val(w) +(\varphi(w)\sqcap \{0\})$. Notice that if $w,w'\in \supp(\varphi)$ are $\Gamma$-equivalent, then by the fact that $\varphi$ is a model we have $\varphi(w)=\varphi(w')$, and by~\Cref{lem:valuation_is_well_defined} we know $\val(w)=\val(w')$, thus $t_w = t_{w'}$. With this in mind, we define the set of tiles \[ T_{\varphi} = \{t_{w} : w \in  \supp(\varphi) \}. \]



    We will show that $T_{\varphi}$ is indeed a punctured tiling, that is $T_{\varphi}\in \Omega_\circ(\PP)$. If we have that, then from the definition of $\Psi$ it would follow that $\Psi(T_{\varphi})=\varphi$ and thus that $\Psi$ is surjective. For the remainder of the section, we fix a $\Gamma$-model $\varphi$ and let $T_{\varphi}$ be as above.

    For an integer $n \geq 0$, we define the set $\mathcal{R}(n)$ of all words whose valuation and those of their prefixes have norm at most $n\rho$, that is  \[ \mathcal{R}(n) = \{ w \in \supp(\varphi): \mbox{ for all prefixes } w' \mbox{ of } w, \norm{\val(w')} \leq n\rho \}. \]
    
    \begin{lemma}\label{lem:induccion-culia-fea}
        For every integer $n\geq 0$, and every pair of words $u,v\in \mathcal{R}(n)$, we have that $\varphi(u) + \val(u)$ matches $\varphi(v) + \val(v)$ in $B_{K\rho}(\val(u))\cap B_{K\rho}(\val(v))$.
    \end{lemma}

    The case $n=0$ is trivial: by our definition, $\val(s)>0$ for every generator, thus the only word such that every prefix has valuation at most $0$ is the empty word, that is, $\mathcal{R}(0) = \{\varepsilon\}$.

    Let $n \geq 1$ and suppose the inductive hypothesis holds for $n$. We will show that it also holds for $n+1$ through Claims~\ref{claim:short-range} and~\ref{claim:long-range}. before proving those claims, we will need two preliminary results.
    
    Consider the collection of tiles
    $$T_n = \{t_w : w \in \mathcal{R}(n)\}.$$
    That is, the set of all tiles which can be read from the model starting form a word in $\mathcal{R}(n)$. 

    \begin{claim}\label{claim:well-tiled-patch}
       $T_n$ is a partial tiling, and for $n \geq 1$ its support contains $B_{\rho (n-1)}$.
    \end{claim}

    \begin{proof}
        First we check that $T_n$ is a partial tiling. Suppose there are $u,v \in \mathcal{R}(n)$ such that $\operatorname{int}(t_u) \cap \operatorname{int}(t_v) \neq \varnothing$. By the inductive hypothesis we have that $\varphi(u) + \val(u)$ matches $\varphi(v) + \val(v)$ in $B_{K\rho}(\val(u))\cap B_{K\rho}(\val(v))$, in particular if we let $x \in \operatorname{int}(t_u) \cap \operatorname{int}(t_v)$, we must have that $t_u$ matches $t_v$ in $\{x\}$, thus $t_v = t_v$. This shows that distinct tiles in $T_n$ have pairwise disjoint interiors.

        Next we check that $B_{\rho (n-1)}\subset \supp(T_n)$. Let \[ \eta = \inf_{y \in \RR^d\setminus \supp(T_{n})} \norm{y}.  \]
     
     We have that $\eta>0$ as $t_{\varepsilon} = \varphi(\varepsilon) \sqcap \{0\} \in \PP$ contains the origin in its interior. Let $0 <\delta<\frac{\min(\eta,\rho)}{2}$ and take $x \in \R^d\setminus \supp(T_{n})$ which satisfies $\norm{x} - \delta \leq \eta$. Let $y = (1-\frac{2\delta}{\norm{x}})x$, as $2\delta<\eta\leq \norm{x}$, we have that $0 < (1-\frac{2\delta}{\norm{x}})<1$ and thus $\norm{y} = \norm{x}-2\delta$, from which we obtain that $y \in \supp(T_{n})$. It follows that there exists $w \in \mathcal{R}(n)$ such that $y \in t_w$. In particular, $\norm{y-\val(w)}\leq \rho$. We deduce that \[ \norm{x-\val(w)} \leq \norm{x-y}+\norm{y-\val(w)}+ \leq 2\delta+\rho <2\rho.   \]

    As $\varphi(w)$ covers $B_{K\rho}$, it follows that here exists $t \in \varphi(w)$ such that $x-\val(w)\in t$ and thus $\norm{\pos(t)} < 3\rho$. It follows that if we take $s = (\varphi(w),\pos(t))$ then necessarily $ws\in \supp(\varphi)$ and we have $x \in t_{ws} = t+\val(w)$. Moreover, \[\norm{\val(t_{ws})} \leq \norm{x}+\rho \leq \eta + \rho+\delta.\]

    If we suppose $\eta \leq \rho(n-1)$, we would have $\norm{\val(t_{ws})} \leq \rho n + \delta$. Noting that $\val$ takes values on a discrete subgroup $G$ of $\RR^d$ (that is, $G =\langle \val(s): s \in S\rangle$), there is $\delta>0$ small enough such that if $\norm{\val(t_{ws})} \leq \rho n + \delta$ implies that in fact $\norm{\val(t_{ws})} \leq \rho n$. Taking this value of $\delta$ 
     we get that $\norm{\val(t_{ws})} \leq \rho n$, thus $ws \in \mathcal{R}(n)$, which is a contradiction as this would imply that $x \in \supp(T_n)$. We conclude that $\eta > \rho(n-1)$.
    \end{proof}

    Next we show that given $w \in \mathcal{R}(n+1)$ we can replace it for another word for which all strict subwords have valuation at most $n\rho$.

    \begin{claim}\label{claim:forma_canonica}
        Let $w\in \mathcal{R}(n+1)$ be a nonempty word. There exists $u \in \mathcal{R}(n)$ and $s \in S$ such that $us \in \supp(\varphi)$ and $us$ is $\Gamma$-equivalent to $w$. In particular $\varphi(us)+\val(us) = \varphi(w)+\val(w)$.
    \end{claim}

    \begin{proof}
        Write $w=s_1\dots s_k$ with each $s_i \in S$ and let $w_i = s_1\dots s_i$ and $w_0 = \varepsilon$. If for every $i\in \{1,\dots,k-1\}$ we have $\norm{\val(w_i)}\leq n\rho$, then we can just take $u = w_{k-1}$ and $s = s_k$. Otherwise, let $j$ be the smallest positive integer such that $\norm{\val(w_j)}>n \rho$.

        Take $y,z \in \RR^d$ as follows:
        \begin{align*}
            y & = \begin{cases}
            (n-1)\rho\frac{\val(w_{j-1})}{\norm{\val(w_{j-1})}} & \mbox{if } \norm{\val(w_{j-1})} > 0,\\
            0 & \mbox{otherwise}.
        \end{cases}\\
        z & = \begin{cases}
            (n-1)\rho\frac{\val(w_{j+1})}{\norm{\val(w_{j+1})}} & \mbox{if } \norm{\val(w_{j+1})} > 0,\\
            0 & \mbox{otherwise}.
        \end{cases} 
        \end{align*}
        As $w\in \mathcal{R}(n+1)$, we have that both $\norm{\val(w_{j-1})}$ and $\norm{\val(w_{j+1})}$ are at most $(n+1)\rho$. Moreover, as $\norm{\val(w_j)}>n \rho$, we have that $\norm{\val(w_{j-1})}$ and $\norm{\val(w_{j+1})}$ are at least $(n-3)\rho$ and thus we deduce that both values belong to the interval $\bigl( (n-3)\rho, (n+1)\rho)\bigr)$.
        
        It follows that by construction of $y$ and $z$ we have that both $\norm{y-\val(w_{j-1})} \leq 2\rho$ and $\norm{\val(w_{j+1})-z}\leq 2\rho$. From this and the fact $\norm{\val(w_{j-1})- \val(w_{j+1})}\leq 6\rho$ we obtain that
        \[ \norm{y-z} \leq \norm{y-\val(w_{j-1})} + \norm{\val(w_{j-1})- \val(w_{j+1})} + \norm{ \val(w_{j+1})-z} \leq 10\rho.   \]

         \begin{figure}[h!]
        \centering
        \begin{tikzpicture}
        {\clip (-6,1.5) rectangle (6,6);
            \draw[thick, fill = black!05] (0,0) circle (5);
            \draw[thick, dashed, color=black!80] (0,0) circle (6);
            \draw[thick, dashed, color=black!80] (0,0) circle (4);
             \node[blue] (A) at ({4.5*cos(110)},{4.5*sin(110)}){$\bullet$};
             \node[blue] at ({0.5+4.5*cos(110)},{4.5*sin(110)}){$w_{j-1}$};
             \node[blue] (B) at ({5.3*cos(80)},{5.3*sin(80)}){$\bullet$};
             \node[blue] at ({0.3+5.3*cos(80)},{5.3*sin(80)}){$w_{j}$};
             \node[blue] (C) at ({5.8*cos(50)},{5.8*sin(50)}){$\bullet$};
             \node[blue] at ({0.5+5.8*cos(50)},{5.8*sin(50)}){$w_{j+1}$};
             \node[purple] (BB) at ({4*cos(110)},{4*sin(110)}){$\bullet$};
             \node[purple] at ({4*cos(110)},{4*sin(110)-0.2}){$y$};
             \node[purple] (CC) at ({4*cos(50)},{4*sin(50)}){$\bullet$};
             \node[purple] at ({4*cos(50)},{4*sin(50)-0.2}){$z$};
             \node[black!50!green] (D0) at ({3.7*cos(120)},{3.7*sin(120)}){$\bullet$};
             \node[black!50!green] at ({3.7*cos(120)+0.2},{3.7*sin(120)}){$t_0$};
             \node[black!50!green] (D1) at ({3.7*cos(105)},{3.7*sin(105)}){$\bullet$};
             \node[black!50!green] at ({3.7*cos(105)+0.2},{3.7*sin(105)}){$t_1$};
             \node[black!50!green] (D2) at ({3.1*cos(95)},{3.1*sin(95)}){$\bullet$};
             \node[black!50!green] at ({3.1*cos(95)+0.2},{3.1*sin(95)}){$t_2$};
             \node[black!50!green] (D3) at ({4.1*cos(90)},{4.1*sin(90)}){$\bullet$};
             \node[black!50!green] at ({4.1*cos(90)+0.2},{4.1*sin(90)}){$t_3$};
             \node[black!50!green] (D4) at ({3*cos(65)},{3*sin(65)}){$\bullet$};
             \node[black!50!green] at ({3*cos(65)+0.2},{3*sin(65)}){$t_8$};
             \node[black!50!green] (D5) at ({4.2*cos(60)},{4.2*sin(60)}){$\bullet$};
             \node[black!50!green] at ({4.2*cos(60)+0.2},{4.2*sin(60)}){$t_9$};
             \node[black!50!green] (D6) at ({3.3*cos(45)},{3.3*sin(45)}){$\bullet$};
             \node[black!50!green] at ({3.3*cos(45)+0.2},{3.3*sin(45)}){$t_{10}$};
             \draw[color =black!50, -> ] (A) to (D0);
             \draw[color =black!50, -> ] (D0) to (D1);
             \draw[color =black!50, -> ] (D1) to (D2);
             \draw[color =black!50, -> ] (D2) to (D3);
             \draw[color =black!50, dotted ] (D3) to (D4);
             \draw[color =black!50, -> ] (D4) to (D5);
             \draw[color =black!50, -> ] (D5) to (D6);
             \draw[color =black!50, -> ] (D6) to (C);
             \draw[color =black, -> ] (A) to (B);
             \draw[color =black, -> ] (B) to (C);}
             \draw [fill = white, color=white] (-6,1.5) rectangle (6,2);
             \node at (4.5,1.8) {\scalebox{0.8}{$n\rho$}};
             \node at (5.5,1.8) {\scalebox{0.8}{$(n\text{+}1)\rho$}};
             \node at (3.5,1.8) {\scalebox{0.8}{$(n\text{-}1)\rho$}};
            
        \end{tikzpicture}
        \caption{Sketch of the proof of~\Cref{claim:forma_canonica}}
        \label{fig:proof_forma_canonica}
    \end{figure}
        
        By~\Cref{lem:sausage} applied to $y,z$ and $T_n$, we obtain that there is a sequence $t_0,\dots,t_{10}$ of tiles in $T_n$ which are at consecutive distance at most $3\rho$ and at distance $\rho$ from the interval between $y$ and $z$. For each $\ell \in \{0,\dots,10\}$ take $u_{\ell}\in \mathcal{R}(n)$ such that $t_{\ell}=t_{u_{\ell}}$, see~\Cref{fig:proof_forma_canonica} for a sketch of this construction. Consider now the tuples 
        \begin{align*}
            a_0 &  = (\varphi(w_{j-1}), \val(u_0) -\val(w_{j-1}))\\
            a_i & = (\varphi(u_{i-1}),\val(u_{i-1})-\val(u_i)) \mbox{ for } i \in \{1,\dots,10\}\\
            a_{11} & = (\varphi(u_{10}),\val(w_{j+1})-\val(u_{10})).
        \end{align*}

        Notice that each of the second coordinates gives a vector of length at most $3\rho$. Moreover, by the inductive hypothesis, all of the patches in the first coordinate match pairwise, thus $a_0,\dots,a_{11} \in S$ and  $w_{j-1}a_0\dots a_{11} \in \supp(\varphi)$.
        
        Finally, notice that if we remove the last generator, we have $w_{j-1}a_0\dots a_{10} \in \mathcal{R}(n)$. Moreover, we have \[\val(a_0\dots a_{11}) = \val(w_{j+1})-\val(w_{j-1}) =\val(s_js_{j+1}).\]
        As $|a_0\dots a_{11}| + |s_js_{j+1}|\leq L$, it follows that $(a_0,\dots, a_{11},s_js_{j+1}) \in R$ (recall the definition of $R$ before~\Cref{def:patchblueprint}), in particular, the words $w_{j-1}a_0\dots a_{11}$ and $w_{j+1}$ are $\Gamma$-equivalent. Replacing $w_{j+1}$ by $w_{j-1}a_0\dots a_{11}$, we obtain a $\Gamma$-equivalent word where the value $k-j$ has been reduced by at least $1$. We obtain the desired decomposition iterating this procedure.
    \end{proof}

    The next claim proves~\Cref{lem:induccion-culia-fea} in the case where $x$ is ``close'' to $B_{n\rho}$. The main idea is to show that that, given $u,v \in \mathcal{R}(n+1)$ there are words $u',v'\in \mathcal{R}(n)$ whose value is also within range of $x$ and whose associated patches will match with those of $u,v$ and thus coincide on $\{x\}$. A picture sketching the argument is given in~\Cref{fig:proof_close_range}

    \begin{figure}
        \centering
        \begin{tikzpicture}
        {\clip (-6,1.5) rectangle (6,8);
            \draw[thick, fill = black!05] (0,0) circle (5);
            \draw[thick, dashed, color=black!80] (0,0) circle (5.5);
            \draw[thick, dashed, color=black!80] (0,0) circle (4.5);
            \node[blue] (A) at ({5.4*cos(140)},{5.4*sin(140)}){$\bullet$};
            \node[blue, rotate=40] at ({-0.3+5.4*cos(140)},{0.2+5.4*sin(140)}){$\val(u)$};
            \node[blue] (B) at ({5.2*cos(40)},{5.2*sin(40)}){$\bullet$};
            \node[blue, rotate=-40] at ({0.4+5.2*cos(40)},{0.4+5.2*sin(40)}){$\val(v)$};
            \node[blue] (C) at ({7.8*cos(80)},{7.8*sin(80)}){$\bullet$};
            \node[blue] at ({0.2+7.8*cos(80)},{7.8*sin(80)}){$x$};
            \node[purple] (C2) at ({2.5*cos(80)},{2.5*sin(80)}){$\bullet$};
            \node[purple] at ({0.2+2.5*cos(80)},{-0.2+2.5*sin(80)}){$x'$};

            \draw[color =black, dashed ] (C2) to (A);
            \draw[color =black, dashed ] (C2) to (B);

            \node[purple] (AA) at ({0.5*5.4*cos(140)+0.5*2.5*cos(80)},{0.5*5.4*sin(140)+0.5*2.5*sin(80)}){$\bullet$};
            \node[purple, rotate=0] at ({0.5*5.4*cos(140)+0.5*2.5*cos(80)},{-0.3+0.5*5.4*sin(140)+0.5*2.5*sin(80)}){$\val(u')$};
            
            \node[purple] (BB) at ({0.5*5.2*cos(40)+0.5*2.5*cos(80)},{0.5*5.2*sin(40)+0.5*2.5*sin(80)}){$\bullet$};
            \node[purple, rotate=0] at ({0.5*5.2*cos(40)+0.5*2.5*cos(80)},{-0.3+0.5*5.2*sin(40)+0.5*2.5*sin(80)}){$\val(v')$};
             
             \draw[color =black, - ] (A) to node [midway, above, rotate=40] {\scalebox{0.8}{$\leq K\rho$}} (C) ;
             \draw[color =black, - ] (B) to node [midway, above, rotate=-60] {\scalebox{0.8}{$\leq K\rho$}} (C) node [midway, above] {$\leq K\rho$};
             \draw[color =black, - ] (C) to (C2);

             \draw[color =black, dashed ] (C) to (AA);
            \draw[color =black, dashed ] (C) to (BB);
             
        }
             \draw [fill = white, color=white] (-6,1.5) rectangle (6,2);
             \node at (4.5,1.8) {\scalebox{0.8}{$n\rho$}};
             \node at (5.2,1.8) {\scalebox{0.8}{$(n\text{+}1)\rho$}};
             \node at (3.8,1.8) {\scalebox{0.8}{$(n\text{-}1)\rho$}};
            
        \end{tikzpicture}
        \caption{Sketch of the proof of~\Cref{claim:short-range}}
        \label{fig:proof_close_range}
    \end{figure}
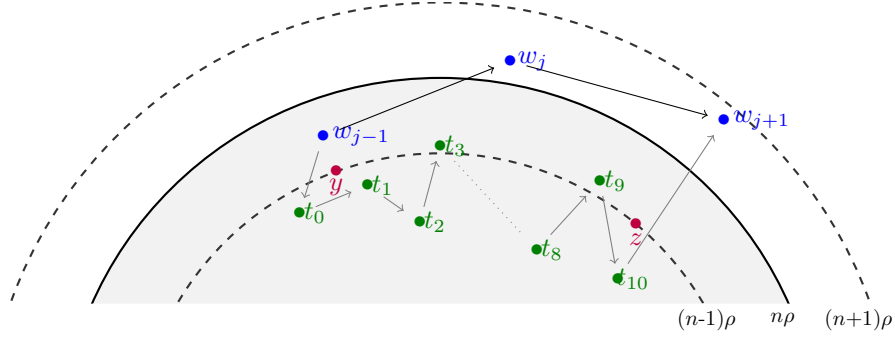

    \begin{claim}\label{claim:short-range}[short range consistency]
        Let $u,v \in \mathcal{R}(n+1)$. For any $x\in B_{K\rho}(\val(u))\cap B_{K\rho}(\val(v))$ with $\norm{x}\leq (n+K-12)\rho$ we have that $\varphi(u)+\val(u)$ matches $\varphi(v)+\val(v)$ on $\{x\}$.
    \end{claim}

    \begin{proof}
        We will first show that there exists a word $u'\in \mathcal{R}(n)$ for which $x\in B_{K\rho}(\val(u'))$ and such that $\varphi(u')+\val(u')$ matches with $\varphi(u)+\val(u)$ on $\{x\}$. Notice that if $\|\val(u)\|\leq n\rho$, by~\Cref{claim:forma_canonica} we may take $u' \in \mathcal{R}(n)$ which is $\Gamma$-equivalent to $u$ and thus satisfies the requirements.
        
        Let us study the case where $n\rho <\norm{\val(u)} \leq (n+1)\rho$. We claim that there exists a word $u'\in \mathcal{R}(n)$ with the property that both $\val(u)$ and $x$ are at distance at most $(K-2)\rho$ of $\val(u')$. Indeed, let \[x' = \begin{cases}
            (n-6)\rho\frac{x}{\norm{x}} & \mbox{if } \norm{x} > 0 \mbox{ and } n \geq 6,\\
            0 & \mbox{otherwise.}
        \end{cases}\] 

    As $x \in B_{K\rho}(\val(u))$ and $\norm{\val(u)}> n\rho$, we obtain that $\norm{x} > (n-K)\rho$. From this and $\norm{x}\leq (n+K-12)\rho$ we obtain that $\norm{x'-x} \leq (K-6)\rho$.

    For $\lambda \in [0,1]$ set $x_{\lambda} = \lambda x'+(1-\lambda) \val(u)$. Elementary computations show that 
    \begin{enumerate}
        \item $\norm{x_{\lambda}} \leq (n+1-6\lambda)\rho$.
        \item $\norm{x_{\lambda} - x} \leq (K-6\lambda)\rho$ 
        \item $\norm{x_{\lambda} - \val(u)} \leq 2\lambda(K-3)\rho.$
    \end{enumerate}
    Setting $z = x_{\lambda}$ for $\lambda = \frac{1}{2}$, we obtain that $\norm{z}\leq (n-2)\rho$, $\norm{z-x}\leq (K-3)\rho$ and $\norm{z-\val(u)} \leq (K-3)\rho$. As $\norm{z}\leq (n-2)\rho$, it follows by~\Cref{claim:well-tiled-patch} that $B_{(n-2)\rho}$ is contained in the support of $T_n$. Thus we deduce that there exists $u' \in \mathcal{R}(n)$ such that $z \in t_{u'}$, in particular $\norm{z-\val(u')}\leq \rho$. This implies $\|\val(u')\|\leq (n-1)\rho$, $\norm{\val(u')-x}\leq (K-2)\rho$ and $\norm{\val(u')-\val(u)} \leq (K-2)\rho$ as required.\\

    We will next show that $\varphi(u')+\val(u')$ matches with $\varphi(u)+\val(u)$ on $\{x\}$. 
    Applying~\Cref{claim:forma_canonica} to $u$ twice, we obtain $w \in \mathcal{R}(n-1)$ and $a,b \in S$ such that $u$ is $\Gamma$-equivalent to $wab$. Using~\Cref{lem:sausage} with $T_n$, $C=B_{(n-1)\rho}$ and positions $\val(w)$ and $\val(u')$, we obtain a sequence of at most \[\lceil \rho^{-1}\norm{\val(w)-\val(u')}\rceil \leq \lceil \rho^{-1}(\norm{\val(u)-\val(u')} + \norm{\val(ab)})\rceil \leq K+4\] tiles connecting them in $T_n$. From these tiles we get a word $p_{w,u'}$ of length at most $K+4$ such that $wp_{w,u'}$ is $\Gamma$-equivalent to $u'$.

    Take $T' = \varphi(u')+\val(u')$ and consider the tiles $t_{u'}$ and $t_u$. It is clear that $t_{u'} \in T'$, to see that $t_u \in T'$, note that as both $u', w \in \mathcal{R}(n)$, by the inductive hypothesis we have that $\varphi(u')+\val(u')$ matches $\varphi(w)+\val(w)$ at $\val(u)$, hence as $(\varphi(w)+\val(w))\sqcap \{\pos(u)\} = t_u$, we deduce that $t_u \in T'$. Furthermore, as $\norm{\val(u')-\val(u)} \leq (K-2)\rho$, we deduce that $B_{2\rho}(\pos(u))\subset \supp(T')$. Finally, notice that $\norm{x-\pos(t_{u})} = \norm{x-\val(u)} \leq K\rho$ and $\norm{x-\pos(t_{u'})} = \norm{x-\val(u')} \leq (K-2)\rho$ 

    By~\Cref{lem:approach}, there exists a path of tiles of length at most $2 + \lceil\rho^{-1}\norm{\val(u)-\val(u')} \rceil \leq K$ between $t_{u'}$ and $t_u$ with the property that for each tile $t$ in this path we have $x \in B_{K\rho}(\pos(t))$. From here, we get a word $p_{u'u}$ of length at most $K$, with the property that for each subword $p'$ we have that $x \in B_{K\rho}(\val(u'p'))$. In particular, as subsequent words must match in their intersection, we deduce that $\varphi(u'p_{u'u})+\val(u'p_{u'u})$ matches $\varphi(u')+\val(u')$ at $\{x\}$.

    Consider the words $ab$ and $p_{w,u'}p_{u'u}$. Clearly they have the same initial state (which is $\varphi(w)$) and $\val(ab) = \pos(u) - \pos(w) = \val(p_{w,u'}p_{u'u})$. As $|ab|+|p_{w,u'}p_{u'u}| \leq 2K+6 \leq L$, we deduce that $ab$ and $p_{w,u'}p_{u'u}$ are $\Gamma$-equivalent. This in turn implies that $u$ and $u'p_{u'u}$ are $\Gamma$-equivalent, and thus $\varphi(u)+\val(u)$ matches $\varphi(u')+\val(u')$ at $\{x\}$.\\

    Applying the same argument to $v$ we obtain that there exists $v'\in \mathcal{R}(n)$ such that $\varphi(v')+\val(v')$ matches with $\varphi(v)+\val(v)$ on $\{x\}$. As both $u',v' \in \mathcal{R}(n)$, the inductive hypothesis shows that $\varphi(v')+\val(v')$ matches with $\varphi(u')+\val(u')$ on $\{x\}$, from where we deduce that $\varphi(v)+\val(v)$ matches with $\varphi(u)+\val(u)$ on $\{x\}$.   \end{proof}

    The final claim deals with the case when $x$ is ``far'' from $B_{K\rho}$. The main observation, is that in this case the hypothesis that $x\in B_{K\rho}(\val(u))\cap B_{K\rho}(\val(v))$ will force $\val(u)$ and $\val(v)$ to be close, thus we will be able to connect them back and forth using a pair of words whose added lengths are at most $2K+14\leq L$ and their valuations are all within distance $K\rho$ of $x$, thus showing what we want.

     \begin{claim}[long range consistency]\label{claim:long-range}
        Let $u,v\in \mathcal{R}(n+1)$. For any $x\in B_{K\rho}(\val(u))\cap B_{K\rho}(\val(v))$ such that $\norm{x}\geq (n+K-12)\rho$ we have that $\varphi(u)+\val(u)$ matches $\varphi(v)+\val(v)$ on $\{x\}$.
    \end{claim}

    \begin{proof}
    We first show that in this case we have $\norm{\val(u)-\val(v)}\leq (K-2)\rho$.
    
    Suppose that $\norm{\val(u)-\val(v)}> (K-2)\rho$ and consider the orthogonal projection $x'$ of $x$ onto the line $\{\val(u) + r(\val(v)-\val(u)) \ : \ r\in \R\}$. As both $\norm{x-\val(u)}$ and $\norm{x-\val(v)}$ are at most $K\rho$, we deduce that $x'$ is at distance at most $2\rho$ from the segment $[\val(u),\val(v)] = \{ \lambda \val(u) + (1-\lambda)\val(v) : \lambda \in [0,1]\}$ and thus that $\norm{x'} \leq (n+3)\rho$. Since $\norm{x}\geq (n+K-12)\rho$, we deduce that $\norm{x-x'}\geq (K-15)\rho$. By the Pythagorean theorem we have \[ \norm{\val(u)-x'}^2 = \norm{x-\val(u)}^2 -\norm{x-x'}^2 \leq K^2\rho^2 - (K-15)^2\rho^2.    \]
    From where we obtain that \[  \norm{\val(u)-x'} \leq \rho( \sqrt{30K-225}). \]
    Analogously, we deduce the same bound for $\norm{\val(v)-x'}$ and thus we get \[ \norm{\val(u)-\val(v)} \leq 2\rho( \sqrt{30K-225}).  \]

    As $K \geq 117$, we have that $2( \sqrt{30K-225})\leq K-2$ and thus we deduce that \[ \norm{\val(u)-\val(v)} \leq (K-2)\rho.  \]
    Yielding a contradiction.

    Now consider the partial tiling $T'=\varphi(u)+\val(u)$. By~\Cref{claim:short-range} and the facts that $\norm{\val(u)-\val(v)} \leq (K-2)\rho$ and $\val(v)\leq (n+1)\rho$, we deduce that $t_v \in T'$. Applying~\Cref{lem:approach} to $t_u,t_v$ in $T'$ with $C=B_{K\rho}(\val(u))$, we get $k \leq 2+ \norm{\val(v)-\val(u)}/\rho \leq K$ and a sequence of tiles $t_0,\dots, t_k\in T'$ with $t_0=t_u$, $t_k = t_v$ and such that $\norm{\pos(t_{k+1})-\pos(t_k)}\leq 3\rho$ and $x \in \bigcap_{i=0}^k B_{\rho K}(t_i)$. 
    
    Consequently, we may extract a word $w=s_1\dots s_k \in S^*$ by letting  \[s_{i+1} = (\varphi(us_1 \dots s_{i}),\pos(t_{i+1})-\pos(t_i)) \quad \mbox{ for } j \in \{1,\dots,k-1\}.\] 
    This word has the property that $\val(uw) = \val(v)$, and that for each $i\in \{1,\dots,k\}$, $x \in B_{K\rho}(\pos(u s_1\dots s_i))$. From the definition of $M$ we deduce that $\varphi(u)+\val(u)$ matches $\varphi(uw)+\val(uw)$ at $\{x\}$.

    Using~\Cref{claim:forma_canonica} twice, we can find $u',v'\in \mathcal{R}(n-1)$ and $s_u,s'_u,s_v,s'_v \in S$ such that $u's'_us_u$ and $v's'_vs_v$ are $\Gamma$-equivalent to $u$ and $v$ respectively. In particular, $\norm{\val(u')-\val(v')}\leq (K-2)\rho+12\rho = (K+10)\rho$. Applying~\Cref{lem:sausage} we can extract a word $\tilde{w}$ of length at most $K+10$ with the property that $u'\tilde{w}\in \mathcal{R}(n)$ and such that $\val(u'\tilde{w})=\val(v')$. By induction hypothesis, we deduce that $\underline{u'\tilde{w}}_{\Gamma} = \underline{v'}_{\Gamma}$. It follows that we may construct a word $w'$ of length $|\tilde{w}|+4=K+14$ connecting $u$ and $v$ such that $\underline{uw'}_\Gamma =\underline{v}_{\Gamma}$. As $|w|+|w'|\leq 2K+14 \leq L$, we deduce that $v$ is $\Gamma$-equivalent to $uw$ and thus that $\varphi(u)+\val(u)$ matches $\varphi(v)+\val(v)$ at $\{x\}$.\end{proof}

    Putting together Claims~\ref{claim:short-range} and~\ref{claim:long-range} we get~\Cref{lem:induccion-culia-fea}. As $T_{\varphi} = \bigcup_{n \geq 1}T_n$, by~\Cref{claim:well-tiled-patch} we obtain that $T_{\varphi}$ is indeed a tiling of $\RR^d$ by punctured tiles in $\PP$, thus we have proven~\Cref{prop:geometric_main}.

\subsection{The Domino Problem on Geometric Tilings} We first show that the domino problem of the patch blueprint is undecidable whenever the dimension of the ambient space is at least $2$. 

\begin{prop}
\label{prop:patchdomino}
    Fix $K \geq 117$ and $L \geq 2K+14$. Let $\Gamma = \Gamma(\PP,K,L)$ be the patch blueprint of a set of punctured tiles $\PP$ of $\RR^d$ with FLC. If $d \geq 2$, then
    \begin{enumerate}
        \item the $\Gamma$-domino problem is undecidable.
        \item For every $\varphi \in \mathcal{M}(\Gamma)$, the $\varphi$-domino problem is undecidable.
    \end{enumerate}
\end{prop} 

\begin{proof}
It is well known that the domino problem for the group $\ZZ^d$ is undecidable for $d\geq 2$ and it is clear that its blueprint is minimal, as it contains a single model. As $\ZZ^d$ is quasi-isometric to $\RR^d$, it suffices by~\Cref{thm:domino,thm:domino_model} to show that every model graph of $\Gamma$ is quasi-isometric to $\R^d$. 
    
Consider a model $\varphi\in\M(\Gamma)$. We define the map $f\colon G(\Gamma,\varphi)\to \RR^d$ by setting $f(\underline{w}_{\Gamma}) = \val(w)$ for $w \in \supp(\varphi)$. Notice that $f$ is well-defined by~\Cref{lem:valuation_is_well_defined}. By Proposition~\ref{prop:geometric_main} there exists a unique tiling $T\in\Omega_0(\PP)$ such that $\Psi(T) = \varphi$. Furthermore, for all $w \in \supp(\varphi)$ we have $t_w = (\varphi(w) \sqcap \{0\})+\val(w) \in T$. Notice also that $\pos(t_w) = \val(w)$ for every $w \in \supp(\varphi)$.

Now, if we denote by $d$ the quasi-metric on $G(\Gamma,\varphi)$, by the definition of the generating set we have
$$\|f(\underline{w}_{\Gamma}) - f(\underline{v}_{\Gamma})\|\leq 3\rho\cdot d(\underline{w}_{\Gamma}, \underline{v}_{\Gamma}),$$
for all $w, v\in\supp(\varphi)$. For the other inequality, take $w,v \in\supp(\varphi)$. By applying Lemma~\ref{lem:sausage} to $t_w$ and $t_v$, there exists a path from $\underline{w}_{\Gamma}$ to $\underline{v}_{\Gamma}$ in $G(\varphi,\Gamma)$ of length at most  $\lceil \rho^{-1}\norm{\pos(t_w) - \pos(t_v)}\rceil$. In other words,
$$\rho\cdot d(\underline{w}_{\Gamma}, \underline{v}_{\Gamma}) - \rho \leq \|f(\underline{w}_{\Gamma}) - f(\underline{v}_{\Gamma})\|.$$
This proves that $f$ is a quasi-isometric embedding. To prove the quasi-density condition, take $x\in\R^d$. As $T$ is a tiling, there exists $t\in T$ such that $x\in t$ and $\|x-\pos(t)\|\leq \rho$. As $\varphi = \Psi(T)$, there exists $w\in\supp(\varphi)$ such that $\varphi(w) = t$ and $\val(w) = \pos(t)$. Therefore, $\norm{f(w)-x} = \norm{x-f(w)}\leq\rho$. This proves $f$ is a quasi-isometry between $G(\Gamma, \varphi)$ and $\R^d$.
\end{proof}

Now that we know that the Domino Problem is undecidable for the patch blueprint (on both variants), we interpret this result in terms of the underlying tiling. \\

Let $A$ be a finite alphabet and let $\mathcal{T}$ be a set of tiles. A \define{colored tile} is a tuple $(t,a)\in \mathcal{T}\times A$.  Given a colored tile $c=(t,a)$ its translation by $v\in \RR^d$ is given by $c+v = (t+v,a)$. A colored (partial) tiling is a set $T = \{ (t_i,a_i)\}_{i \in I}$ of colored tiles such that its \define{geometric projection} $\pi(T) = \{ t_i\}_{i \in I}$ is a (partial) tiling of $\RR^d$. Similarly, we define colored clusters, patches and their translations.\\

Let $\PP$ be a set of punctured tiles and $A$ a finite alphabet. A colored partial tiling $T = \{c_i\}_{i \in I}$ is \define{generated} by $(\PP,A)$ if for every $i \in I$ there exists $v \in \RR^d$ such that $c_i+v \in \PP\times A$. The space of colored tilings generated by $(\PP,A)$ is denoted by $\Omega(\PP,A)$ and the subspace of punctured colored tilings $T$ which satisfy $\pi(T)\in \Omega_{\circ}(\PP)$ is denoted by $\Omega_{\circ}(\PP,A)$.

We can define interesting sets of colored tilings by forbidding colored partial tilings with finite support.

\begin{defn}
   Let $\PP$ be a set of punctured tiles and $A$ a finite alphabet. Let $\mathcal{F}$ be a collection of colored partial tilings with finite support. The \define{geometric subshift} generated by $\PP,A$ and $\mathcal{F}$ is the space \[  \Omega(\PP,A,\mathcal{F}) = \{ T \in \Omega(\PP,A) : \mbox{ for all } P \in \mathcal{F} \mbox{ and } v \in \R^d, P+v \not\subset T\}.  \]
   If $\FF$ is finite, we say that $\Omega(\PP,A,\mathcal{F})$ is a geometric subshift of \define{finite type}.
\end{defn}

Moreover, we can consider colorings of a single (uncolored) tiling. Given $\widehat{T}\in \Omega(\PP)$, we set \[ \Omega(\PP,A,\FF,\widehat{T}) = \{ T \in \Omega(\PP,A,\FF) : \pi(T) = \widehat{T}\}. \]

\begin{example}
    Consider the set $\PP$ of punctured hat monotiles from~\Cref{ex:monotile} and take $A = \{ \begin{tikzpicture}
        \draw[black, thick, fill=azul] (0,0) circle (0.1); 
    \end{tikzpicture},\begin{tikzpicture}
        \draw[black, thick, fill=rojo] (0,0) circle (0.1); 
    \end{tikzpicture} \}$ as the alphabet which consists of the colors blue and red respectively. Let $\FF$ be the set of all patches which consist on two adjacent tiles colored with red. Notice that there are finitely many of these patches. The geometric subshift of finite type $\Omega(\PP,A,\FF)$ is an analogue of the hard-square subshift on tilings by the monotile (see Example~\ref{ex:hardsquare}). A colored patch without forbidden patterns is shown in~\Cref{fig:monotiling_colores}.
\end{example}

\begin{figure}[ht!]
    \centering
    \begin{tikzpicture}[scale=0.7]
\foreach\i in {0,...,11} \foreach[count=\j]\jj in {A,...,H} 
  \coordinate (\jj\i) at ({\j+0.5*mod(\i,2)},{\i*cos(30)});
\pic[tile1,fill=rojo,scale=0.7]                      at (B1) {tile};
\pic[tile1,rotate=240, fill=azul,scale=0.7]           at (D2) {tile};
\pic[tile1,rotate=120,fill=azul,scale=0.7]           at (G2) {tile};
\pic[tile1,rotate=300,xscale=-1,fill=azul,scale=0.7] at (B3) {tile};
\pic[tile1,rotate=300,fill=azul,scale=0.7]           at (E3) {tile};
\pic[tile1,rotate=60, fill=rojo,scale=0.7]            at (H3) {tile};
\pic[tile1,rotate=120,fill=azul,scale=0.7]           at (A4) {tile};
\pic[tile1,fill=rojo,scale=0.7]                      at (D4) {tile};
\pic[tile1,rotate=60,fill=azul,scale=0.7]            at (B5) {tile};
\pic[tile1,fill=azul,scale=0.7]                      at (E5) {tile};
\pic[tile1,rotate=180,fill=azul,scale=0.7]           at (G6) {tile};
\pic[tile1,rotate=240,fill=azul,scale=0.7]           at (B7) {tile};
\pic[tile1,rotate=180,fill=rojo,scale=0.7]           at (A8) {tile};
\pic[tile1,rotate=120,fill=azul,scale=0.7]           at (G8) {tile};
\pic[tile1,rotate=120,fill=rojo,scale=0.7]           at (E9) {tile};
\pic[tile1,rotate=240,fill=azul,scale=0.7]           at (B9) {tile};
\end{tikzpicture}
    \caption{A patch of the hard square subshift defined over tilings by the hat monotile.}
    \label{fig:monotiling_colores}
\end{figure}

Next we define the domino problem for geometric subshifts over a fixed set of punctured tiles $\PP$. Intuitively, this is the decision problem where one asks, given a finite alphabet $A$ and a finite set of forbidden partial tilings with finite support $\FF$, whether $\Omega(\PP,A,\FF) \neq \varnothing$. In particular, we are interested in the decidability of this problem, that is, does there exists a Turing machine which takes as input an instance of the decision problem, and halts if and only if $\Omega(\PP,A,\FF) \neq \varnothing$? \\

There is one problem with this naive approach: as we do not ask for any computability condition on the set $\PP$, it is not true that one can computably generate partial tilings by tiles in $\PP$ (or even decide simple things such as if the translation of two tiles has nonempty intersection). Therefore it is not obvious how to encode $\FF$. However, there is a way to abstractly encode geometric subshifts of finite type if one assumes that $\PP$ has FLC. Indeed, let $\rho = \inf\{ r >0 : \mbox{ for all } p \in \PP, p \subset B_r\}$ and take the set $\mathcal{D}$ of all $v \in \RR^d$ with $0 < \norm{v}\leq 3\rho$ for which there exists $p,p' \in \PP$ such that $\{p, p'+v\}$ is a partial tiling that occurs in some $T\in \Omega_{\circ}(\PP)$. By the assumption of FLC, we have that $\mathcal{D}$ is finite, and it is not hard to see (for instance, using~\Cref{lem:sausage}) that if $T\in \Omega_{\circ}(\PP)$ and $t\in T$, then $t = p+x$ for some $p \in \PP$ and some $x$ in the discrete additive subgroup $\langle \mathcal{D}\rangle \leqslant \RR^d$ which is generated by $\mathcal{D}$. We call $\langle \mathcal{D}\rangle $ the \define{punctured position group} of $\PP$.
Now, suppose $\mathcal{D}= \{v_1,\dots,v_k\}$. There is a surjective homomorphism $\eta \colon \ZZ^k \to \langle \mathcal{D}\rangle$ given by \[ \eta(n_1,\dots,n_k) = \sum_{i=1}^k n_iv_i. \]

\begin{defn}
    Let $A$ be a finite alphabet, $\PP$ be a set of punctured tiles and $\langle \mathcal{D}\rangle$ be its punctured position group. A \define{colored pretiling coding} is a pair $(F,\xi,c)$ where $F$ is a finite subset of $\ZZ^k$, $\xi \colon F \to \PP$ and $c\colon F \to A$. The \define{encoded colored pretiling} generated by $(F,\xi,c)$ is the collection of colored tiles \[\mathtt{E}(F,\xi,c) = \{ (\xi(v),c(v)) + \eta(v) \}_{v \in F}.     \]
    We say that $(F,\xi,c)$ is \define{consistent} if $\mathtt{E}(F,\xi,c)$ is a colored partial tiling. Given a collection $\mathcal{C}$ of colored pretiling codings, we write \[\mathcal{F}(\mathcal{C}) = \{ \mathtt{E}(F,\xi,c) : (F,\xi,c) \in \mathcal{C} \mbox{ is consistent} \}.   \]
\end{defn}

We remark that every colored partial tiling with finite support which occurs as the restriction of some tiling in $\Omega(\PP,A)$ can be encoded by a colored pretiling coding, and thus every for every set $\FF$ of colored partial tilings with finite support, there exists a set $\mathcal{C}$ of colored pretiling codings such that $\FF(\mathcal{C})=\FF$. In particular, if we identify the alphabet $A$ with its cardinality, there is a natural bijection from the set of all pairs $(A,\mathcal{C})$ to the natural numbers. Let $\langle A,\mathcal{C}\rangle$ denote this number.

\begin{defn}
    Let $\PP$ be a set of punctured tiles.
    \begin{enumerate}
        \item The \define{$\PP$-domino problem} is the decision problem which given as input an alphabet $A$ and a finite set of colored pretilings $\mathcal{C}$, whether $\Omega(\PP,A,\FF(\mathcal{C}))\neq \varnothing$.
        \item Let $\widehat{T}\in \Omega(\PP)$ be a tiling. The $\widehat{T}$-\define{domino problem} is the decision problem which given as input an alphabet $A$ and a finite set of colored pretilings $\mathcal{C}$, whether $\Omega(\PP,A,\FF(\mathcal{C}), \widehat{T})\neq \varnothing$.
    \end{enumerate}
\end{defn}


\begin{theorem}
\label{thm:domino_geom}
Let $d\geq 2$ and $\PP$ be a finite set of punctured tiles with FLC. Then
\begin{enumerate}
    \item The $\PP$-domino problem is undecidable.
    \item For every tiling $\widehat{T}\in \Omega(\PP)$, the $\widehat{T}$-domino problem is undecidable.
\end{enumerate}
\end{theorem}

\begin{proof}
 Take $K = 117$, $L=248$ and let $\Gamma = \Gamma(\PP,K,L)$ be the corresponding patch blueprint. By Proposition~\ref{prop:patchdomino} the $\Gamma$-domino problem is undecidable and for every $\varphi \in \M(\Gamma)$ the $\varphi$-domino problem is undecidable. We will first describe a total computable map that transforms nearest neighbor forbidden patterns for $\Gamma$ into a set of set of colored pretilings of $\PP$.
 
 Let $\mathcal{N}$ be a set of nearest neighbor forbidden patterns for $\Gamma$ over some finite alphabet $A$. Recall that a nearest neighbor pattern in this case is a map $p \colon \{\varepsilon,s\} \to (M\times A)$ for some $s \in S$. Consider $p \in \mathcal{N}$ and write $(m,a)=p(\varepsilon)$ and $(m',b) = p(s)$. For every $t \in m \cup (m' + \val(s))$, choose some $z_t\in \ZZ^k$ such that $\eta(z_t) = \pos(t)$. Write particularly $z_0$ and $z_s$ for values with $\eta(z_0)=0$ and $\eta(z_{s})=\val(s)$. We define \[ F_p = \{ z_t : t \in m \cup (m' + \val(s))\} \mbox{ and } \xi_p \colon F_p \to \PP \mbox{ with } \xi_p(z_t) = t - \pos(t). \] 
We associate to $p$ the collection $\mathcal{C}_{p}$ of colored pretiling codings given by \[\mathcal{C}_p = \{ (F_p,t_p,c) :  c(z_0)=a, c(z_{s})=b\}.    \]
Next we consider the collection \[ \mathcal{C} = \bigcup_{p \in \mathcal{N}} \mathcal{C}_p.  \]
We remark that $\mathcal{C}$ can be computed from $\mathcal{N}$. Indeed, as $M$ and $S$ are finite, there are finitely many possibilities for $m\cup(m+\val(s))$ and thus the association $t \to z_t$ can be hard-coded in an algorithm, and the rest is directly computable. \\

Next, we will show that if $T = \{(p_i,a_i)\}_{i \in I}\in \Omega(\PP,A,\mathcal{F}(\mathcal{C}))$ is a colored tiling such that $\pi(T)\in \Omega_{\circ}(\PP)$ with $\varphi= \Psi(\pi(T))$, then there exists $x\in (A\cup\{\varnothing\})^{S^{*}}$ such that $(\varphi,x)\in X[\Gamma,\mathcal{N}]$. 

Consider a colored tiling $T = \{(p_i,a_i)\}_{i \in I}\in \Omega(\PP,A,\mathcal{F}(\mathcal{C}))$. By Proposition~\ref{prop:geometric_main}, we can take $\varphi = \Psi(\pi(T))$ and define $x\in (A\cup\{\varnothing\})^{S^{*}}$ by
 \[x(w) = \begin{cases}
     a_i & \textnormal{if } w \in \supp(\varphi) \mbox{ and } \val(w)=\pos(p_i) \mbox{ for some } i \in I \\
     \varnothing & \textnormal{otherwise.}
 \end{cases}\]

By the definition of $\Psi$, notice that for $w \in \supp(\varphi)$ we have that $\varphi(w)= (\pi(T)-\val(w))\sqcap B_{K\rho}$. In particular, if we choose $i \in I$ with $t_i = (\pi(T)-\val(w))\sqcap \{0\}$, it follows that $x(w)=a_i$. Therefore, $\supp(x)=\supp(\varphi)$. Furthermore, by~\Cref{lem:movement_is_well_defined} we have that if $w,w'$ are $\Gamma$-equivalent, then $\val(w)=\val(w')$. It follows that $x(w)=x(w')$, thus conditions (s1) and (s2) of~\Cref{def:phi_subshift} are satisfied.

Next, suppose there exist $p \in \mathcal{N}$ with support $\{\varepsilon,s\}$ and $w \in \supp(\varphi)$ such that $(\varphi(w),x(w))=p(\varepsilon)=(m,a)$ and $(\varphi(ws),x(ws))=p(s)= (m',b)$. Consider \[T_p = (T - \val(w)) \sqcap (B_{K\rho}\cup B_{K\rho}(\val(s))).\]
From the definition of $\Psi$ it follows that $\pi(T_p) = m \cup (m'+\val(s))$. Consider the map $c \colon F_p \to A$ given by $c(z)= a_i$ where $a_i$ is such that $(p_i,a_i) \in T_p$ and $\pos(p_i)=\val(z)$ for $z\in F_p$. With this choice, it follows that $c(z_0)=x(w)=a$ and $c(z_s)=x(ws)=b$. Therefore $\mathtt{E}(F_m,\xi_m,c)$ is consistent and \[ \mathtt{E}(F_m,\xi_m,c) + \val(w) \subset T \]
 As $(F_p,\xi_p,c) \in \mathcal{C}_p \subset \mathcal{C}$, we deduce that $T \notin \Omega(\PP,A,\mathcal{F}(\mathcal{C}))$, which is a contradiction. Therefore, condition (s3) is verified and $(\varphi, x)\in X[\Gamma, \mathcal{N}]$.\\

Next we show the other direction, that if, that given $(\varphi, x)\in X[\Gamma,\mathcal{N}]$ we can construct $T=\{(t_i,a_i)\}_{i \in I}\in \Omega(\PP,A,\FF(\mathcal{C})$ such that $\pi(T)=\Psi^{-1}(\varphi)$.

Let $(\varphi, x)\in X[\Gamma,\mathcal{N}]$. By Proposition~\ref{prop:geometric_main} we can define a tiling $\{t_i\}_{i \in I} = \Psi^{-1}(\varphi)\in\Omega_0(\PP)$. By definition of $\Psi$, for each $i \in I$ there is $w \in \supp(\varphi)$ such that $\pos(t_i)=\val(w)$. Once again as $\Psi$ is a homeomorphism, it follows that if $\val(w)=\val(w')$, then $w$ is $\Gamma$-equivalent to $w'$. We can then unambiguously define $a_i = x(w)$, and consider the colored tiling $T=\{(t_i,a_i)\}_{i \in I}\in \Omega(\PP,A)$.

Suppose there is $p \in \mathcal{N}$ and a consistent $(F_p,\xi_p,c)\in \mathcal{C}_p$ and $v \in \RR^d$ such that $\mathtt{E}(F_p,\xi_p,c)+v \subset T$. If we write as before $p(\varepsilon)=(m,a)$ and $p(s)= (m',b)$, we have $\pi(\mathtt{E}(F_p,\xi_p,c)) = m\cup (m'+\val(s))$. By the definition of $\Psi$, we get that $v=\val(w)$ for some $w \in \supp(\varphi)$ and thus we deduce that $\varphi(w)=m$ and $\varphi(ws)=m'$. Finally, as $c(z_0)=a$ and $c(z_s)=b$, we deduce from the definition of $x$ that $x(w)=a$ and $x(ws)=b$, hence \[(\varphi(w),x(w))=p(\varepsilon) \mbox{ and }(\varphi(ws),x(ws))=p(s).\]
It follows that condition (s3) of~\Cref{def:phi_subshift} is not satisfied, and thus $(\varphi,x)\notin X[\Gamma,\mathcal{N}]$, which is a contradiction. Thus we deduce that  $T=\{(t_i,a_i)\}_{i \in I}\in \Omega(\PP,A,\FF(\mathcal{C}))$.

Now let $\widehat{T}\in \Omega(\PP)$. Shifting, we may suppose $\widehat{T}\in \Omega_{\circ}(\PP)$ as clearly for any $v\in \RR^d$ we have \[\Omega(\PP,A,\FF(\mathcal{C}), \widehat{T})\neq \varnothing \mbox{ if and only if } \Omega(\PP,A,\FF(\mathcal{C}), \widehat{T}+v)\neq \varnothing.\]  The previous arguments show that if we let $\varphi = \Psi(\widehat{T})$, then $\Omega(\PP,A,\FF(\mathcal{C}), \widehat{T})\neq \varnothing$ if and only if $X[\Gamma,\varphi,\FF]\neq \varnothing$. This proves both many-one reductions\end{proof}

We remark that for $d=2$~\Cref{thm:domino_geom} on its fixed tiling variant recovers the main result of Hellouin de Menibus, Lutfalla and Vanier~\cite[Theorem 2]{de2024decision}. The higher-dimensional case and the version for space of tilings are new.

\subsection{Euclidean tilings up to isometries}

So far, we have considered tilings of $\RR^d$ generated by a finite set of punctured tiles $\mathcal{P}$ by translations. One could also consider the set of tilings generated by isometries of $\RR^d$, in other words, in addition to translation we now also allow rotations and reflections of the punctured tiles. Formally, we consider the set of tilings $\Omega_{\operatorname{Isom}}(\PP)$ which consists on all tilings $\{t_i\}_{i \in I}$ of $\RR^d$ such that for every $i \in I$ we have $t_i = g_ip_i$ for some isometry $g_i \in \operatorname{Isom(\RR^d)}$ and $p_i \in \mathcal{P}$.

In this setting, one has to be more careful when defining finite local complexity. The appropriate notion is no longer to have finitely many patches that intersect a given ball up to translation, but up to isometries (this is usually called ``up to congruence'' in the literature). There are tiling spaces where these two notions differ, for instance the Pinwheel tiling~\cite{RadinPinwheel_1994}.

We argue that we can modify the patch blueprint so that it codifies these tiling spaces instead. Recall that every isometry of $\RR^d$ can be written in the form $x\mapsto Ax+v$ where $v \in \RR^d$ is a translation and $A \in O(d)$ is an orthogonal matrix. As in the previous case, we let $\rho$ be the radius of the smallest closed ball that contains every punctured prototile and fix positive integers $K$ and $L$. Under the assumption of FLC, there are finitely many isometry classes of patches of the form $T\sqcap B_{K\rho}$ and we take $\widehat{M}$ as a set of  representatives for these classes.

Next we set $\widehat{S}$ as the set of pairs $(m,v,A)$ where $m\in \widehat{M}$, $0<|v|\leq 3\rho$, $A\in O(d)$ and such that there exists some $m'\in \widehat{M}$ such that $Am'+v$ matches with $m$ in the intersection of their supports. We set $\init(s)=m$ and $\ter(s)$ as the set of such $m'$. Again, by FLC this set is finite.

Naturally, we now define $\val(s)\in \operatorname{Isom}(\RR)$ as the isometry given by $x\mapsto Ax+v$, and extend it to $\widehat{S}^*$ in the natural way, that is, $\val(\varepsilon)$ is the trivial isometry, and $\val(s_1\dots s_n) = \val(s_n)\circ \dots \circ \val(s_1)$. 

Next, let $\widehat{R}$ exactly as in the patch blueprint, that is, the set of pairs of words $(w,w')\in \widehat{S}^*\times \widehat{S}^*$ such that $|w|+|w'|\leq L$, $\init(w)=\init(w')$ and $\val(w)=\val(w')$. Note that here we are also imposing that resulting isometry in a cycle amounts to zero, thus we also include reflections and rotations into the equation.

Finally, define the \define{isometric patch blueprint} as $\widehat{\Gamma}(\mathcal{P},K,L) = (\widehat{M},\widehat{S},\init,\ter,\widehat{R})$.

It can be verified that, in the same way that the patch blueprint captures the space of tilings $\Omega(\PP)$ up to translation (codified by the fact that the tiles are punctured), the isometric patch blueprint captures the space $\Omega_{\operatorname{Isom}}(\PP)$ up to isometries. This follows the same lines as the proof of~\Cref{prop:geometric_main} and the undecidability of the analogous domino problem can also be established for these tiling spaces.

\subsection{Geometric tilings of the hyperbolic space} Consider the hyperbolic space $(\mathbb{H}^n,d)$, where $\mathbb{H}^n = \{ x = (x_1,\dots,x_n) \in \RR^n : x_n >0 \}$ is the upper space model and $d$ is the hyperbolic metric on $\mathbb{H}^n$. We can define tiles, tilings and patches in the same way as in the Euclidean space. 

Given a finite set of punctured tiles $\PP$, we can consider the space of tilings $\Omega_{\mathbb{H}^n}(\PP)$ where every tile is the image of some $p\in \PP$ up to an isometry of $\mathbb{H}^n$ and define finite local complexity as in the previous subsection. We can also define an isometric patch blueprint $\Gamma = \Gamma_{\mathbb{H}^n}(\PP,K,L)$ as before by replacing isometries of $\RR^d$ by isometries of $\mathbb{H}^n$. Finally, we can define a natural continuous injection \[ \Psi\colon \Omega_{\mathbb{H}^n}(\PP)/\operatorname{Isom}(\mathbb{H}^n) \to \mathcal{M}(\Gamma). \]

However, we do not know if an analogue of~\Cref{prop:geometric_main} holds, that is, we do now know if there are parameters $K$ and $L$ for which this map is surjective and thus we don't know if this blueprint does capture the space of hyperbolic tilings of $\mathbb{H}^n$ up to isometries.

\begin{question}
    Suppose $\PP$ is a set of punctured tiles in $\mathbb{H}^n$ with finite local complexity. Is the natural map $\Psi\colon \Omega_{\mathbb{H}^n}(\PP)/\operatorname{Isom}(\mathbb{H}^n) \to \mathcal{M}(\Gamma_{\mathbb{H}^n}(\PP,K,L))$ a homeomorphism for some values of $K$ and $L$?
\end{question}

We note that we cannot just copy the proof of~\Cref{prop:geometric_main}, as it uses the Pythagorean theorem which is no longer valid in hyperbolic space. We also note that it is easy to show that the centers of tiles in any such hyperbolic tiling form a discrete set which is quasi-isometric to $\mathbb{H}^n$, and also it is known~\cite{bartholdi2023domino} that the domino problem for any group quasi-isometric to $\mathbb{H}^n$ is undecidable for $n \geq 2$, thus if the above question has positive answer, we would be able to prove the undecidability of the domino problem for geometric hyperbolic tilings as well.



\section*{Acknowledgments}
The authors would like to thank the anonymous reviewer and Benjamin Hellouin de Menibus  for their useful comments and suggestions. S. Barbieri was supported by the ANID grant FONDECYT regular 1240085, and by the projects AMSUD240026 and ECOS230003. N. Bitar was supported by the ANR project IZES ANR-22-CE40-0011.

\printbibliography
\Addresses

\end{document}

%% file: fig/123.tex
\begin{tikzpicture}[scale=1]
				\draw [->, blue!50!black, thick, shorten >= 0.2cm] (0,0) -- (-2,-1) node [pos=0.5, below] {$\ell$};
				\draw [->, blue!50!black, thick, shorten >= 0.2cm] (0,0) -- (2,-1) node [pos=0.5, below] {$r$};
                
                \draw [->, blue!50!black, thick, shorten >= 0.2cm] (-2,-1) -- (-2,-2) node [pos=0.5, right] {$s$};
				\draw [->, blue!50!black, thick, shorten >= 0.2cm] (2,-1) -- (0.5,-2) node [pos=0.5, below] {$\ell$};
                \draw [->, blue!50!black, thick, shorten >= 0.2cm] (2,-1) -- (3.5,-2) node [pos=0.5, below] {$r$};

                \draw [->, blue!50!black, thick, shorten >= 0.2cm] (-2,-2) -- (-3.5,-3) node [pos=0.5, below] {$\ell$};
                \draw [->, blue!50!black, thick, shorten >= 0.2cm] (-2,-2) -- (-0.5,-3) node [pos=0.5, below] {$r$};
				\draw [->, blue!50!black, thick, shorten >= 0.2cm] (0.5,-2) -- (0.5,-3) node [pos=0.5, right] {$s$};
                \draw [->, blue!50!black, thick, shorten >= 0.2cm] (3.5,-2) -- (3.5,-3) node [pos=0.5, right] {$s$};

                \draw [->, blue!50!black, thick, shorten >= 0.2cm] (-3.5,-3) -- (-4.5,-4) node [pos=0.5, below] {$\ell$};
                \draw [->, blue!50!black, thick, shorten >= 0.2cm] (-3.5,-3) -- (-2.5,-4) node [pos=0.5, below] {$r$};
                \draw [->, blue!50!black, thick, shorten >= 0.2cm] (-0.5,-3) -- (-0.5,-4) node [pos=0.5, right] {$s$};
				\draw [->, blue!50!black, thick, shorten >= 0.2cm] (0.5,-3) -- (0.5,-4) node [pos=0.5, right] {$s$};
                \draw [->, blue!50!black, thick, shorten >= 0.2cm] (3.5,-3) -- (2.5,-4) node [pos=0.5, below] {$\ell$};
                \draw [->, blue!50!black, thick, shorten >= 0.2cm] (3.5,-3) -- (4.5,-4) node [pos=0.5, below] {$r$};
			
				\draw[fill = white] (0,0) circle (0.2) node {$2$};
                \draw[fill = white] (-2,-1) circle (0.2) node {$1$};
                \draw[fill = white] (2,-1) circle (0.2)node {$2$};

                \draw[fill = white] (-2,-2) circle (0.2) node {$2$};
                \draw[fill = white] (0.5,-2) circle (0.2) node {$1$};
                \draw[fill = white] (3.5,-2) circle (0.2)node {$1$};

                \draw[fill = white] (-3.5,-3) circle (0.2) node {$2$};
                \draw[fill = white] (-0.5,-3) circle (0.2) node {$1$};
                \draw[fill = white] (0.5,-3) circle (0.2) node {$1$};
                \draw[fill = white] (3.5,-3) circle (0.2) node {$2$};

		\end{tikzpicture}

%% file: fig/hyperbolic.tex
\begin{tikzpicture}[scale=1]
		
        \begin{scope}[shift={(0,0.05)}]
        \draw [->, blue!50!black, thick, shorten >= 0.2cm] (0,0) -- (-4,-1) node [pos=0.6, above] {\scalebox{0.6}{$s^0_0$}};
        \draw [->, blue!50!black, thick, shorten >= 0.2cm] (0,0) -- (4,-1) node [pos=0.6, above] {\scalebox{0.6}{$s^1_0$}};
        \end{scope}
        \begin{scope}[shift={(0,-0.05)}]
            \draw [<-, blue!50!red, thick, shorten <= 0.2cm] (0,0) -- (-4,-1) node [pos=0.4, below] {\scalebox{0.6}{$p_0$}};
            \draw [<-, blue!50!red, thick, shorten <= 0.2cm] (0,0) -- (4,-1) node [pos=0.4, below] {\scalebox{0.6}{$p_1$}};
        \end{scope}
        
        \draw[fill = white] (0,0) circle (0.15) node {\scalebox{0.6}{$0$}};
        
        \begin{scope}[shift={(0,0.05)}]
        \draw [->, blue!50!black, thick, shorten >= 0.2cm] (-4,-1) -- (-6,-2) node [pos=0.6, above] {\scalebox{0.6}{$s^0_0$}};
        \draw [->, blue!50!black, thick, shorten >= 0.2cm] (-4,-1) -- (-2,-2) node [pos=0.6, above] {\scalebox{0.6}{$s^1_0$}};
        \draw [->, blue!50!black, thick, shorten >= 0.2cm] (4,-1) -- (2,-2) node [pos=0.6, above] {\scalebox{0.6}{$s^0_1$}};
        \draw [->, blue!50!black, thick, shorten >= 0.2cm] (4,-1) -- (6,-2) node [pos=0.6, above] {\scalebox{0.6}{$s^1_1$}};
        \draw [->, green!50!black, thick, shorten <= 0.6cm, shorten >= 0.6cm] (-4,-1) -- (4,-1) node [pos=0.5, above] {\scalebox{0.6}{$t_0^+$}};
        \end{scope}
        \begin{scope}[shift={(0,-0.05)}]
            \draw [<-, red!50!black, thick, shorten <= 0.2cm] (-4,-1) -- (-6,-2) node [pos=0.4, below] {\scalebox{0.6}{$p_0$}};
            \draw [<-, red!50!black, thick, shorten <= 0.2cm] (-4,-1) -- (-2,-2) node [pos=0.4, below] {\scalebox{0.6}{$p_1$}};
            \draw [<-, red!50!black, thick, shorten <= 0.2cm] (4,-1) -- (2,-2) node [pos=0.4, below] {\scalebox{0.6}{$p_0$}};
            \draw [<-, red!50!black, thick, shorten <= 0.2cm] (4,-1) -- (6,-2) node [pos=0.4, below] {\scalebox{0.6}{$p_1$}};
            \draw [<-, brown!50!black, thick, shorten <= 0.6cm, shorten >= 0.6cm] (-4,-1) -- (4,-1) node [pos=0.5, below] {\scalebox{0.6}{$t_1^-$}};
        \end{scope}
        \draw[fill = white] (-4,-1) circle (0.15) node {\scalebox{0.6}{$0$}};
        \draw[fill = white] (4,-1) circle (0.15)node {\scalebox{0.6}{$1$}};

        \begin{scope}[shift={(0,0.05)}]
        \draw [->, blue!50!black, thick, shorten >= 0.2cm] (-6,-2) -- (-7,-3) node [pos=0.6, above] {\scalebox{0.6}{$s^0_0$}};
        \draw [->, blue!50!black, thick, shorten >= 0.2cm] (-6,-2) -- (-5,-3) node [pos=0.6, above] {\scalebox{0.6}{$s^1_0$}};
        \draw [->, blue!50!black, thick, shorten >= 0.2cm] (-2,-2) -- (-3,-3) node [pos=0.6, above] {\scalebox{0.6}{$s^0_1$}};
        \draw [->, blue!50!black, thick, shorten >= 0.2cm] (-2,-2) -- (-1,-3) node [pos=0.6, above] {\scalebox{0.6}{$s^1_1$}};
        \draw [->, blue!50!black, thick, shorten >= 0.2cm] (2,-2) -- (1,-3) node [pos=0.6, above] {\scalebox{0.6}{$s^0_0$}};
        \draw [->, blue!50!black, thick, shorten >= 0.2cm] (2,-2) -- (3,-3) node [pos=0.6, above] {\scalebox{0.6}{$s^1_0$}};
        \draw [->, blue!50!black, thick, shorten >= 0.2cm] (6,-2) -- (5,-3) node [pos=0.6, above] {\scalebox{0.6}{$s^0_1$}};
        \draw [->, blue!50!black, thick, shorten >= 0.2cm] (6,-2) -- (7,-3) node [pos=0.6, above] {\scalebox{0.6}{$s^1_1$}};
        \draw [->, green!50!black, thick, shorten <= 0.4cm, shorten >= 0.4cm] (-6,-2) -- (-2,-2) node [pos=0.5, above] {\scalebox{0.6}{$t_0^+$}};
        \draw [->, green!50!black, thick, shorten <= 0.4cm, shorten >= 0.4cm] (-2,-2) -- (2,-2) node [pos=0.5, above] {\scalebox{0.6}{$t_1^+$}};
        \draw [->, green!50!black, thick, shorten <= 0.4cm, shorten >= 0.4cm] (2,-2) -- (6,-2) node [pos=0.5, above] {\scalebox{0.6}{$t_0^+$}};
        \end{scope}

        \begin{scope}[shift={(0,-0.05)}]
        \draw [<-, red!50!black, thick, shorten <= 0.2cm] (-6,-2) -- (-7,-3) node [pos=0.4, below] {\scalebox{0.6}{$p_0$}};
        \draw [<-, red!50!black, thick, shorten <= 0.2cm] (-6,-2) -- (-5,-3) node [pos=0.4, below] {\scalebox{0.6}{$p_1$}};
        \draw [<-, red!50!black, thick, shorten <= 0.2cm] (-2,-2) -- (-3,-3) node [pos=0.4, below] {\scalebox{0.6}{$p_0$}};
        \draw [<-, red!50!black, thick, shorten <= 0.2cm] (-2,-2) -- (-1,-3) node [pos=0.4, below] {\scalebox{0.6}{$p_1$}};
        \draw [<-, red!50!black, thick, shorten <= 0.2cm] (2,-2) -- (1,-3) node [pos=0.4, below] {\scalebox{0.6}{$p_0$}};
        \draw [<-, red!50!black, thick, shorten <= 0.2cm] (2,-2) -- (3,-3) node [pos=0.4, below] {\scalebox{0.6}{$p_1$}};
        \draw [<-, red!50!black, thick, shorten <= 0.2cm] (6,-2) -- (5,-3) node [pos=0.4, below] {\scalebox{0.6}{$p_0$}};
        \draw [<-, red!50!black, thick, shorten <= 0.2cm] (6,-2) -- (7,-3) node [pos=0.4, below] {\scalebox{0.6}{$p_1$}};
        \draw [<-, brown!50!black, thick, shorten <= 0.4cm, shorten >= 0.4cm] (-6,-2) -- (-2,-2) node [pos=0.5, below] {\scalebox{0.6}{$t_1^-$}};
        \draw [<-, brown!50!black, thick, shorten <= 0.4cm, shorten >= 0.4cm] (-2,-2) -- (2,-2) node [pos=0.5, below] {\scalebox{0.6}{$t_0^-$}};
        \draw [<-, brown!50!black, thick, shorten <= 0.4cm, shorten >= 0.4cm] (2,-2) -- (6,-2) node [pos=0.5, below] {\scalebox{0.6}{$t_1^-$}};
        \end{scope}

        \draw[fill = white] (-6,-2) circle (0.15) node {\scalebox{0.6}{$0$}};
        \draw[fill = white] (-2,-2) circle (0.15) node {\scalebox{0.6}{$1$}};
        \draw[fill = white] (2,-2) circle (0.15)node {\scalebox{0.6}{$0$}};
        \draw[fill = white] (6,-2) circle (0.15)node {\scalebox{0.6}{$1$}};

        \begin{scope}[shift={(0,0.05)}]
        \draw [->, green!50!black, thick, shorten <= 0.2cm, shorten >= 0.2cm] (-7,-3) -- (-5,-3) node [pos=0.5, above] {\scalebox{0.6}{$t_0^+$}};
        \draw [->, green!50!black, thick, shorten <= 0.2cm, shorten >= 0.2cm] (-5,-3) -- (-3,-3) node [pos=0.5, above] {\scalebox{0.6}{$t_1^+$}};
        \draw [->, green!50!black, thick, shorten <= 0.2cm, shorten >= 0.2cm] (-3,-3) -- (-1,-3) node [pos=0.5, above] {\scalebox{0.6}{$t_0^+$}};
        \draw [->, green!50!black, thick, shorten <= 0.2cm, shorten >= 0.2cm] (-1,-3) -- (1,-3) node [pos=0.5, above] {\scalebox{0.6}{$t_1^+$}};
        \draw [->, green!50!black, thick, shorten <= 0.2cm, shorten >= 0.2cm] (1,-3) -- (3,-3) node [pos=0.5, above] {\scalebox{0.6}{$t_0^+$}};
        \draw [->, green!50!black, thick, shorten <= 0.2cm, shorten >= 0.2cm] (3,-3) -- (5,-3) node [pos=0.5, above] {\scalebox{0.6}{$t_1^+$}};
        \draw [->, green!50!black, thick, shorten <= 0.2cm, shorten >= 0.2cm] (5,-3) -- (7,-3) node [pos=0.5, above] {\scalebox{0.6}{$t_0^+$}};
        \end{scope}
        
        \begin{scope}[shift={(0,-0.05)}]
        \draw [<-, brown!50!black, thick, shorten <= 0.2cm, shorten >= 0.2cm] (-7,-3) -- (-5,-3) node [pos=0.5, below] {\scalebox{0.6}{$t_1^-$}};
        \draw [<-, brown!50!black, thick, shorten <= 0.2cm, shorten >= 0.2cm] (-5,-3) -- (-3,-3) node [pos=0.5, below] {\scalebox{0.6}{$t_0^-$}};
        \draw [<-, brown!50!black, thick, shorten <= 0.2cm, shorten >= 0.2cm] (-3,-3) -- (-1,-3) node [pos=0.5, below] {\scalebox{0.6}{$t_1^-$}};
        \draw [<-, brown!50!black, thick, shorten <= 0.2cm, shorten >= 0.2cm] (-1,-3) -- (1,-3) node [pos=0.5, below] {\scalebox{0.6}{$t_0^-$}};
        \draw [<-, brown!50!black, thick, shorten <= 0.2cm, shorten >= 0.2cm] (1,-3) -- (3,-3) node [pos=0.5, below] {\scalebox{0.6}{$t_1^-$}};
        \draw [<-, brown!50!black, thick, shorten <= 0.2cm, shorten >= 0.2cm] (3,-3) -- (5,-3) node [pos=0.5, below] {\scalebox{0.6}{$t_0^-$}};
        \draw [<-, brown!50!black, thick, shorten <= 0.2cm, shorten >= 0.2cm] (5,-3) -- (7,-3) node [pos=0.5, below] {\scalebox{0.6}{$t_1^-$}};

        \end{scope}

        \draw[fill = white] (-7,-3) circle (0.15) node {\scalebox{0.6}{$0$}};
        \draw[fill = white] (-5,-3) circle (0.15) node {\scalebox{0.6}{$1$}};
        \draw[fill = white] (-3,-3) circle (0.15) node {\scalebox{0.6}{$0$}};
        \draw[fill = white] (-1,-3) circle (0.15) node {\scalebox{0.6}{$1$}};
        \draw[fill = white] (1,-3) circle (0.15) node {\scalebox{0.6}{$0$}};
        \draw[fill = white] (3,-3) circle (0.15) node {\scalebox{0.6}{$1$}};
        \draw[fill = white] (5,-3) circle (0.15) node {\scalebox{0.6}{$0$}};
        \draw[fill = white] (7,-3) circle (0.15) node {\scalebox{0.6}{$1$}};

		\end{tikzpicture}

%% file: fig/hard_square.tex
\begin{tikzpicture}[scale=1]
		
        \begin{scope}[shift={(0,0.05)}]
        \draw [->, shorten >= 0.2cm] (0,0) -- (-4,-1);
        \draw [->, shorten >= 0.2cm] (0,0) -- (4,-1);
        \end{scope}
        \begin{scope}[shift={(0,-0.05)}]
            \draw [<-, shorten <= 0.2cm] (0,0) -- (-4,-1);
            \draw [<-, shorten <= 0.2cm] (0,0) -- (4,-1);
        \end{scope}
        
        \draw[fill = azul] (0,0) circle  (0.2);
        
        \begin{scope}[shift={(0,0.05)}]
        \draw [->,shorten >= 0.2cm] (-4,-1) -- (-6,-2);
        \draw [->,shorten >= 0.2cm] (-4,-1) -- (-2,-2);
        \draw [->,shorten >= 0.2cm] (4,-1) -- (2,-2);
        \draw [->,shorten >= 0.2cm] (4,-1) -- (6,-2);
        \draw [->,shorten <= 0.6cm, shorten >= 0.6cm] (-4,-1) -- (4,-1);
        \end{scope}
        \begin{scope}[shift={(0,-0.05)}]
            \draw [<-,shorten <= 0.2cm] (-4,-1) -- (-6,-2);
            \draw [<-,shorten <= 0.2cm] (-4,-1) -- (-2,-2);
            \draw [<-,shorten <= 0.2cm] (4,-1) -- (2,-2);
            \draw [<-,shorten <= 0.2cm] (4,-1) -- (6,-2);
            \draw [<-,shorten <= 0.6cm, shorten >= 0.6cm] (-4,-1) -- (4,-1);
        \end{scope}
        \draw[fill = rojo] (-4,-1) circle  (0.2) ;
        \draw[fill = azul] (4,-1) circle  (0.2);

        \begin{scope}[shift={(0,0.05)}]
        \draw [->,shorten >= 0.2cm] (-6,-2) -- (-7,-3);
        \draw [->,shorten >= 0.2cm] (-6,-2) -- (-5,-3);
        \draw [->,shorten >= 0.2cm] (-2,-2) -- (-3,-3);
        \draw [->,shorten >= 0.2cm] (-2,-2) -- (-1,-3);
        \draw [->,shorten >= 0.2cm] (2,-2) -- (1,-3);
        \draw [->,shorten >= 0.2cm] (2,-2) -- (3,-3);
        \draw [->,shorten >= 0.2cm] (6,-2) -- (5,-3);
        \draw [->,shorten >= 0.2cm] (6,-2) -- (7,-3);
        \draw [->,shorten <= 0.4cm, shorten >= 0.4cm] (-6,-2) -- (-2,-2);
        \draw [->,shorten <= 0.4cm, shorten >= 0.4cm] (-2,-2) -- (2,-2);
        \draw [->,shorten <= 0.4cm, shorten >= 0.4cm] (2,-2) -- (6,-2);
        \end{scope}

        \begin{scope}[shift={(0,-0.05)}]
        \draw [<-,shorten <= 0.2cm] (-6,-2) -- (-7,-3);
        \draw [<-,shorten <= 0.2cm] (-6,-2) -- (-5,-3);
        \draw [<-,shorten <= 0.2cm] (-2,-2) -- (-3,-3);
        \draw [<-,shorten <= 0.2cm] (-2,-2) -- (-1,-3);
        \draw [<-,shorten <= 0.2cm] (2,-2) -- (1,-3);
        \draw [<-,shorten <= 0.2cm] (2,-2) -- (3,-3);
        \draw [<-,shorten <= 0.2cm] (6,-2) -- (5,-3);
        \draw [<-,shorten <= 0.2cm] (6,-2) -- (7,-3);
        \draw [<-,shorten <= 0.4cm, shorten >= 0.4cm] (-6,-2) -- (-2,-2);
        \draw [<-,shorten <= 0.4cm, shorten >= 0.4cm] (-2,-2) -- (2,-2);
        \draw [<-,shorten <= 0.4cm, shorten >= 0.4cm] (2,-2) -- (6,-2);
        \end{scope}

        \draw[fill = azul] (-6,-2) circle  (0.2) ;
        \draw[fill = azul] (-2,-2) circle  (0.2) ;
        \draw[fill = azul] (2,-2) circle  (0.2);
        \draw[fill = azul] (6,-2) circle  (0.2);

        \begin{scope}[shift={(0,0.05)}]
        \draw [->,shorten <= 0.2cm, shorten >= 0.2cm] (-7,-3) -- (-5,-3);
        \draw [->,shorten <= 0.2cm, shorten >= 0.2cm] (-5,-3) -- (-3,-3);
        \draw [->,shorten <= 0.2cm, shorten >= 0.2cm] (-3,-3) -- (-1,-3);
        \draw [->,shorten <= 0.2cm, shorten >= 0.2cm] (-1,-3) -- (1,-3);
        \draw [->,shorten <= 0.2cm, shorten >= 0.2cm] (1,-3) -- (3,-3);
        \draw [->,shorten <= 0.2cm, shorten >= 0.2cm] (3,-3) -- (5,-3);
        \draw [->,shorten <= 0.2cm, shorten >= 0.2cm] (5,-3) -- (7,-3);
        \end{scope}
        
        \begin{scope}[shift={(0,-0.05)}]
        \draw [<-,shorten <= 0.2cm, shorten >= 0.2cm] (-7,-3) -- (-5,-3);
        \draw [<-,shorten <= 0.2cm, shorten >= 0.2cm] (-5,-3) -- (-3,-3);
        \draw [<-,shorten <= 0.2cm, shorten >= 0.2cm] (-3,-3) -- (-1,-3);
        \draw [<-,shorten <= 0.2cm, shorten >= 0.2cm] (-1,-3) -- (1,-3);
        \draw [<-,shorten <= 0.2cm, shorten >= 0.2cm] (1,-3) -- (3,-3);
        \draw [<-,shorten <= 0.2cm, shorten >= 0.2cm] (3,-3) -- (5,-3);
        \draw [<-,shorten <= 0.2cm, shorten >= 0.2cm] (5,-3) -- (7,-3);

        \end{scope}

        \draw[fill = rojo] (-7,-3) circle  (0.2);
        \draw[fill = azul] (-5,-3) circle  (0.2);
        \draw[fill = rojo] (-3,-3) circle  (0.2);
        \draw[fill = azul] (-1,-3) circle  (0.2);
        \draw[fill = rojo] (1,-3) circle  (0.2);
        \draw[fill = azul] (3,-3) circle  (0.2);
        \draw[fill = rojo] (5,-3) circle  (0.2);
        \draw[fill = azul] (7,-3) circle  (0.2);

		\end{tikzpicture}